\documentclass[a4paper,twoside,11pt,reqno]{amsart}

\usepackage[margin=1.1in]{geometry}
\usepackage{amsmath} 
\usepackage{amssymb} 
\usepackage{amsopn}
\usepackage{amsthm} 
\usepackage{amsfonts} 
\usepackage{mathrsfs}
\usepackage{stmaryrd}
\usepackage{accents} 
\usepackage{colonequals}
\usepackage{tikz-cd}
\usepackage{enumitem}

\usetikzlibrary{fit,matrix}
\usepackage{cite}

\usepackage{verbatim}

\usepackage[OT2,T1]{fontenc}

\usepackage[pdfpagelabels,pdftex]{hyperref}
\usepackage{setspace}
\usepackage{amssymb}
\usepackage{euscript}
\usepackage{cite}
\usepackage[all]{xy}
\usepackage{tabularx}

\theoremstyle{plain}
\usepackage{amsfonts}
\usepackage{graphicx}
\usepackage{amsmath}
\hypersetup{
  colorlinks   = true, %Colours links instead of ugly boxes
  urlcolor     = blue, %Colour for external hyperlinks
  linkcolor    = black, %Colour of internal links
  citecolor   = black, %Colour of citations
pdftitle={},
  pdfauthor={},
  pdfsubject={},
  pdfkeywords={},
  breaklinks=true,
  bookmarksopen=true,
  bookmarksnumbered=true,
  pdfpagemode=UseOutlines,
  plainpages=false}

\newcommand\bfG{\mathbf{G}}

\DeclareMathOperator{\GL}{GL}

\DeclareMathOperator\Gal{Gal}

\DeclareMathOperator\diag{diag}
\DeclareMathOperator\Nm{Nm}

\DeclareMathOperator\Fr{Fr}

\DeclareMathOperator\Ind{Ind}
\DeclareMathOperator\cind{cInd}
\DeclareMathOperator\N{N}
\DeclareMathOperator\im{im}
\DeclareMathOperator\tr{tr}
\DeclareMathOperator\val{val}
\DeclareMathOperator\bfGL{\bf GL}
\DeclareMathOperator\bfU{\bf U}
\DeclareMathOperator\bfT{\bf T}
\DeclareMathOperator\bfX{\bf X}
\DeclareMathOperator{\cInd}{c-Ind}

\newcommand{\sm}{{\,\smallsetminus\,}}

\newcommand\from{\colon}
\newcommand\bG{\mathbb G}
\newcommand\bU{\mathbb U}
\newcommand\QQ{\mathbb Q}

\newcommand\bT{\mathbb T}

\newcommand\bW{\mathbb W}
\newcommand\FF{\mathbb F}
\newcommand\cA{\mathcal A}

\newcommand\bZ{\mathbb Z}

\newcommand\bF{\mathbb{F}}
\newcommand\bQ{\mathbb{Q}}
\newcommand\bK{\mathbb{K}}
\newcommand\bL{\mathbb{L}}
\newcommand\cE{\mathscr{E}}
\newcommand\fp{\mathfrak{p}}

\newcommand\fa{\mathfrak{a}}

\newenvironment{psmallmatrix}
  {\left(\begin{smallmatrix}}
  {\end{smallmatrix}\right)}

\DeclareMathOperator\Spec{Spec}
\DeclareMathOperator\Perf{Perf}

\newcommand\prolim{\mathop{\underleftarrow{\lim} }}

\newcommand\cO{\mathcal O}

\newcommand\bA{\mathbb A}

\makeatletter
\newtheorem*{rep@theorem}{\rep@title}
\newcommand{\newreptheorem}[2]{%
\newenvironment{rep#1}[1]{%
 \def\rep@title{#2 \ref{##1}}%
 \begin{rep@theorem}}%
 {\end{rep@theorem}}}
\makeatother

\newtheorem{thm}{Theorem}[section]
\newtheorem*{thm*}{Theorem}

\newtheorem{prop}[thm]{Proposition}

\newtheorem{cor}[thm]{Corollary}
\newtheorem{Cor}[thm]{Corollary}
\newtheorem*{cor*}{Corollary}
\newtheorem{lm}[thm]{Lemma}

\newtheorem{conj}[thm]{Conjecture}

\newtheorem*{thma}{Theorem A}
\newtheorem*{thmb}{Theorem B}

\theoremstyle{remark}

\newreptheorem{thm}{Theorem}
\newreptheorem{prop}{Proposition}
\newreptheorem{cor}{Corollary}

\theoremstyle{definition}
\newtheorem{Def}[thm]{Definition}

\theoremstyle{remark}
\newtheorem{rem}[thm]{Remark}

\newenvironment{pro*}[1][Proof]{{\it{#1:}} }{}

\newcounter{absatzcounter}[section]
\numberwithin{equation}{section}

\begin{document}

\title[On loop Deligne--Lusztig varieties]{On loop Deligne--Lusztig varieties of Coxeter-type for inner forms of $\GL_n$}
\author{Charlotte Chan and Alexander B. Ivanov}
\address{Department of Mathematics \\
MIT \\
Cambridge, MA 02139, USA}
\email{charchan@umich.edu}
\address{Mathematisches Institut \\ Universit\"at Bonn \\ Endenicher Allee 60 \\ 53115 Bonn, Germany}
\email{ivanov@math.uni-bonn.de}
\maketitle

\begin{abstract}
For a reductive group $G$ over a local non-archimedean field $K$ one can mimic the construction from classical Deligne--Lusztig theory by using the loop space functor. We study this construction in the special case that $G$ is an inner form of ${\rm GL}_n$ and the loop Deligne--Lusztig variety is of Coxeter type. After simplifying the proof of its representability, our main result is that its $\ell$-adic cohomology realizes many irreducible supercuspidal representations of $G$, notably almost all among those whose L-parameter factors through an unramified elliptic maximal torus of $G$. This gives a purely local, purely geometric and -- in a sense -- quite explicit way to realize special cases of the local Langlands and Jacquet--Langlands correspondences.
\end{abstract}

% \tableofcontents

\section{Introduction}

Let $\bfG$ be an inner form of ${\bf GL}_n$ ($n \geq 2$) over a local non-archimedean field $K$ and let $G = \bfG(K)$ be the group of its $K$-points. Let $\bfT \subseteq \bfG$ be a maximal elliptic unramified torus. Then $\bfT$ is uniquely determined up to $G$-conjugation and $T = \bfT(K) \cong L^\times$ where $L/K$ is the unramified extension of degree $n$. Following an idea of Lusztig \cite{Lusztig_79}, in \cite{CI_ADLV} we constructed a scheme $X$ over $\overline{\FF}_q$ with an action by $G \times T$, which can be seen as an analogue of a Deligne--Lusztig variety for the $K$-groups $\bfT \subseteq \bfG$. As in  classical Deligne--Lusztig theory \cite{DeligneL_76}, this allowed us to attach to a smooth character $\theta \colon T \rightarrow \overline{\bQ}_\ell^\times$ the corresponding isotypic component $R_T^G(\theta)$ of the $\ell$-adic Euler characteristic of $X$, a smooth virtual $G$-representation. In \cite{CI_ADLV}, we studied these objects when $\theta$ is primitive (i.e., the Howe decomposition of $\theta$ has exactly one member): we showed that $R_T^G(\theta)$ is irreducible supercuspidal and isomorphic to the representation attached to $(L/K,\theta)$ by Howe \cite{Howe_77}, and hence provides a geometric and purely local realization of the local Langlands and Jacquet--Langlands correspondences.

These results indicate that $X$, and more generally other schemes obtained by similar Deligne--Lusztig-type constructions for other reductive groups over $K$, allow a quite \emph{explicit}, \emph{purely local}, and \emph{purely geometric} way to realize the local Langlands correspondence and/or some instances of automorphic induction for at least those irreducible representations of $\bfG$, whose L-parameter factors through an unramified torus. This is highly desirable, as the existing local proofs of the local Langlands correspondence are purely algebraic (e.g.\ via Bushnell--Kutzko types), and the existing geometric proofs tend to be very inexplicit and/or use global arguments (except for \cite{BoyarchenkoW_16}, which -- similar to \cite{CI_ADLV} -- only deals with primitive $\theta$). Moreover, an exact analogue of the classical Deligne--Lusztig theory over non-archimedean local fields is highly interesting in its own right.

The first goal of the present article is to give a more satisfactory definition of $X$ and simplify the proof of its representability. The second goal is to show that $R_T^G(\theta)$ is irreducible supercuspidal and realizes the local Langlands and Jacquet--Langlands correspondences for a much wider class of irreducible  supercuspidal representations of $G$ (almost all among those, whose L-parameter factors through $T \subseteq G$), thus going far beyond the corresponding results of \cite{BoyarchenkoW_16} and \cite{CI_ADLV}. As the methods from \cite{CI_ADLV} for primitive $\theta$ do not apply anymore, our main concern here will be to develop new geometric methods to study the cohomology of Deligne--Lusztig constructions of Coxeter type over local fields, in particular generalizing results of \cite{Lusztig_04} away from the case when $\theta$ is regular (in the sense of \textit{op.\ cit.}), and providing nice description for the quotient of (subschemes of) $X$ by unipotent radicals of rational parabolic subgroups of $G$, which generalizes (in the special case for $\bfG,\bfT$) to the situation over $K$ particular results of \cite{Lusztig_76_Inv}. Some of these methods immediately work for all reductive groups, and some rely on $\bfG$ being an inner form of $\bfGL_n$.

To describe our result, we need more notation. First of all, there is an unique integer $\kappa \in \{0, 1, \dots, n-1\}$ such that if $n' = \gcd(n,\kappa)$, $n=n'n_0$, $\kappa = n'k_0$, we have $\bfG \cong \bfGL_{n'}(D_{k_0/n_0})$, where $D_{k_0/n_0}$ denotes the central division algebra over $K$ with Hasse-invariant $k_0/n_0$.

Let $\varepsilon$ be any character of $K^\times$ with $\ker(\varepsilon) = {\rm N}_{L/K}(L^\times)$, the image of the norm map of $L/K$. Denote by 
\begin{itemize}
\item $\mathscr{X}$ the set of smooth characters of $L^\times$ with trivial stabilizer in $\Gal(L/K)$,
\item $\mathscr{G}_K^\varepsilon(n)$ the set of isomorphism classes of smooth $n$-dimensional representations $\sigma$ of the Weil group $\mathcal{W}_K$ of $K$ satisfying $\sigma \cong \sigma \otimes (\varepsilon \circ {\rm rec}_K)$,
\item $\mathscr{A}_K^\varepsilon(n,\kappa)$ the set of smooth irreducible supercuspidal representations $\pi$ of $G$ (= ${\bf G}(K)$ with ${\bf G}$ corresponding to $\kappa$) such that $\pi \cong \pi \otimes (\varepsilon \circ {\rm Nrd}_G).$
\end{itemize}
There are natural \emph{bijections}
\begin{center}
\begin{tikzcd}
\mathscr{X}/\Gal(L/K) \arrow[r] &[20pt] \mathscr{G}_K^\varepsilon(n) \arrow[r, "\rm LL"] &[20pt] \mathscr{A}_K^\varepsilon(n,0) \arrow[r,"\rm JL"]&[20pt] \mathscr{A}_K^\varepsilon(n,\kappa) \\[-20pt]
 \theta \arrow[r,mapsto, start anchor={[xshift=6.5ex]}, end anchor={[xshift=-3ex]} ] &[20pt] \sigma_\theta \arrow[r,mapsto,start anchor={[xshift=3ex]}, end anchor={[xshift=-1ex]}] &[20pt] {\rm LL}(\sigma_\theta) =: \pi_\theta^{\GL_n} \arrow[r, mapsto, start anchor={[xshift=1ex]}, end anchor={[xshift=-2.3ex]}] &[20pt] {\rm JL}(\pi_\theta^{\GL_n}) =: \pi_\theta. 
\end{tikzcd}
\end{center}
The latter two maps are the local Langlands and the Jacquet--Langlands correspondences respectively. Here $\sigma_\theta := \Ind_{\mathcal{W}_L}^{\mathcal{W}_K} (\theta\cdot \mu)$ is the induction to $\mathcal{W}_K$ of the character $\mathcal{W}_L \rightarrow \mathcal{W}_L^{\rm ab} \stackrel{{\rm rec}_L}{\rightarrow} L^\times \stackrel{\theta \cdot \mu}{\rightarrow} \overline{\bQ}_\ell^\times$, where $\mu$ is the \emph{rectifier}, i.e.\ the unramified character of $L^\times$ defined by $\mu(\varpi) = (-1)^{n-1}$ (here $\varpi$ uniformizer of $L$). 

Our main result is the following theorem, which confirms Lusztig's conjecture \cite{Lusztig_79} in this setting for a large class of characters $\theta$.

\begin{thma}\hypertarget{thm:A}
Assume that $p>n$. Let $\theta \colon T\cong L^\times \rightarrow \overline{\QQ}_\ell^\times$ be a smooth character such that $\theta|_{U_L^1}$ has trivial stabilizer in $\Gal(L/K)$. Then $\pm R_T^G(\theta)$ is a genuine $G$-representation and
\[
\pm R_T^G(\theta) \cong \pi_\theta.
\]
In particular, $\pm R_T^G(\theta)$ is irreducible supercuspidal and $\sigma_\theta \leftrightarrow \pm R_T^G(\theta)$ is a realization of the local Langlands and Jacquet--Langlands correspondences.
\end{thma}

In \cite[Section 12]{CI_ADLV}, for all $p,n,\kappa$ we establish Theorem \hyperlink{thm:A}{A} under the assumptions that either: $\theta$ is primitive (this is a stronger condition on $\theta$ than in the theorem above), or $\theta$ is unramified and has trivial $\Gal(L/K)$-stabilizer. When $\kappa$ is coprime to $n$, then $\bfG$ is the unit group of a central division algebra, and the proof of Theorem \hyperlink{thm:A}{A} is considerably easier: for all $p,n$ and any $\theta$ with trivial $\Gal(L/K)$-stabilizer, the result holds by combining Lusztig \cite{Lusztig_79} and Boyarchenko \cite{Boyarchenko_12} along with a result of Henniart \cite[Th\'eor\`eme 3.1]{Henniart_92} (see \cite[Section 7]{Chan_siDL}). When $\kappa = 1$ and $n = 2$, Theorem \hyperlink{thm:A}{A} was first done in \cite{Ivanov_15_ADLV_GL2_unram}. For $\bfG = \bfGL_2$ and ramified elliptic tori a similar result was shown in \cite{Ivanov_15_ADLV_GL2_ram,Ivanov_18_wild}. Tackling the problem of understanding $R_T^G(\theta)$ for all $n,\kappa$ under the far weaker assumptions on $\theta$ in Theorem \hyperlink{thm:A}{A} means that we must overcome two major issues: unlike the setting that $(\kappa,n) = 1$, the representations $R_T^G(\theta)$ are not always compact inductions from finite-index subgroups of $G$, and unlike the setting that $n = 2$, not all $\theta$ are primitive. These issues turned out be quite difficult to resolve; we now explain the strategy of the proof of Theorem \hyperlink{thm:A}{A} and discuss the geometric methods used in it.

%For $\theta$ trivial on $U_L^1$ and with trivial $\Gal(L/K)$-stabilizer, as well as for $\theta$ primitive, Theorem \hyperlink{thm:A}{A} is shown in \cite{CI_ADLV} for all $p,n$. When $G$ is the group of units of a central division algebra over $K$, Theorem \hyperlink{thm:A}{A} gets easier and essentially follows (for all $p,n$ and all $\theta$ with trivial $\Gal(L/K)$-stabilizer) from Lusztig's original work \cite{Lusztig_79} along with a result of Henniart \cite[3.1 Th\'eor\`eme]{Henniart_92}, see \cite{Chan_siDL}.  The (also relatively easier) case $\bfG = \bfGL_2$ was first done in \cite{Ivanov_15_ADLV_GL2_unram}. For $\bfG = \bfGL_2$ and ramified elliptic tori a similar result was shown in \cite{Ivanov_15_ADLV_GL2_ram,Ivanov_18_wild}.

To begin with, ${\bf G}$ has a (unique up to conjugacy) smooth affine model ${\bf G}_\cO$ over the integers $\cO_K$ of $K$, whose $\cO_K$-points are the maximal compact subgroup $G_{\cO} \cong \bfGL_{n'}(\cO_{D_{k_0/n_0}})$, where $\cO_{D_{k_0/n_0}}$ is the ring of integers of $D_{k_0/n_0}$. Moreover, $G_\cO$ can be chosen compatibly with $\bfT$ so that $T_\cO := T \cap G_\cO \cong U_L$ is the maximal compact subgroup of $T$. As is shown in \cite{CI_ADLV} (see also Proposition \ref{prop:representability} below), $X$ admits a scheme-theoretically disjoint decomposition,
\begin{equation}\label{eq:scheme_X_dec}
X = \coprod_{g\in G/G_\cO} g.X_\cO, \qquad \text{where $X_\cO = \prolim_h X_h$}
\end{equation}
is a subscheme equal to an inverse limit of affine perfect schemes $X_h$ ($h = 1, 2, 3,\ldots$), each perfectly finitely presented over $\overline{\FF}_q$. Here $X_\cO$ carries an action of $G_\cO \times T_\cO$ and $X_h$ inherits an action of a certain finite (Moy--Prasad) quotient $G_h \times T_h$ of it. Then $X_1$ is (the perfection of) a classical Deligne--Lusztig variety attached to the reductive quotient of the special fiber of ${\mathbf G}_\cO$ (isomorphic to ${\rm Res}_{\FF_{q^{n_0}}/\FF_q} \mathbf{GL}_{n'}$), and the deeper-level varieties $X_h$ coincide with the (perfections of) varieties considered in \cite{Lusztig_04,Stasinski_09} when ${\bf G}_\cO \otimes_{\cO_K} \FF_q$ is reductive (i.e., $\kappa = 0$), resp.\ with those in \cite{CI_MPDL} in the general case. In this setting, $T_h = U_L/U_L^h = (\cO_L/\frak p_L^h)^\times$ and $G_h$ is a finite quotient of ${\bf GL}_{n'}(\cO_{D_{k_0/n_0}})$ by a congruence subgroup (for example, in the split case $\kappa = 0$, $G_h = \GL_n(\cO_L/\frak p_L^h)$).

Let $Z$ be the center of $G$. Then $T = ZT_\cO$. For a character $\theta$ of $T \cong L^\times$ trivial on the $h$-units $U_L^h$, \eqref{eq:scheme_X_dec} plus the fact that the fibers of $X_h/\ker(T_h \rightarrow T_{h-1}) \rightarrow X_{h-1}$ are affine spaces of a fixed dimension, gives $R_T^G(\theta) = \cind_{ZG_\cO}^G R_{T_h}^{G_h}(\theta)$, where $R_{T_h}^{G_h}(\theta)$ is the $\theta|_{U_L}$-isotypic component of the $\ell$-adic Euler characteristic of $X_h$ (extended to a $ZG_\cO$-representation by letting $z\in Z\cong K^\times$ act by $\theta(z)$). 

The proof of Theorem \hyperlink{thm:A}{A} consists of five steps:
\begin{itemize}
\item[(1)] Show that $\pm R_{T_h}^{G_h}(\theta)$ is an irreducible $G_h$-representation. See Section \ref{sec:Mackey}.
\item[(2)] By similar methods as in (1), show for a certain closed $G_h \times T_h$-stable perfect subscheme $X_{h,n'} \subseteq X_h$, that $\pm H_c^\ast(X_{h,n'})_{\theta}$ is irreducible and $\pm R_{T_h}^{G_h}(\theta) \cong \pm H_c^\ast(X_{h,n'})_\theta$. See Section \ref{sec:VariationMackey}.
\item[(3)] Show (using (1)) that the induction $\pm R_T^G(\theta) = \cind_{ZG_\cO}^G (\pm R_{T_h}^{G_h}(\theta))$ is admissible (equivalently, a finite direct sum of irreducible supercuspidals). See Section \ref{sec:Cuspidality}.
\item[(4)] Use \cite{CI_DrinfeldStrat} to compute the degree $\deg H_c^\ast(X_{h,n'})_\theta$ of the (finite-dimensional) representation $\pm H_c^\ast(X_{h,n'})_\theta$, which is then by (2) also equal to $\deg R_{T_h}^{G_h}(\theta)$. See Section \ref{sec:degree_geometric} and \cite{CI_DrinfeldStrat}.
\item[(5)] Using (3) together with the traces of  $R_T^G(\theta)$ \cite[Theorem 11.2]{CI_ADLV} and of $\pi_\theta$ on very regular elements (cf. Section \ref{sec:traces_ell_elements} for a definition) in $T \subseteq G$, conclude by using an argument due to Henniart \cite{Henniart_92} using linear independence of characters, along with matching $\deg R_{T_h}^{G_h}(\theta)$ from (4) with the explicitly known formal degree of $\pi_\theta$ \cite{CorwinMS_90}. See Section \ref{sec:comparison}.
\end{itemize}

Let us briefly comment on steps (1)-(4) here. Step (1) relies on a precise analysis of the quotient $G_h \backslash (X_h \times X_h)$ in the setting described above. This is remarkably delicate, with difficulties that distinguish it both from the proof of the $h=1$ scalar product formula in classical Deligne--Lusztig theory \cite[Theorem 6.8]{DeligneL_76} and from the analogous formula for $h>1$ in the regular case \cite[Proposition 2.3]{Lusztig_04} (for more details, see Remark \ref{rem:two_principles}). This careful analysis culminates in showing the following particular Mackey formula for ``Deligne--Lusztig induced'' $G_h$-representations.

\begin{thmb}[see Theorem \ref{thm:general_irred}, Corollary \ref{cor:irreducibility}] \hypertarget{thm:B}
Let $\theta,\theta'$ be two characters of $T_h$. Then 
\[
\left\langle R_{T_h}^{G_h}(\theta), R_{T_h}^{G_h}(\theta') \right\rangle_{G_h} = \# \left\{ w \in W_\cO^F \colon \theta'  = \theta \circ {\rm ad}(w) \right\},
\]
where $W_\cO$ is the Weyl group of the special fiber of $\bfG_\cO$ and $F$ is the Frobenius of $\bfG$ acting on it. Moreover, if the stabilizer of $\theta|_{U_L^1}$ in $\Gal(L/K)[n']$, the unique subgroup of $\Gal(L/K)$ of order $n'$, is trivial, then $\pm R_{T_h}^{G_h}(\theta)$ is an irreducible $G_h$-representation and the map
\begin{align*}
\{\text{characters $\theta \colon T_h \rightarrow \overline\bQ_\ell^\times$ in general position}\}\slash W_\cO^F &\rightarrow \{\text{irreducible $G_h$-representations} \}  \\
\theta &\mapsto \pm R_{T_h}^{G_h}(\theta)
\end{align*}
is injective.
\end{thmb}

The point of Theorem \hyperlink{thm:B}{B} is that in our setting, it nullifies the assumption ``$\theta$ regular'' which was crucial in \cite{Lusztig_04}.
%The theorem looks like a special case (for given $G_h,T_h$) of the results of \cite{Lusztig_04}, but the assumption ``$\theta$ regular'' (equivalently, primitive) which was crucial in  \cite{Lusztig_04} is removed. 
One can hope that similar methods as used in the proof of Theorem \hyperlink{thm:B}{B} could lead to a general Mackey formula for all elliptic unramified tori in reductive $K$-groups.

Step (2) is a technically more elaborate version of the same idea implemented in step (1). It is (among other things) responsible for the assumptions that $p>n$ and that $\theta|_{U_L^1}$ has trivial stabilizer in $\Gal(L/K)$. See Remark \ref{rem:two_principles}.

Step (3) relies on the study of the quotient $N_h \backslash X_h$ where $N_h \subseteq G_h$ is a subgroup corresponding to the unipotent radical of a proper parabolic subgroup of $\bfG$. Once this quotient is described (Lemma \ref{lm:description_of_quotient_kappa_arbitrary}), (3) is easy to show. The main technical role in this description is played by classical minor identities, dating back to 1909 results of Turnbull \cite{Turnbull_1909}.

Step (4), mainly performed in \cite{CI_DrinfeldStrat} is based on the determination of the action of Frobenius (over $\FF_{q^n}$) in the cohomology of $X_{h,n'}$. This determination is strongly related to the amazing fact that $X_{h,n'}$ is a maximal variety over $\FF_{q^n}$, i.e., $\#X_{h,n'}(\FF_{q^n})$ attains its Weil--Deligne bound, prescribed by the Lefschetz fix point formula and the dimensions of the single $\ell$-adic cohomology groups. 

\smallskip

Finally, we state a conjecture for other groups, and discuss the construction of a $p$-adic Deligne--Lusztig stack, which allows to relate our results to the work of Zhu and Xiao--Zhu on the stack of isocrystals \cite{Zhu_20, XiaoZ_17}, as well as to the recent work of Fargues--Scholze \cite{FarguesScholze}. In the rest of the introduction, $\bfG$ denotes any unramified reductive group over $K$.

\subsection*{Related work}

This paper, in both its results and its methods, has served as motivation for many directions. We mention several works that have been established since the time this paper was written. We additionally indicate expected connections to recent and fast-moving developments in the realm of the local Langlands conjecture. In the rest of the introduction, $\bfG$ denotes any unramified reductive group over $K$ and $\bfG'$ denotes an extended pure inner form of $\bfG$. Let $\bfT \hookrightarrow \bfG'$ be an arbitrary (not necessarily elliptic) unramified maximal torus. As usual, we write $G = \bfG(K),$ $G' = \bfG'(K),$ $T = \bfT(K)$.

\subsubsection*{Various Deligne--Lusztig spaces for $p$-adic groups}
In \cite{CI_MPDL}, we defined affine perfect schemes $X_h$ of perfectly finite presentation associated to certain parahoric subgroups of $\bfG'$, extending \cite{Lusztig_04,Stasinski_09}. Each has an action of a certain quotient $T_h \times G_h$ of Moy--Prasad subgroups, allowing one to define a deeper version of Deligne--Lusztig induction, which we denote by $R_{T_h}^{G_h'}(\theta)$. 

In an even more general setup (where, in particular, $\bfG'$ is allowed to be an inner form of a Levi subgroup of $\bfG$), certain ``big'' $p$-adic Deligne--Lusztig spaces attached to $\bfG$, $\bfG'$ and $\bfT$ were introduced in \cite{Ivanov_DL_indrep}. They admit an action by $G' \times T$. In the special case of Coxeter tori and inner forms of $\bfGL_n$ these spaces give back the space $X$ as in \eqref{eq:scheme_X_dec}.

A priori these big $p$-adic Deligne--Lusztig spaces are huge objects (in particular, infinite-dimensional and often not representable by schemes). To develop a formalism capturing the $\ell$-adic cohomology of such objects is the task of an ongoing project of L. Mann with the second author. This formalism then attaches a smooth $G'$-representation $R_T^{G'}(\theta)$ to a smooth character $\theta \colon T \rightarrow \overline\bQ_\ell^\times$. Without using this formalism we are able to define $R_T^{G'}(\theta)$ in the special case of the present article in an ad hoc way using \eqref{eq:scheme_X_dec} and the representations $R_{T_h}^{G_h'}(\theta)$. By \cite[Thm.1.1]{Ivanov_pDL_Cox}, the same also works whenever $\bfG$ is classical and $\bfT$ is Coxeter.

\smallskip

Below we describe a couple of works on $R_{T_h}^{G_h'}(\theta)$.

% Below, we describe a couple works on this for $\bfG$ outside the special case $\bfG = \mathbf{GL}_n$. Defining $R_T^{G'}(\theta)$ in general, however, is a quite serious task; see below for recent work tackling this problem.

\subsubsection*{$R_{T_h}^{G_h'}(\theta)$ and supercuspidal representations}

The work in the present paper establishes the irreducible supercuspidality of the compact induction $\cInd(R_{T_h}^{G_h'}(\theta))$ in the case that $\bfG = \GL_n$ for most characters $\theta$. For arbitrary $\bfG$ and $\bfT$, when the size of the residue field of $K$ is sufficiently large, the irreducible supercuspidality of this compact induction is established in work of the first author with M.\ Oi \cite{ChanOi_toral} for characters $\theta$ satisfying a genericity condition analogous to primitivity. Moreover, the correspondence $\theta \mapsto \cInd(R_{T_h}^{G_h'}(\theta))$ is the ``correct'' correspondence, in the sense of DeBacker--Spice \cite{DeBSpi}, and is compatible with current constructions of supercuspidal $L$-packets. This is achieved via an explicit algebraic comparison to Yu's construction of supercuspidal representations \cite{Yu_01}, in contrast to the \textit{geometric} methods established in the present paper.

\subsubsection*{Mackey formula for $R_{T_h}^{G_h'}(\theta)$}

If $\bfG'$ is $K$-split and $\bfT$ is Coxeter, the analogue of Theorem \hyperlink{thm:B}{B} holds, due to work of the second author and O.\ Dudas \cite[Thm.~3.2.3]{DudasI_20}. The proof of \cite[Thm.~3.2.3]{DudasI_20} is a quite non-straightforward generalization of our proof of Theorem \hyperlink{thm:B}{B}.

\subsection*{A $p$-adic Deligne--Lusztig stack. Relation to Fargues--Scholze}

Finally, in \S\ref{sec:DL_sheaves_on_Isoc_and_Bun} we put our results in the context of the work of Zhu \cite{Zhu_20}, Xiao--Zhu \cite{XiaoZ_17} and Fargues--Scholze \cite{FarguesScholze}. The starting point is that there is a natural way to organize (certain variant of) the spaces $X = X^\bfG$ into a family, where $\bfG$ continuously varies through all extended pure inner forms of $\bfGL_n$ (or more generally, of any given unramified group). More precisely, fix an unramified reductive group $\bfG$. Zhu and Xiao--Zhu consider the stack ${\rm Isoc}_{\bfG}$, parametrizing certain $\bfG$-isocrystals, whose geometric points correspond to extended pure inner forms of $\bfG$. We  sketch the construction of a $v$-stack $X_w \rightarrow {\rm Isoc}_{\bfG}$, whose fibers are (in the special case $\bfG = \bfGL_n$, $w$ Coxeter), essentially the spaces $X/T$.

Attached to any character $\theta$ of $T$, there should\footnote{Sections \ref{sec:DL_sheaf_on_isoc} and \ref{sec:diamond_attached_to_isoc} are conditional on the  existence of a six functor formalism for solid pro-\'etale sheaves on $v$-stacks on ${\rm Perf}_{\overline{\bF}_q}$, which is not developed yet (this is, however, the aim of the ongoing work of L. Mann and the second author).} exist a ``$p$-adic Deligne--Lusztig complex'' ${\rm DL}_{w,\theta}$ in an appropriate category $D_{\rm lis}({\rm Isoc}_{\bfG},\overline\bQ_\ell)$ of lisse pro-\'etale sheaves on $X_w$. Theorem \hyperlink{thm:A}{A} can now be restated as the computation of the restriction of ${\rm DL}_{w,\theta}$ to the basic locus, cf. Corollary \ref{cor:main_result_in_terms_of_DLsheaf}. Moreover, through the corresponding $v$-stack ${\rm Isoc}_{\bfG}^\diamond$ over perfectoid spaces over $\overline\bF_q$, ${\rm Isoc}_{\bfG}$ is related to the stack ${\rm Bun}_{\bfG}$ of $\bfG$-bundles on the Fargues--Fontaine curve, a central object of \cite{FarguesScholze}. Via this relation, ${\rm DL}_{w,\theta}$ gives rise to an object ${\rm DL}'_{w,\theta} \in D_{\rm lis}({\rm Bun}_{\bfG},\overline\bQ_\ell)$. We conclude by stating Conjecture \ref{conj:DL_sheaf_on_Bun_G} concerning its behavior and relation to Fargues' conjectural Hecke eigensheaf (which in the case of the present article was constructed by Ansch\"utz--LeBras \cite[Thm. 2.1]{AnschuetzLB_21}).

\subsection*{Acknowledgements}
We would like to thank Guy Henniart, Tasho Kaletha, and Peter Scholze for helpful advice, and Andreas Mihatsch, Johannes Ansch\"utz and Ian Gleason for several useful discussions on the subject of this article. We are very grateful to the anonymous referee for several very interesting suggestions and insights, especially concerning \S\ref{sec:DL_sheaves_on_Isoc_and_Bun}. The first author was partially supported by the DFG via the Leibniz Prize of Peter Scholze, by NSF grant DMS-1641185 (US Junior Oberwolfach Fellow), and by an NSF Postdoctoral Research Fellowship, DMS-1802905. The second author was supported by the DFG via the Leibniz Preis of Peter Scholze.

\section{Notation}\label{sec:notation} 

For a non-archimedean local field $M$ we denote by $\cO_M, \fp_M, U_M = \cO_M^\times$ resp.\ $U_M^h = 1 + \fp_M^h$ (with $h \geq 1$) its integers, maximal ideal, units resp.\ $h$-units.
 
Throughout the article we fix a non-archimedean local field $K$ with uniformizer $\varpi$ and residue field $\FF_q$ of characteristic $p$ with $q$ elements. We denote by $\breve K$ the completion of a fixed maximal unramified extension of $K$, and by $\cO_{\breve K}$ the integers of $\breve K$. The residue field $\overline{\FF}_q$ of $\breve K$ is an algebraic closure of $\FF_q$, and $\varpi$ is still a uniformizer of $\breve K$. We write $\sigma$ for the Frobenius automorphisms of $\breve K/K$ and of $\overline{\FF}_q/\FF_q$. 

Fix an integer $n \geq 2$. We denote by $K \subseteq L \subseteq \breve K$ the unique subextension of degree $n$. Moreover, for any positive divisor $r$ of $n$ we denote by $K \subseteq K_r \subseteq K_n = L$ the unique subextension of degree $r$ over $K$.

Fix another integer $0 \leq \kappa < n$ and write $n = n'n_0$, $\kappa = n'k_0$, where $n' = {\rm gcd}(n,\kappa)$. Then $n_0, k_0$ are coprime. 

Fix a prime $\ell \neq p$ and let $\overline \QQ_\ell$ be a fixed algebraic closure of $\QQ_\ell$. All cohomology groups of (perfections of) quasi-projective schemes of finite type over $\overline\FF_q$ will be compactly supported \'etale cohomology groups with coefficients in $\overline \QQ_\ell$. For such a scheme $Y$ (and more generally, whenever the cohomology groups are defined), we write $H_c^\ast(Y) := \sum_{i \in \bZ} H_c^i(Y,\overline{\QQ}_\ell)$ (the coefficients always will be $\overline{\QQ}_\ell$, so there is no ambiguity).

Unless otherwise stated, all representations of locally compact groups appearing in this article will be smooth with coefficients in $\overline \QQ_\ell$.

\section{Coxeter-type loop Deligne--Lusztig scheme in type \texorpdfstring{$\tilde A_{n-1}$}{$A_{n-1}$} }

Let $n = n'n_0 \geq 2$ and $\kappa = n'k_0$ with $\gcd(k_0,n_0) = 1$ be as in Section \ref{sec:notation}. This notation remains fixed throughout the article.

In this section we review some constructions and results concerning loop Deligne--Lusztig schemes of Coxeter type for inner forms of $\bfGL_n$ from \cite{CI_ADLV}, and we simplify the proof of representability (Proposition \ref{prop:representability}).

\subsection{Inner forms of $\bfGL_n$ and elliptic tori}\label{sec:basic_notation}
Inside the group $\bfGL_n$ over $K$ we fix a split maximal torus $\bfT_0$ and the unipotent radicals $\bfU_0, \bfU_0^-$ of two opposite $K$-rational Borel subgroups containing $\bfT_0$. Let the roots of $\bfT_0$ in $\bfU_0$ be the positive roots, determining a set $S_0$ of simple roots. Conjugating if necessary, we may assume that $\bfT_0$ is the diagonal torus and $\bfU_0$ is the group of upper triangular unipotent matrices.

\subsubsection{Forms of $\bfGL_n$} The Kottwitz map \cite{Kottwitz_85}
\[
\kappa_{\bfGL_n} = {\rm val} \circ \det \colon B(\bfGL_n)_{\rm basic} \rightarrow \bZ
\]
for $\bfGL_n$ defines a bijection between the set of basic $\sigma$-conjugacy classes in ${\bf GL}_n(\breve K)$ and $\bZ$. 
Fix a basic element $b \in \bfGL_n(\breve K)$ with $\kappa_{\bfGL_n}(b) = \kappa$. Let $\bfG$ be the $K$-group defined by 
\[
\bfG(R) = \{ g \in \bfGL_n(R \otimes_K \breve K) \colon g^{-1}b\sigma(g) = b \}
\]
(this is the group $J_b$ from \cite[1.12]{RapoportZ_96}). Then $\bfG$ is an inner form of $\bfGL_n$ and we may identify $\bfG(\breve K) = \bfGL_n(\breve K)$. 
The Frobenius on $\bfG(\breve K)$ is $F_b \colon g \mapsto b\sigma(g)b^{-1}$. The $K$-points of $\bfG$ are
\[
G := \bfG(K) \cong \bfGL_{n'}(D_{k_0/n_0}).
\]
We may identify the adjoint Bruhat--Tits building of $\bfG$ over $\breve K$ with that of $\bfGL_n$. Denote both of them by $\mathscr{B}_{\breve K}$. The adjoint Bruhat--Tits building of $\bfG$ over $K$ is the subcomplex $\mathscr{B}_K = \mathscr{B}_{\breve K}^{F_b}$. Let ${\bf x}_b \in \mathscr{B}_K$ be a vertex. Bruhat--Tits theory \cite[5.2.6]{BruhatT_84} attaches to ${\bf x}_b$ a (maximal) parahoric $\cO_K$-model $\bfG_\cO$ of $\bfG$, whose $\cO_K$-points 
\[
G_\cO := \bfG_{\cO}(\cO_K) \cong \bfGL_{n'}(\cO_{D_{\kappa_0/n_0}}),
\] 
form a maximal compact subgroup of $G$.

\begin{rem}
The groups $\bfG,\bfG_\cO,G,G_\cO$ depend on the choice of $b$, but if $b' = h^{-1}b\sigma(h)$ ($h\in \bfGL_n(\breve K)$) is another choice inside the same basic $\sigma$-conjugacy class, with corresponding groups $\bfG',G'$, then conjugation with $h$ defines an isomorphism of $\bfG, G$ and $\bfG', G'$, and if ${\bf x}_b$ is mapped by $h$ to ${\bf x}_{b'}$, then conjugation by $h$ maps $\bfG_\cO, G_\cO'$ to $\bfG_\cO,G_\cO'$. As at the end we are interested in isomorphism classes of representations of $G$ (or $G_\cO$), which are not affected by these isomorphisms, we leave the choice of $b$ unspecified as long as possible. When we need concrete realizations of $\bfG, \bfG_\cO, G, G_\cO$ (in Sections \ref{sec:Step1}, \ref{sec:proof_no_triv_char_kappa_0} and \ref{sec:proof_no_triv_char_kappa_arbitrary}) we will exploit the freedom of choosing different $b$'s inside the same basic $\sigma$-conjugacy class). 
\end{rem}

\subsubsection{Forms of $\bfT_0$}\label{sec:form_of_torus} Let $W = W(\bfT_0,\bfGL_n)$ be the Weyl group of $\bfT_0$ in $\bfGL_n$, then $(W, S_0)$ form a Coxeter system. Let $w_1 = \left(\begin{smallmatrix}0&1\\1_{n-1}&0\end{smallmatrix}\right)\in W$. It is a Coxeter element of $(W, S_0)$. Let $\dot w_1 \in N_{\bfGL_n}(\bfT_0)(\breve K)$ be a lift of $w_1$.
% , satisfying $\kappa_{\bfGL_n}(\dot w) = \kappa$. 
Then ${\rm Ad}(\dot w_1)$ induces an automorphism of the apartment $\mathscr{A}_{\bfT_0,\breve K} \subseteq \mathscr{B}_{\breve K}$ of $\bfT_0$. It has precisely one fixed point ${\bf x}_{\dot w_1}$ as $w_1$ is Coxeter. 
Let $\mathcal{G}$ be the parahoric $\cO_K$-model of $\bfGL_n$ attached to this fixed point.
Let $\mathcal{T}$ be the schematic closure of $\bfT_0$ in $\mathcal{G}$. Let $\bfT$ denote the (outer) form of $\bfT_0$, which splits over $\breve K$, and is endowed with the Frobenius $F_{\dot w_1} \colon t \mapsto \dot w_1 \sigma(t)\dot w_1^{-1}$ (independent of the lift $\dot w_1$), and similarly let $\bfT_\cO$ be the (outer) form of $\mathcal{T}$, which splits over $\cO_{\breve K}$, and is endowed with the same Frobenius. We get the group
\[
T := \bfT(K) \cong L^\times \quad \text{ and its subgroup } \quad T_\cO := \bfT_\cO(\cO_{\breve K}) \cong \cO_L^\times,
\]
where $L/K$ is unramified of degree $n$. In fact, $T=\{\diag(x,\sigma(x),\dots,\sigma^{n-1}(x)) \colon x \in L^\times \}$ (recall that $\bfT_0$ is diagonal), and the isomorphism with $L^\times$ is determined up to composition with an element in $\Gal(L/K)$.

\subsubsection{Case $b=\dot w_1$.}\label{sec:case_bequalw} In the special case $b = \dot w_1$ and ${\bf x}_b = {\bf x}_{\dot w_1}$, we have only one Frobenius $F := F_b = F_{\dot w_1}$, $\bfG_\cO$ is a form of $\mathcal{G}$, and $\bfT$ is an elliptic maximal torus of $\bfG$, and $\bfT_\cO$ is a maximal torus of $\bfG_{\cO}$. There are unique (closed, reduced) subgroups $\bfU,\bfU^-$ of $\bfG$, such that $\bfU(\breve K) = \bfU_0(\breve K)$, $\bfU^-(\breve K) = \bfU_0^-(\breve K)$ inside $\bfG(\breve K) = \bfGL_n(\breve K)$. Inside $\bfG_\cO$ we will need the schematic closures $\bfU_\cO$ and $\bfU_\cO^-$ of $\bfU$ and $\bfU^-$. 

The Frobenius $F$ acts on the roots of $\bfT$ in $\bfG$, so that there is a unique subgroup $F\bfU \subseteq \bfG$, satisfying $(F\bfU)(\breve K) = F(\bfU(\breve K))$, and similarly for $\bfU^-, \bfU_{\cO}, \bfU_\cO^-$. Identifying $W$ with the Weyl group of $\bfT$ in $\bfG$, $F$ acts on $W$. Moreover, $W^F = \langle w_1 \rangle \cong \bZ/n\bZ$ is the subgroup generated by $w_1$. It acts on $T$ and the chosen isomorphism $T \cong L^\times$ induces an isomorphism $W^F \cong \Gal(L/K)$, sending $w_1$ to the image of $\sigma$ in $\Gal(L/K)$.

The maximal torus in the reductive quotient of the special fiber $\bfT_\cO \otimes_{\cO_K} \FF_q \subseteq (\bfG_\cO \otimes_{\cO_K} \FF_q)^{\rm red}$ is elliptic. Explicitly, these groups are isomorphic to ${\rm Res}_{\FF_{q^n}/\FF_q} \bG_m \subseteq {\rm Res}_{\FF_{q^{n_0}}/\FF_q} \GL_{n',\FF_{q^{n_0}}}$. Let $W_\cO$ be the Weyl group of $\bfT_\cO \otimes_{\cO_K} \FF_q$ in $(\bfG_\cO \otimes_{\cO_K} \FF_q)^{\rm red}$. It is naturally a subgroup of $W$, $F$ acts on $W_\cO$ and $W_\cO^F$, which is generated by $w_1^{n_0}$, is isomorphically mapped onto $\Gal(L/K_{n_0})$ under the above isomorphism $W^F \cong \Gal(L/K)$.

\subsection{Perfect schemes} 

Let $k$ be a perfect field of characteristic $p$ and let $X$ be a $k$-scheme. Let $\phi = \phi_X \colon X \rightarrow X$ be the absolute Frobenius morphism of $X$, that is $\phi$ is the identity on the underlying topological space and is given by $x \mapsto x^p$ on $\cO_X$. The scheme $X$ is called \emph{perfect} if $\phi$ is an isomorphism. Let ${\rm Alg}_{k}$ denote the category of all $k$-algebras, and let ${\rm Perf}_{k}$ be the full subcategory of perfect $k$-algebras. Then the restriction functor which sends a perfect $k$-scheme, regarded as a functor on ${\rm Alg}_{k}$, to a functor on ${\rm Perf}_{k}$ is fully faithful \cite[A.12]{Zhu_17}. Thus we equally may regard a perfect scheme as a functor on ${\rm Perf}_{k}$, which has an open covering by representable functors in the usual sense. Every $k$-scheme $X_0$ admits a \emph{perfection}, namely $X_0^{\rm perf} := \lim_{\phi} X_0$, which is a perfect scheme. For example, the perfection of ${\rm Spec}\, k[T]$ is ${\rm Spec}\, k[T^{1/p^{\infty}}]$, where $k[T^{1/p^{\infty}}] = \bigcup_{r\geq 0} k[T^{p^{-r}}]$.

Except stated otherwise, throughout this article we will work with perfect schemes over $k = \overline{\FF}_q$ (or $k = \FF_q$). So, to simplify notation we write $\bA^m = \bA^m_k$ resp. $\bG_a$ resp. $\bG_m$ for the \emph{perfection} of the $m$-dimensional affine space resp. the additive resp. the multiplicative group over $k$. A morphism $f\colon \Spec A \rightarrow \Spec B$ of affine perfect schemes is \emph{perfectly finitely presented}, if there is a $A = (A_0)_{\rm perf}$ for a finitely presented $B$-algebra $A_0$ \cite[3.10,3.11]{BhattS_17}. For further results on perfect schemes we refer to \cite[Appendix A.1]{Zhu_17} and \cite[\S3]{BhattS_17}. Here we only mention the following lemmas.

\begin{lm}\label{lm:perf_fin_pres}
Let $X \subseteq \bA_k^m$ be a closed perfect subscheme of the $m$-dimensional perfect affine space. Then $X \rightarrow {\rm Spec}\, k$ is perfectly finite presented. 
\end{lm}

\begin{proof}
Let $T = (T_1,T_2,\dots,T_m)$ be some coordinates on $\bA_k^m$. Let $\fa$ be the ideal of $X$ in the coordinate ring $k[T^{p^{-\infty}}]$ of $\bA_k^m$. Then it is easy to check that $X$ is the perfection of $X_0 = \Spec k[T]/(\fa \cap k[T])$, which is (reduced and) finitely presented over $k$.
\end{proof}

\begin{lm}\label{lm:mono_perf_schemes}
Let $f \colon X \rightarrow Y$ be a morphism of perfect $k$-schemes with $X$ separated. The following are equivalent:
\begin{itemize}
\item[(i)] $f$ is a monomorphism (of fpqc- or \'etale sheaves on $\Perf_k$)
\item[(ii)] for every algebraically closed field $K/k$, $f(K) \colon X(K) \rightarrow Y(K)$ is injective.
\end{itemize}
\end{lm}
\begin{proof}
Assume (ii). To deduce (i) it is enough to show that for any $R \in {\rm Perf}_k$, $f(R) \colon X(R) \rightarrow Y(R)$ is injective. Let $x,y \colon \Spec R \rightarrow X$ be two elements of $X(R)$, such that $fx = fy \in Y(R)$. For each point $p \in \Spec R$, choose a morphism $i_p \colon \Spec K_p \rightarrow \Spec R$ with image $p$, and with $K_p$ an algebraically closed field. Then $fxi_p = fyi_p \in Y(K_p)$ for each $p$, and from (ii) we deduce $xi_p = yi_p$. As $X$ is separated, the equalizer of $x,y$ is a closed subscheme of $\Spec R$, say equal to $\Spec R/I$ for some ideal $I \subseteq R$. Now, $xi_p = yi_p$ for \emph{all} field valued points of $\Spec R$ implies that $I \subseteq \bigcap_{\fp \in \Spec R} \fp = {\rm rad}(0) = 0$, as $R$ perfect and hence reduced. 
The other direction is clear. 
\end{proof}

\subsection{Witt vectors and loop groups}

If $K$ has positive characteristic, we denote by $\bW$ the ring scheme over $\FF_q$, where for any $\FF_q$-algebra $R$, $\bW(R) = R[\![\varpi]\!]$. If $K$ has mixed characteristic, we denote by $\bW$ the $K$-ramified Witt ring scheme over $\FF_q$ so that $\bW(\FF_q) = \cO_K$ and $\bW(\overline \FF_q) = \cO_{\breve K}$ (see e.g.\ \cite[1.2]{FarguesFontaine_book}). Let $\bW_h = \bW/V^h \bW$ be the truncated  ring scheme, where $V \from \bW \to \bW$ is the multiplication by $\varpi$ (if ${\rm char}\, K >0$) resp. the Verschiebung morphism (if ${\rm char} K = 0$). We regard $\bW_h$ as a functor on $\Perf_{\FF_q}$, where it is represented by $\bA^h_{\FF_q}$. We denote by $\bW_h^\times$ the perfect group scheme of invertible elements of $\bW$ and for $1\leq a<h$, we denote by $\bW_h^{\times,a} = {\rm ker}(\bW_h^\times \rightarrow \bW_a^\times)$ the kernel of the natural projection.

If $\bfX$ is a $\breve K$-scheme, the loop space $L\bfX$ of $\bfX$ is the functor on ${\rm Perf}_{\overline{\FF}_q}$,
\[
R \mapsto L\bfX(R) = \bfX(\bW(R)[\varpi^{-1}]).
\]
If $\bfX$ is an affine $\breve K$-scheme of finite type, $L\bfX$ is represented by an ind-(perfect scheme) \cite[Proposition 1.1]{Zhu_17}. If $\mathcal{X}$ is a $\cO_{\breve K}$-scheme, the spaces of (truncated) positive loops of $\mathcal{X}$ are the functors on ${\rm Perf}_{\overline{\FF}_q}$,
\[
R \mapsto L^+\mathcal{X}(R) = \mathcal{X}(\bW(R)) \quad \text{ resp. } \quad R \mapsto L^+_h\mathcal{X}(R) = \mathcal{X}(\bW_h(R)).
\]
($h\geq 1$). If $\mathcal{X}$ is an affine $\cO_{\breve K}$-scheme of finite type,  $L^+\mathcal{X}$, $L^+_h \mathcal{X}$ are represented by affine perfect $\overline{\FF}_q$-schemes, and $L^+_h \mathcal{X}$ are perfectly finitely presented (by Lemma \ref{lm:perf_fin_pres}). The same holds with $\overline{\FF}_q$ replaced by $\FF_q$.

\begin{rem}
We could evaluate $\bW$, $\bW_h$ and $L^+\mathcal{X}$, $L_h^+\mathcal{X}$ on all $R \in {\rm Alg}_{\overline{\FF}_q}$ and thereby work with schemes $L_h^+\mathcal{X}$ of finite type over $\overline{\FF}_q$, instead of perfect schemes. 
% We exploit this in \cite{CI_DrinfeldStrat} and also in \cite{CI_ADLV} (see especially Proposition 6.4 of the present paper, which is originally appeared in \cite[Theorem 11.2]{CI_ADLV}). 
Still, one must take care when working with the functors $L, L^+$ in the mixed characteristic setting---see for example \cite[Remark 9.3]{BhattS_17} and \cite[end of Section 1.1.1]{Zhu_17} for some warnings. But even in the equal characteristic case, when working with $L\bfX$, we are really forced to work in the category of perfect schemes; indeed, as we use an argument on geometric points in the proof of Proposition \ref{prop:representability}, we can only make our final conclusion when there is no non-reduced structure (which is the case only after perfection). Therefore, for the entirety of this paper, we pass to perfect schemes everywhere. As passing to the perfection is a universal homeomorphism, this does not affect \'etale cohomology.

% We could evaluate $\bW$, $\bW_h$ and $L^+\mathcal{X}$, $L_h^+\mathcal{X}$ on all $R \in {\rm Alg}_{\overline{\FF}_q}$ and therefore work with schemes $L_h^+\mathcal{X}$ of finite type over $\overline{\FF}_q$, instead of perfect schemes. When ${\rm char}\, K > 0$, this causes no problems, so we equally good could work with schemes of finite type over $\overline{\FF}_q$ instead of their perfections. When ${\rm char}\, K = 0$, these objects behave badly on non-perfect algebras (see e.g. \cite[Remark 9.3]{BhattS_17}, \cite[end of 1.1.1]{Zhu_17}). Therefore we pass to perfect schemes everywhere. Perfections are universal homeomorphisms, hence do not affect \'etale cohomology. 
%
%We point out that when working $L\bfX$, there is a further reason to pass to perfection, which is also valid in characteristic $p$. Indeed, in the proof of Proposition \ref{prop:representability} below, an argument on geometric points is used, and the conclusion can only be made if there is no non-reduced structure (which is the case after perfection).
\end{rem}

\subsection{The perfect $\overline{\FF}_q$-space \texorpdfstring{$X_{\dot w}^{DL}(b)$}{$X$}}\label{sec:perfect_DL_space}
By a \emph{perfect $\overline{\FF}_q$-space} we mean an fpqc-sheaf on $\Perf_{\overline{\FF}_q}$.  Let $b$ be any basic element with $\kappa_{\bfGL_n}(b) = \kappa$. Let $\dot w_1 \in N_{\bfGL_n}(\bfT_0)(\breve K)$ be any lift of $w_1$.  Let $\dot X_{\dot w_1}^{DL}(b)$ denote the fpqc-sheafification of the presheaf on ${\rm Perf}_{\overline{\FF}_q}$,
\begin{equation}\label{eq:perfect_DL_space_w_b}
R \mapsto \{g \in L\bfGL_n(R)/L\bfU_0(R) \colon g^{-1}b\sigma(g) \in L\bfU_0(R)\dot w_1 L\bfU_0(R) \}.
\end{equation}
If $G,T$ are as in Section \ref{sec:basic_notation}, the group $G\times T$ acts on $\dot X_{\dot w}^{DL}(b)$  by $g,t \colon x \mapsto gxt$.

\begin{lm}\label{lm:changing_b_dotw} Let $b$ be basic with $\kappa_{\bfGL_n}(b) = \kappa$ and let $\dot w_1$ be any lift of $w_1$.
\begin{itemize}
\item[(i)] If $b' = h^{-1}b\sigma(h)$ for some $h \in \bfGL_n(\breve K)$, and if $G' = \bfG'(K)$ is the group attached to $b'$ as in Section \ref{sec:basic_notation}, then ${\rm Ad}_h \colon G \rightarrow G'$, $g \mapsto h^{-1}gh$ is an isomorphism. Moreover, left multiplication by $h$ induces an isomorphism of $\overline{\FF}_q$-spaces $\dot X_{\dot w_1}^{DL}(b) \cong \dot X_{\dot w_1}^{DL}(b')$, which is equivariant with respect to the isomorphism $({\rm Ad}_h, {\rm id}) \colon G \times T \rightarrow G'\times T$.
\item[(ii)] Let $\dot w_1'$ be a second lift of $w_1$ to ${\bf GL}_n(\breve K)$.  Assume that $\kappa_{\bfGL_n}(\dot w_1) = \kappa_{\bfGL_n}(\dot w_1')$. Then there exists a $\tau \in \bfT_0(\breve K)$ with $\dot w_1' = \tau^{-1}\dot w_1\sigma(\tau)$. Let $T' = \bfT'(K)$ be the group attached to $\dot w_1'$ as in Section \ref{sec:basic_notation}. Then ${\rm Ad}_\tau \colon T \rightarrow T'$, $t \mapsto \tau^{-1}t\tau$ is an isomorphism. Moreover, right multiplication by $\tau$ induces an isomorphism of $\overline{\FF}_q$-spaces $\dot X_{\dot w_1}^{DL}(b) \cong \dot X_{\dot w_1'}^{DL}(b)$, which is equivariant with respect to the isomorphism $({\rm id}, {\rm Ad}_\tau) \colon G \times T \rightarrow G \times T'$.
\item[(iii)] $\dot X_{\dot w_1}^{DL}(b) = \varnothing$, unless $\kappa_{\bfGL_n}(\dot w_1) = \kappa$.
\end{itemize}
\end{lm}

\begin{proof}
(i): This is an easy computation.
% (i): The first claim is immediate. The second follows as $g \mapsto gh$ defines the required isomorphism already on the level of presheaves.
(ii): The fiber over $w_1$ in ${\bf GL}_n(\breve K)$ is a principal homogeneous space under ${\bf T}(\breve K)$, and it is easy to see that as $w_1$ is Coxeter, the map $t \mapsto {\rm Ad}(w_1)(t)^{-1}\sigma(t)$ from ${\bf T}(\breve K)$ to $\{ \tau \in \bfT(\breve K) \colon \kappa_{\bfGL_n}(\tau) = 0\}$ is surjective. The rest is an easy computation.
(iii): As $\dot X_{\dot w_1}^{DL}(b)$ is an inverse limit of perfectly finitely presented perfect $\overline{\FF}_q$-schemes, it suffices to show that $\dot X_{\dot w_1}^{DL}(b)(\overline{\FF}_q) = \varnothing$. This holds as $\kappa_{\bfGL_n}(g^{-1}b\sigma(g)) = \kappa_{\bfGL_n}(b) = \kappa$ and $\kappa_{\bfGL_n}(LU(\overline{\FF}_q)) = 0$.
\end{proof}

\subsection{Representability} We simplify the proof of representability of $X_{\dot w_1}^{DL}(b)$ from \cite{CI_ADLV}. Let $b = \dot w_1$ be basic with $\kappa_{\bfGL_n}(b) = \kappa$. Then we are in the setup of Section \ref{sec:case_bequalw}. Write $F \colon L\bfG \rightarrow L\bfG$ for the $\overline{\FF}_q$-morphism of ind-(perfect schemes) corresponding to $F \colon \bfG(\breve K) \rightarrow \bfG(\breve K)$, $g \mapsto b\sigma(g)b^{-1}$. Define the fpqc-sheafification $X'$ of the presheaf on ${\rm Perf}_{\overline{\FF}_q}$,
\[
R \mapsto \{ x \in L\bfG(R) \colon x^{-1}F(x) \in F(L{\bf U}) \}/L({\bf U} \cap F{\bf U}).
\]
The group $G \times T$ acts on $X'$ by $g,t \colon x \mapsto gxt$. 
Define $X_\cO$ as the fpqc-sheafification of the presheaf on ${\rm Perf}_{\overline{\FF}_q}$,
\[
X_\cO \colon R\mapsto \{x \in L^+\bfG_\cO(R) \colon x^{-1}F(x) \in L^+(F\bfU_\cO \cap \bfU_\cO^-)(R) \}.
\]
Being the preimage of $L^+(F\bfU_\cO \cap \bfU_\cO^-)$ under the Lang-morphism ${\rm Lang}_{F} \colon L^+\bfG_\cO \rightarrow L^+\bfG_\cO$, $g \mapsto g^{-1}F(g)$, $X_\cO$ is representable by a perfect $\overline{\FF}_q$-scheme. Further, the group $G_\cO \times T_\cO$ acts on $X_\cO$ by $(g,t) \colon x \mapsto gxt$. As $T$ is generated by $T_\cO$ and the central element $\varpi \in T \subseteq G$, the obvious action of $G \times T_\cO$ on $\coprod_{G/G_\cO} g.X_\cO$ extends to an action of $G\times T$ by letting $(1,\varpi)$ act in the same way as $(\varpi,1)$.

\begin{prop}[\cite{CI_ADLV}]\label{prop:representability} Let $b = \dot w_1 \in N_{\bfGL_n}(\bfT_0)(\breve K)$ be basic with $\kappa_{\bfGL_n}(b) = \kappa$, and mapping to $w_1 \in W$. There are $G\times T$-equivariant isomorphisms of perfect $\overline{\FF}_q$-spaces
\begin{equation}\label{eq:disjoint_decomposition_of_X}
X_b^{DL}(b) \cong X' \cong \coprod_{g \in G/G_\cO} g.X_\cO.
\end{equation}
In particular, $X_b^{DL}(b), X'$ are representable by perfect $\overline{\FF}_q$-schemes. 
\end{prop}
\begin{proof}
The same computation as at the end of \cite[\S3]{CI_ADLV} shows that $G\times T$-equivariantly $X_b^{DL}(b) \cong X'$ as $\overline{\FF}_q$-spaces. As the right hand side of \eqref{eq:disjoint_decomposition_of_X} is representable, it suffices to show the second isomorphism in \eqref{eq:disjoint_decomposition_of_X}. Consider the fpqc-sheafification $X''$ of the presheaf on ${\rm Perf}_{\overline{\FF}_q}$, 
\[ R \mapsto \{g \in L\bfG(R) \colon g^{-1}F(g) \in L(F{\bf U} \cap {\bf U}^-)(R)\}. \]
As $w$ is Coxeter, the map
\[
L(F{\bf U} \cap \bfU) \times L(F{\bf U} \cap {\bf U}^-) \rightarrow L(F{\bf U}), \quad (h,g) \mapsto h^{-1}gF(h)
\]
is an isomorphism of fpqc-sheaves (this follows by a concrete calculation -- similar to the part of the proof of \cite[Lemma 2.12]{CI_ADLV} showing equation (7.7) of \emph{loc.~cit.}  -- which can be performed on $R$-points for any $R \in \Perf_{\overline{\FF}_q}$. Compare also \cite{HeL_12}), so that $X' \cong X''$. But $X''$ is the pull-back of the closed sub-(ind-scheme) $L(F{\bf U} \cap {\bf U}^-)$ under the Lang map ${\rm Lang}_F \colon L\bfG \rightarrow L\bf G$, $g \mapsto g^{-1}F(g)$, which is a morphism of ind-schemes, hence $X''$ is representable by an ind-(perfect scheme). 

For $\tau \in \bfT(K)$, $x \mapsto \tau^{-1}x\tau$ defines an equivariant isomorphism between $X''$ and the analogue of $X''$, where $b$ is replaced by $\tau^{-1}b \tau$. Thus we may take $b = \left(\begin{smallmatrix}0 & \varpi^{\kappa} \\ 1_{n-1} & 0 \end{smallmatrix}\right) \cdot \varepsilon$ with $\varepsilon \in \bfT(\cO_{\breve K})$. Fix $R \in {\rm Perf}_{\overline{\FF}_q}$. Let $g \in L\bfG(R) = \bfG(\bW(R)[\varpi^{-1}])$ with $g^{-1}F(g) =: a \in L(F{\bf U} \cap \bfU^-)(R)$. For $1\leq i\leq n-1$, let $a_i \in L\bG_a(R)$ denote the $(i+1,1)$-th entry of the matrix $a$. Then the matrix $g$ is determined by its first column, denoted $v$ (for $1\leq i\leq n$ the $i$-th column is then equal to $(b\sigma)^{i-1}(v)$).  Moreover $v$ has to satisfy $(b\sigma)^n(v) = \varpi^\kappa(v + \sum_{i = 1}^{n-1} a_i (b\sigma)^i(v))$, an equation which takes place in $L\bG_a(R)^n$. Assume $R$ is an algebraically closed field. The valuations of the coefficients of the characteristic polynomial of a $\sigma$-linear endomorphism lie over its Newton polygon, which in our case coincide with the Newton polygon of the isocrystal attached to $b\sigma$, and is just the straight line segment connecting the origin and the point $(n,\kappa)$ in the plane (cf. \cite[Lemma 6.1]{CI_ADLV} for the precise statement). This shows $\val(a_i) \geq -\frac{i\kappa}{n}$ for $1\leq i\leq n-1$. But after explicitly determining the affine root subgroups contained in $\bfG_\cO(\cO_{\breve K})$ (this is a similar computation to \cite[Example 8.8]{CI_ADLV}), this translates to the statement that $a \in L^+(F\bfU_\cO \cap \bfU_\cO^-)(R)$.
As $X''$ is a ind-(perfect scheme), this implies that $X''$ is equal to the fpqc-sheafification of
\[
R \mapsto \{g \in L\bfG(R) \colon g^{-1}F(g) \in L^+(F\bfU_\cO \cap \bfU_\cO^-)(R)\}.
\]
Consider the projection $\pi \colon L\bfG \rightarrow L\bfG/L^+\bfG_\cO$. If $g \in X''(R) \subseteq L\bfG(R)$, then $F(g) \in gL^+(F\bfU_\cO \cap \bfU_\cO^-)(R) \subseteq gL^+\bfG_\cO(R)$. Thus $X''$ maps under $\pi$ to the discrete subset\\ $(L\bfG/L^+\bfG_\cO)^F = G/G_\cO$. Hence $X''$ is isomorphic to the right hand side of \eqref{eq:disjoint_decomposition_of_X}, and we are done.
\end{proof}

\begin{cor}\label{cor:representability_general}
Let $b \in \bfGL_n(\breve K)$ be basic, $\dot w_1$ a lift of $w_1$ such that $\kappa_{\bfGL_n}(b) = \kappa_{\bfGL_n}(\dot w_1) = \kappa$. Then $X_{\dot w_1}^{DL}(b) \cong \coprod_{G/G_\cO}g X_\cO$ is representable by a perfect $\overline{\FF}_q$-scheme.
\end{cor}
\begin{proof}
This follows from Lemma \ref{lm:changing_b_dotw} and Proposition \ref{prop:representability}.
\end{proof}

\subsection{Representations $R_T^G(\theta)$ and $R_{T_h}^{G_h}(\theta)$}\label{sec:representations_prelim}

Let a basic $b$ and a lift $\dot w_1$ be as in Section \ref{sec:basic_notation} with $\kappa_{\bfGL_n}(b) = \kappa_{\bfGL_n}(\dot w_1) = \kappa$ be fixed. 
In Section \ref{sec:basic_notation} we attached to $b,\dot w_1$ the locally pro-finite groups $G,T$ and their maximal compact subgroups $G_\cO,T_\cO$. In \cite[7.2]{CI_ADLV} we defined families (indexed by $h \geq 1$) of perfectly finitely presented perfect group schemes over $\FF_q$, with $\FF_q$-points $G_h$, $T_h$ such that $G_\cO = \prolim_h G_h$ and $T_\cO = \prolim_h T_h$, and showed that $G\times T$-equivariantly,
\[
X_{\dot w_1}^{DL}(b) \cong \coprod_{G/G_\cO}g.X_\cO, \quad \text{ with } \quad X_\cO \cong \prolim_h X_h 
\] 
such that $X_\cO$ is acted on by $G_\cO \times T_\cO$\footnote{Note that $T$ is generated by $T_\cO$ and a central element of $G$, when $G,T$ are both regarded as subgroups of $\bfGL_n(\breve K)$, so that $\coprod_{G/G_\cO}g.X_\cO$ admits also a natural right $T$-action.}, each $X_h$ is a perfectly finitely presented perfect $\overline{\FF}_q$-scheme acted on by $G_h \times T_h$, and all morphisms are compatible with all actions. Moreover, $X_h$ is the perfection of a smooth affine $\overline{\FF}_q$-scheme of finite type.
We identify $X_{\dot w_1}^{DL}(b)$ with $\coprod_{G/G_\cO}g.X_\cO$ via this isomorphism. The groups $G_h$ and $T_h$ are certain Moy--Prasad quotients of $G_\cO$ and $T_\cO$, and hence essentially independent of the choice of $b$, ${\bf x}_b$ and $\dot w_1$. An explicit presentation of $G_h,T_h,X_h$ is reviewed in Section \ref{sec:Step1} below. 

We review the definition of certain \'etale cohomology groups with compact support of $X_{\dot w_1}^{DL}(b)$ and $X_\cO$ (which are not perfectly finitely presented over $\overline\FF_q$). First, for $h \geq 1$ and a character $\chi \colon T_h \rightarrow \overline{\bQ}_\ell^\times$, the $\chi$-isotypic components $H_c^i(X_h)_\chi$ of the $\ell$-adic cohomology groups with compact support are defined\footnote{Recall from Section \ref{sec:notation} that we omit the constant coefficients $\overline{\bQ}_\ell$ from the notation.}, as $X_h$ is the perfection of smooth scheme of finite type over $\overline{\FF}_q$. Second, for $h\geq 1$, the fibers of $X_h/{\rm ker}(T_h \rightarrow T_{h-1}) \rightarrow X_{h-1}$ are isomorphic to $\bA^{n-1}$ \cite[Proposition 7.6]{CI_ADLV}. Let $\chi \colon T_\cO \rightarrow \overline{\bQ}_\ell^\times$ be a smooth character. Then there exists an $h\geq 1$, such that $\chi$ is trivial on ${\rm ker}(T_\cO \rightarrow T_h)$ for some $h \geq 1$. Let $h' \geq h$ and denote the characters induced by $\chi$ on $T_h$ and $T_{h'}$ again by $\chi$. Then $H_c^\ast(X_h)_\chi = H_c^\ast(X_{h'})_\chi$, where $H_c^\ast$ is the alternating sum of the cohomology. Thus we can define $H_c^\ast(X_\cO)_\chi$ as $H_c^\ast(X_{h'})_\chi$ for any $h' \geq h$ and this is independent of $h'$\footnote{Note that the single cohomology groups $H_c^i(X_\cO)_\chi$ are not defined, due to a degree shift: $H_c^i(X_{h'})_\chi = H_c^{i - 2d}(X_h)_\chi$ for an appropriate $d \geq 0$. One can remedy this by introducing homology groups $H_i(Y) := H_c^{2{\rm dim}(Y) - i}(Y)({\rm dim}(Y))$ as in \cite{Lusztig_79}, which removes precisely this shift in degree.}.
So, if $\chi$ is a character of $T_\cO$ of level $h$, we have the $G_h$-representation
\[
R_{T_h}^{G_h}(\chi) := H_c^\ast(X_\cO)_\chi = H_c^\ast(X_h)_\chi.
\]
and we denote the $G_\cO$-representation obtained by inflation via $G_\cO \twoheadrightarrow G_h$ again by $R_{T_h}^{G_h}(\chi)$.

Let $Z \subseteq G$ be the center and let $\widetilde X_\cO := \bigcup_{g \in ZG_\cO} g.X_\cO$ be the union in $X_{\dot w_1}^{DL}(b)$ of all $ZG_\cO$-translates of $X_\cO$. Then $\widetilde X_\cO$ is acted on by $ZG_\cO \times T$ and is a disjoint union of copies of $X_\cO$. Exactly as above for $X_\cO$, for a smooth character $\theta \colon T \rightarrow \overline{\bQ}_\ell^\times$ we may define the smooth $ZG_\cO$-representation $H_c^\ast(\widetilde X_\cO)_\theta$.

\begin{lm}\label{lm:central_extension_RThGhtheta}
Let $\theta \colon T \rightarrow \overline{\bQ}_\ell^\times$ be a smooth character of level $h$. As $G_\cO$-representations, $H_c^\ast(\widetilde X_\cO)_\theta \cong R_{T_h}^{G_h}(\theta)$. As a $ZG_\cO$-representation, $H_c^\ast(\widetilde X_\cO)_\theta$ is just the $G_\cO$-representation $R_{T_h}^{G_h}(\theta)$, with action extended to $Z$ by letting $\varpi \in Z \cong K^\times$ act by the scalar $\theta(\varpi)$. 
\end{lm}
\begin{proof} This is immediate (see e.g.\ \cite[Lemma 4.5]{Ivanov_15_ADLV_GL2_unram}).
\end{proof}

Justified by this lemma we write $R_{T_h}^{G_h}(\theta)$ for the $ZG_\cO$-representation $H_c^\ast(\widetilde X_\cO)_\theta$. For schemes $Y_i$ such that $H_c^\ast(Y_i)$ are defined, put $H_c^\ast(\coprod_{i\in I} Y_i) := \bigoplus_{i \in I}H_c^\ast(Y_i)$. We get our main object of study, the smooth $G$-representation
\[ 
R_T^G(\theta) := H_c^\ast(X_{\dot w_1}^{DL}(b))_\theta = \cind_{ZG_{\cO}}^G R_{T_h}^{G_h}(\theta)
\]
(cf. \cite[Theorem 11.2]{CI_ADLV}).

\begin{rem}
By construction and by Lemma \ref{lm:changing_b_dotw}, the isomorphism class of the $G$-representation $R_T^G(\theta)$ is independent of the choices of representatives $b,\dot w_1$. A similar independence holds for the $ZG_\cO$-representation $R_{T_h}^{G_h}(\theta)$.  
\end{rem}

\subsection{Norms and characters}\label{sec:norms_and_chars}
The following definitions do not depend on the choice of an isomorphism $T \cong L^\times$ (as in Section \ref{sec:form_of_torus}).
\begin{Def}\label{level}
We say that a smooth character $\theta \colon T \cong L^\times \rightarrow \overline{\QQ}_\ell^\times$ is of \emph{level $h$} if it is trivial on $\ker(T_\cO \rightarrow T_h) \cong U_L^h$, but non-trivial on $\ker(T_\cO \rightarrow T_{h-1})\cong U_L^{h-1}$.
\end{Def}
Recall the subextensions $L \supseteq K_r \supseteq K$ (Section \ref{sec:notation}). Whenever $r,s$ are positive divisors of $n$ such that $s$ divides $r$, we denote by $N_{r/s} \colon K_r^\times \rightarrow K_s^\times$ the norm map for the field extension $K_r/K_s$. For any $h\geq h'\geq 1$, it induces maps 
\[
U_{K_r}/U_{K_r}^h \rightarrow U_{K_s}/U_{K_s}^h \quad \text{ and } \quad U_{K_r}^{h'}/U_{K_r}^h \rightarrow U_{K_s}^{h'}/U_{K_s}^h
\]
which are surjective (see e.g.\ \cite[Chap. V,\S2]{Serre_Local_fields}), and which we again denote by $\N_{r/s}$. 

\begin{Def}
\begin{itemize}
\item[(i)] A character $\theta \colon T \cong L^\times \rightarrow \overline{\bQ}_\ell^\times$ resp.\ $\theta \colon T_\cO \cong U_L \rightarrow \overline{\bQ}_\ell^\times$ is \emph{in general position}, if the stabilizer of $\theta$ in $\Gal(L/K)$ is trivial. We say $\theta|_{U_L^1}$ is in \emph{general position}, if the stabilizer of $\theta|_{U_L^1}$ in $\Gal(L/K)$ is trivial. 
\item[(ii)] Let $h \geq 1$. A character $\theta \colon T_h \cong U_L/U_L^h \rightarrow \overline{\bQ}_\ell^\times$ (resp.\ $\theta|_{T_h^1 = U_h^1/U_L^h}$) is \emph{in general position} if its inflation to $T_\cO$ (resp.\ to ${\rm ker}(T_\cO \rightarrow T_1)$) is in general position. 
\end{itemize}
\end{Def}

Note that $\theta \colon T \cong L^\times \rightarrow \overline{\bQ}_\ell^\times$ is in general position if and only if $\theta|_{T_\cO}$ is. 

\begin{lm}\label{lm:factor_through_norm_equiv_cond}
Let $\theta \colon T \cong L^\times \rightarrow \overline{\bQ}_\ell^\times$ be a character. Let $s\in \bZ$. Then \[
\theta \circ \sigma^s = \theta \quad \Leftrightarrow \quad \text{$\theta$ factors through $N_{L/K_{{\rm gcd}(n,s)}}$.}\]
The analogous claim holds for $\theta|_{U_L^1}$. In particular, $\theta$ is in general position if and only if $\theta$ does not factor through any of the maps $\N_{n/r}$ with $r < n$, and $\theta|_{U_L^1}$ is in general position if and only if $\theta|_{U_L^1}$ does not factor through any of the maps $\N_{n/r}$ with $r < n$.
\end{lm}
\begin{proof}
$\theta \circ \sigma^s = \theta$ is equivalent to $\theta$ being trivial on the image of the map $L^\times \rightarrow L^\times$, $x \mapsto x^{-1}\sigma^s(x)$. By Hilbert's Theorem 90, this image is equal to the kernel of the norm map of $L$ over the field stable by $\sigma^s$, which is $K_{{\rm gcd}(n,s)}$.
\end{proof}

\section{A Mackey formula}\label{sec:Mackey}

In this section we prove the following Mackey-type formula for the representations $R_{T_h}^{G_h}(\theta)$. 

\begin{thm}\label{thm:general_irred} Let $\theta,\theta' \colon T_h \rightarrow \overline{\bQ}_\ell^\times$ be two characters. Then 
\[
\left\langle R_{T_h}^{G_h}(\theta), R_{T_h}^{G_h}(\theta') \right\rangle_{G_h} = \# \left\{ w \in W_\cO^F \colon \theta'  = \theta \circ {\rm ad}(w) \right\}.
\]
\end{thm}

\begin{rem} The theorem shows that in the setting considered in this paper and in \cite{CI_ADLV},  the assumption in \cite[Corollary 2.4(b)]{Lusztig_04} resp.\ \cite[Corollary 4.7(ii)]{CI_MPDL} that $\theta$ is regular is obsolete. 
We also note that because part of this proof requires an explicit computation using our choice of Coxeter element $w_1$, Theorem \ref{thm:general_irred} does not allow us to conclude the analogue of the independence-of-choice statements \cite[Corollary 2.4(a)]{Lusztig_04}, \cite[Corollary 4.7(i)]{CI_MPDL}.
\end{rem}
\begin{cor}\label{cor:irreducibility}
Let $\theta \colon T_h \rightarrow \overline{\QQ}_\ell^{\times}$ be a character, whose stabilizer in $\Gal(L/K)[n']$, the unique subgroup of $\Gal(L/K)$ of order $n'$, is trivial.
Then $\pm R_{T_h}^{G_h}(\theta)$ is irreducible $G_h$-representation. In particular, $\Fr_{q^n}$ acts in $\pm R_{T_h}^{G_h}(\theta)$ by multiplication with a scalar. Moreover, the map
\begin{align*}
\big\{\text{characters $\theta \colon T_h \rightarrow \overline\bQ_\ell^\times$ in general position}  \big\}\slash W_\cO^F &\rightarrow \big\{\text{irreducible $G_h$-representations}\big\}  \\
\theta &\mapsto \pm R_{T_h}^{G_h}(\theta)
\end{align*}
is injective.
\end{cor}

\begin{proof}
This follows from the description of $W_\cO^F$ in Section \ref{sec:case_bequalw} and Theorem \ref{thm:general_irred}.
\end{proof}

In course of the proof of Theorem \ref{thm:general_irred} we will make use of the following well-known fact.

\begin{prop}\label{prop:torus_invariants}
Let $X$ be (the perfection of a) quasi-projective scheme over a finite field, $H$ is a torus, and $\alpha \colon X \rightarrow X$ is a finite order automorphism commuting with the $H$-action, then $\tr(\alpha,H_c^\ast(X,\overline{\QQ}_\ell)) = \tr(\alpha,H_c^\ast(X^H,\overline{\QQ}_\ell))$. 
\end{prop}

\begin{proof} See, for example, \cite[10.15]{DigneM_91}.
\end{proof}

\begin{rem}\label{rem:two_principles}
We prove Theorem \ref{thm:general_irred} in \S\ref{sec:Step1}-\ref{sec:Step4} below. To guide the reader, let us explain the main steps of the proof. The overarching idea follows the proof of the analogous result in classical Deligne--Lusztig theory \cite[Theorem~6.8]{DeligneL_76}, but the implementation is much more intricate in our setting. The essential source of this difficulty is in constructing a connected algebraic torus which acts on the schemes $\widehat \Sigma_w$ for $w \in W_\cO$. In the classical setting of $h = 1$, there is no issue whatsoever, but for $h>1$, we are only able to construct such an action for \textit{certain} w ((3) and (4)(b)(c) below). In our present setting we get lucky: it turns out that for all the $w \in W_\cO$ for which we cannot find such an action, the scheme $\widehat \Sigma_w$ is empty ((4)(a))!
\begin{itemize}
\item[(1)] Construct a scheme $\Sigma$ equipped with a $T_h \times T_h$-action, such that $X_h \times X_h/G_h$ is (up to an affine space) $T_h \times T_h$-equivariantly isomorphic to $\Sigma$. We then may rewrite the left hand side of Theorem \ref{thm:general_irred} as $\dim_{\overline \bQ_\ell}H_c^\ast(\Sigma)_{\theta^{-1},\theta'}$.
\item[(2)] Let $\bG_h$ be the Moy-Prasad quotient of $L^+\bfG_\cO$, such that $\bG_h(\bF_q) = G_h$. Then $\bG_1$ can be identified with the reductive quotient of $\bG_h$, and pulling back the Bruhat decomposition of $\bG_1$, we get a ``Bruhat cell'' $\bG_{h,w} \subseteq \bG_h$ for any $w \in W_\cO$. There is a natural map $\Sigma \rightarrow \bG_h$. For $w \in W_\cO$, let $\Sigma_w$ be the preimage of $\bG_{h,w}$, and then replacing $\Sigma_w$ by an appropriate vector bundle $\widehat\Sigma_w \rightarrow \Sigma_w$ (having the same cohomology and still carrying a $T_h \times T_h$-action), one is reduced to showing that 
\begin{equation}\label{eq:coh_of_Sigma_w_survey} \dim_{\overline{\bQ}_\ell} H_c^\ast(\widehat\Sigma_w)_{\theta^{-1},\theta'} = \begin{cases}1 & \text{if $w \in W_\cO^F$ and $\theta'  = \theta \circ {\rm ad}(w)$} \\ 0 & \text{otherwise.} \end{cases} 
\end{equation}
\item[(3)] Extend the $T_h \times T_h$-action on $\widehat\Sigma_w$ to an action of a certain commutative group scheme $H_w$ (cf. beginning of \S\ref{sec:Step3}).
\item[(4)] Show that for $w \in W$ one of the following cases appears:
\begin{itemize}
\item[(a)] $\widehat \Sigma_w = \varnothing$ (cf. \S\ref{sec:Step2}), or
\item[(b)] the connected reductive part $H_{w, \rm red}^\circ$ (which is a torus) of $H_w$ is big enough so that the fixed point locus $\widehat \Sigma_w^{H_{w, \rm red}^\circ}$ is finite (cf. \S\ref{sec:Step3}). Then it is easy to deduce \eqref{eq:coh_of_Sigma_w_survey} for $w$ from Proposition \ref{prop:torus_invariants}. Or,
\item[(c)] $w=1$, in which case we have to develop an essentially different variant of the extension-of-action principle (see step (3)), using another vector bundle $\widetilde \Sigma_w \rightarrow \Sigma_w$ (cf. \S\ref{sec:Step4}). Here, finally, one again concludes using Proposition \ref{prop:torus_invariants}. 
\end{itemize}
\end{itemize}
We note moreover that this method is actually quite different from Lusztig's method in establishing the Mackey-type formula for regular $\theta$ \cite[Proposition 2.3]{Lusztig_04} (see also \cite[Theorem 1.1]{CI_MPDL}).

% After general preparations in Section \ref{sec:Step1}, we show in Section \ref{sec:Step2} that several of the perfect schemes $\widehat\Sigma_w$ (as in \cite[1.9]{Lusztig_04}) are empty in our case; then in Sections \ref{sec:Step3},\ref{sec:Step4} we generalize Lusztig's argument from \cite[1.9, proof of claim (b)]{Lusztig_04} with the extension of action on $\widehat\Sigma_w^{''}$ in two different ways to cover the remaining $\widehat\Sigma_w$. The first generalization uses our concrete situation, whereas the second is quite general.
\end{rem}

% \begin{rem} 
% Theorem \ref{thm:general_irred} was generalized recently to the case of general reductive groups over $\cO_K$ and arbitrary Coxeter elements.
% \end{rem}

\subsection{General preparations}\label{sec:Step1} 
In contrast to \cite{CI_ADLV} where we worked with Coxeter-type and special representatives for $[b]$ (see \cite[\S5.1]{CI_ADLV}), here it is most convenient to work with a third type of representatives. We put 
\begin{equation}\label{eq:new_Cox_rep}
b = \dot w_1 = b_0 t_{\kappa,n} \in {\bfGL}_n(\breve K)
\end{equation}
where 
\[
b_0 \colonequals \left(\begin{matrix} 0 & 1 \\ 1_{n-1} & 0 \end{matrix}\right), \qquad \text{and} \qquad t_{\kappa,n} \colonequals \begin{cases}
\diag(\underbrace{1, \ldots, 1}_{n-\kappa},\underbrace{\varpi, \ldots, \varpi}_{\kappa}) & \text{if $(\kappa,n) = 1$,} \\
\diag(\underbrace{t_{k_0,n_0}, \ldots, t_{k_0,n_0}}_{n'}) & \text{otherwise.}
\end{cases}
\]
are as in  \cite[\S5.2.1]{CI_ADLV}. In particular, we work in the setup of Section \ref{sec:case_bequalw}.

Recall the (unique) fixed point ${\bf x}_b$ of $F$ in the apartment $\mathscr{A}_{\bfT,\breve K}$ of $\bfT$ in $\mathscr{B}_{\breve K}$, and the corresponding maximal parahoric $\cO_K$-model $\bfG_\cO$ of $\bfG$.
We have the stabilizer $\breve G_{{\bf x}_b,0} = {\bf G}_\cO(\cO_{\breve K})$ of ${\bf x}_b$ in ${\bf G}(\breve K) = \bfGL_n(\breve K)$ and its Moy--Prasad filtration \cite{MoyP_94} given by subgroups $\breve G_{{\bf x}_b,r}$ $(r\geq 0)$. Similarly as in \cite[\S5.3]{CI_ADLV}, consider the affine perfect group scheme $\bG$ over $\FF_q$ defined by 
\begin{equation*}
\bG(\overline \FF_q) = \breve G_{{\bf x}_b, 0}, \qquad \bG(\FF_q) = \breve G_{{\bf x}_b,0}^F = G_\cO.
\end{equation*}
and for $h \in \bZ_{\geq 1}$, the affine perfectly finitely presented perfect group scheme $\bG_h$ over $\FF_q$ such that
\begin{equation*}
\bG_h(\overline \FF_q) = \breve G_{{\bf x}_b,0}/\breve G_{{\bf x}_b,(h-1)+}, \qquad G_h := \bG_h(\FF_q) = \breve G_{{\bf x}_b,0}^F/\breve G_{{\bf x}_b,(h-1)+}^F.
\end{equation*}
We denote the Frobenii on $\bG,\bG_h$ again by $F$. The groups $\bG, \bG_h$ possess an explicit description in terms of matrices similar to \cite[\S5.3]{CI_ADLV}.

\begin{rem}\label{rem:rel_to_Cox_reps}
In \cite[Section 7]{CI_ADLV}, we worked instead with the Coxeter representatives $b' = b_0^{e_{\kappa,n}}t_{\kappa,n}$ as in \cite[\S5.2.1]{CI_ADLV}; but if $\gamma$ is as in \cite[\S7.6]{CI_ADLV}, then $b = \gamma b' \gamma^{-1}$, i.e., $b$ is integrally $\sigma$-conjugate to $b'$.
In fact, the groups $\bG, \bG_h$ used here are equal to $\gamma \bG \gamma^{-1}, \gamma \bG_h \gamma^{-1}$ with the latter $\bG, \bG_h$ as in \cite{CI_ADLV}.  
\end{rem}

As (perfect) $\FF_q$-groups, $\bG_1 \cong {\rm Res}_{\FF_{q^{n_0}}/\FF_q} \GL_{n'}$. The above-mentioned description identifies $\bG_1$ with a closed $\FF_q$-subgroup of $\GL_{n,\FF_q}$. In fact, $\bG_{1,\overline{\FF}_q}$ is the closed subgroup of $\GL_{n,\overline{\FF}_q}$ consisting of those $n \times n$-matrices $g = (g_{ij})_{i,j \in \bZ/n\bZ} \in \GL_{n,\overline{\FF}_q}$ for which $X_{ij} = 0$, unless $i \equiv j \mod n_0$; if we now equip $\GL_{n,\overline{\FF}_q}$ with the $\FF_q$-structure given by the Frobenius $F_0 \colon g \mapsto b_0 \sigma(g) b_0^{-1}$ and denote the resulting $\FF_q$-group simply by $\GL_n$, then this defines an $\FF_q$-embedding $\bG_1 \rightarrow \GL_n$.

We regard the symmetric group on $n$ letters $S_n$ as the group of set automorphisms of $\bZ/n\bZ$, and for an element $i \in \bZ/n\bZ$ let $[i]$ be the unique integer between $1$ and $n$ having residue $i$ modulo $n$. We also identify $S_n$ with the Weyl group of the diagonal torus in $\GL_n$ (either over $\overline{\FF}_q$ or $\breve K$) by sending a permutation $v \in S_n$ to the permutation matrix (again denoted $v$) whose non-zero entries are $(v(i),i)$ for $1\leq i\leq n$. 

As $\bG_1$ is naturally isomorphic to the reductive quotient of the special fiber of ${\bf G}_\cO$, the group $W_\cO$ is simply the Weyl group of $\bT_1$ in $\bG_1$. Thus, using the above identifications, $W_\cO$ is the subgroup of $S_n$, isomorphic to $S_{n'} \times \dots \times S_{n'}$ ($n_0$ times), of those permutations which preserve the residue modulo $n_0$.

Applying $L_h^+$ to the inclusions  $\bfT_\cO, \bfU_\cO, \bfU_\cO^- \subseteq \bfG_\cO$ gives closed subgroups $\bT_h, \bU_h, \bU_h^- \subseteq \bG_h$, with $\bT_h$ defined over $\FF_q$ and $\bU_h,\bU_h^-$ defined over $\FF_{q^n}$ (cf. \cite[2.6]{CI_MPDL}). For a closed subgroup $\mathbb{H}_h \subseteq \bG_h$ and $1\leq a\leq h-1$, we write $\mathbb{H}_h^a := \mathbb{H}_h \cap {\rm ker}(\bG_h \rightarrow \bG_a)$. If $\mathbb{H}_h$ is defined over $\FF_q$, we write $H := \mathbb{H}(\FF_q)$ and $H_h^a := \mathbb{H}_h^a(\FF_q)$.

Then we have (by a slight modification -- or conjugation with $\gamma$ from Remark \ref{rem:rel_to_Cox_reps} -- of \cite[Section 7]{CI_ADLV}, in particular, Propositions 7.10,7.11) as perfect $\overline{\FF}_q$-spaces 
\begin{equation}\label{eq:Xh_into_GGh}
X_h \cong \{g \in \bG_h \colon g^{-1}F(g) \in \bU_h^- \cap F\bU_h \} \cong S_h/(\bU_h \cap F\bU_h),
\end{equation}
where 
\[
S_h = \{g \in \bG_h \colon g^{-1}F(g) \in F\bU_h \}\footnote{note that $X_h, S_h$ are indeed perfect schemes as the tensor product of perfect rings over a perfect ring is again perfect (by \cite[3.16]{BhattS_17})},
\]
and the action of $\bU_h \cap F\bU_h$ on $S_h$ is by right multiplication (here and in the following: all presheaves have to be sheafified). Moreover, \eqref{eq:Xh_into_GGh} is $G_h \times T_h$-equivariant with respect to the $G_h \times T_h$-action on the right hand side given by $(g',t) \colon g \mapsto g'gt$.

The fibers of the projection $S_h \rightarrow X_h$ are isomorphic to affine spaces of fixed dimension, so that $R_{T_h}^{G_h}(\theta) = H_c^\ast(S_h)_\theta$. As in \cite[1.9]{Lusztig_04}, if 
\begin{align*}
\Sigma = \{(x,x',y) \in F\bU_h \times F\bU_h \times \bG_h \colon xF(y) = yx' \}
\end{align*}
with the $T_h \times T_h$-action given by $(t,t') \colon (x,x',y) \mapsto (txt^{-1},t'x't^{\prime -1},tyt^{\prime -1})$, then the map 
\begin{equation}\label{eq:XhXh_Sigma}
G_h \backslash (S_h \times S_h) \rightarrow \Sigma,
\end{equation}
induced by $(g,g') \mapsto (g^{-1}F(g), g^{\prime -1}F(g'), g^{-1}g')$ is an $T_h \times T_h$-equivariant isomorphism (the quotient of the left side is taken with respect to the diagonal action).

The group $\bG_1$ is reductive and $\ker(\bG_h \rightarrow \bG_1)$ is unipotent. Thus the Bruhat decomposition $\bG_1 = \coprod_{w \in W_\cO} \bU_1 \bT_1 \dot w \bU_1$ of $\bG_1$ lifts to a decomposition $\bG_h = \coprod_{w \in W_\cO} \bG_{h,w}$, with $\bG_{h,w} = \bU_h \bT_h \dot w \mathbb{K}_h^1 \bU_h$, $\mathbb{K}_h^1 = (\bU_h^-)^1 \cap w^{-1}(\bU_h^-)^1 w$ \cite[Lemma 8.6]{CI_ADLV}. We then have the locally closed decomposition $\Sigma = \coprod_{w \in W_\cO} \Sigma_w$, where
\begin{align*}
\Sigma_w = \{(x,x',y) \in F\bU_h \times F\bU_h \times \bG_{h,w} \colon xF(y) = yx' \}.
\end{align*}
is $T_h \times T_h$-stable. Further, let 
\[
\widehat\Sigma_w = \{ (x,x',y_1,\tau,z,y_2) \in F\bU_h \times F\bU_h \times \bU_h \times \bT_h \times \mathbb{K}_h^1 \times \bU_h \colon xF(y_1 \tau \dot w z y_2) = y_1 \tau \dot w z y_2x' \}. 
\] 
where $\dot w \in \bG_h$ is an (arbitrary but from now on fixed) lift of $w$. It has a $T_h \times T_h$-action by
\begin{equation}\label{eq:ThTh_action_widehatSigma}
(t,t') \colon (x,x',y_1,\tau,z,y_2) \mapsto (txt^{-1},t'x't^{\prime -1},ty_1t^{-1},t \tau \dot wt^{\prime -1} \dot w^{-1},t'zt^{\prime -1},t'y_2t^{\prime -1}).
\end{equation}
Then the map $\widehat\Sigma_w \rightarrow \Sigma_w$ given by $(x,x',y_1,\tau,z,y_2) \mapsto (x,x',y_1\tau zy_2)$ is a $T_h \times T_h$-equivariant Zariski-locally trivial fibration. All in all, as in \cite{Lusztig_04}, using \eqref{eq:XhXh_Sigma} it is enough to show that
\begin{equation}\label{eq:coh_of_Sigma_w} \sum_i (-1)^i \dim_{\overline{\bQ}_\ell} H_c^i(\widehat\Sigma_w)_{\theta^{-1},\theta'} = \begin{cases}1 & \text{if $w \in W_\cO^F$ and $\theta'  = \theta \circ {\rm ad}(w)$} \\ 0 & \text{otherwise.} \end{cases} 
\end{equation}

So far we were essentially following \cite[1.9]{Lusztig_04}, but now we have to deviate.

\subsection{Emptyness of certain $\widehat\Sigma_w$}\label{sec:Step2} Let $w \in W_\cO$. As in \cite[1.9]{Lusztig_04}, make the change the variables $xF(y_1) \mapsto x$, $x'F(y_2)^{-1} \mapsto x'$. We thus may rewrite
\begin{equation}\label{eq:some_presentation_of_Sigma_w}
\widehat\Sigma_w = \{ (x,y_1,\tau,z,y_2) \in F\bU_h \times \bU_h \times \bT_h \times \mathbb{K}_h^1 \times \bU_h \colon xF(\tau \dot w z) \in y_1 \tau \dot w z y_2 F\bU_h \} 
\end{equation}
with the $T_h \times T_h$-action still given by \eqref{eq:ThTh_action_widehatSigma}.

\begin{lm}\label{lm:empty_Sigma_w}
Assume that there exists some $2 \leq i \leq n$ such that $[w(i)] > [w(i-1) + 1] > 1$. Then $\widehat\Sigma_w = \varnothing$. 
\end{lm}
\begin{proof}
We may assume $h=1$, and hence we may ignore $z \in \mathbb{K}_h^1$ whose image in $\bG_1$ is $1$. We use the identification of $\bG_1$ with the closed subgroup of $\GL_n$ from Section \ref{sec:Step1}. Write $y_i = y_{i,1}y_{i,2}$ with $y_{1,1}, y_{2,2} \in \bU_1 \cap F\bU_1$ and $y_{1,2},y_{2,1} \in \bU_1 \cap F\bU_1^-$. Replacing $x$ by $y_{1,1}^{-1}x$ and putting $y_{2,2}$ into the $F\bU_1$ on the right hand side, we are reduced to show that there are no $(x,y_{1,2},y_{2,1}, \tau) \in F\bU_1 \times (\bU_1 \cap F\bU_1^-) \times (\bU_1 \cap F\bU_1^-)  \times \bT_1$ with 
\[ 
\dot w^{-1} \tau y_{1,2}^{-1} x F(\tau \dot w) \in y_{2,1}F(\bU_1).
\] 
Replacing everything by appropriate conjugates resp.\ inverses, it suffices to show that there are no $(x, y, y_{2,1}, \tau') \in F\bU_1 \times (\bU_1 \cap F\bU_1^-) \times (\bU_1 \cap F\bU_1^-)  \times \bT_1$ satisfying 
\[ 
\dot w^{-1} y x F(\dot w) \in \tau' y_{2,1}F\bU_1.
\] 

For a $n\times n$-matrix $X$, let $X_{i,j}$ denote its $(i,j)$th entry. Consider the closed subset 
\[
M = \{ X \in \bG_1 \colon X_{i,i } \in \bG_m \, \forall \, 2\leq i\leq n \text{ and } X_{i,j} = 0 \, \forall \, n \geq i>j>1 \}
\]
of $\bG_1$. We have 
\[
\bU_1 \cap F\bU_1^- = \{X \in \bG_1 \colon X_{i,i} = 1 \,\forall  i \text{ and } X_{i,j} = 0 \,\forall\, (i,j) \text{ with } j \neq 1 \text{ or } i\neq  j\}
\]
One easily checks that $\bT_1 \cdot (\bU_1 \cap F\bU_1^-) \cdot F\bU_1 \subseteq M$. Thus it suffices to check that 
\[
\dot w^{-1} M F(\dot w) \cap M = \varnothing.
\]
For $X \in \bG_1$ (and even more generally for $X \in \GL_n$ and $F$ replaced by $F_0$ as in Section \ref{sec:Step1}), one has the formula
\begin{equation}\label{eq:entries_of_F_conjugate}
(\dot w^{-1} X F(\dot w))_{i,j} = X_{w(i), [w(j-1) + 1]}.
\end{equation}
Let $2 \leq i \leq n$ be such that $[w(i)] > [w(i-1) + 1] > 1$. Then for $X \in M$, the $(i,i)$th diagonal entry of $\dot w^{-1} X F(\dot w)$ is 
\[
(\dot w^{-1} X F(\dot w))_{i,i} = X_{w(i), [w(i-1) + 1]} = 0,
\]
by definition of $M$.  This shows that $X\not\in M$ and we are done.
\end{proof}

As mentioned in Section \ref{sec:case_bequalw}, $W_\cO^F = \langle w_1^{n_0}\rangle$. Clearly, no element from $W_\cO^F$ satisfies the condition in Lemma \ref{lm:empty_Sigma_w}. Thus Lemma \ref{lm:empty_Sigma_w} implies \eqref{eq:coh_of_Sigma_w} for all $w$ satisfying the condition in the lemma.

\subsection{An extension of action}\label{sec:Step3} It remains to show \eqref{eq:coh_of_Sigma_w} for all $w \in W_\cO \subseteq S_n$ for which there is no $2 \leq i \leq n$ satisfying $[w(i)] > [w(i-1) + 1] > 1$. Consider the closed subgroup 
\[ 
H_w = \{ (t,t') \in \bT_h \times \bT_h \colon \dot w^{-1} t^{-1} F(t) \dot w = t^{\prime -1} F(t') \text{ centralizes $\mathbb{K}_h = \bU_h^- \cap \dot w^{-1}\bU_h^- \dot w$ } \}
\]
of $\bT_h \times \bT_h$. It contains $T_h \times T_h$. It is easy to check that the action of $T_h \times T_h$ on $\Sigma_w$ extends to an action of $H_w$ given by the formula
\[
(t,t') \colon (x,y_1,\tau,z,y_2) \mapsto (F(t)xF(t)^{-1}, F(t)y_1F(t)^{-1}, t \tau \dot wt^{\prime -1} \dot w^{-1}, t'zt^{\prime -1}, F(t')y_2F(t')^{-1}).
\]

\begin{lm}\label{lm:big_centralizer}
Let $1 \neq w \in W_\cO$. Assume that there is no $2\leq i\leq n$ with $[w(i)] > [w(i-1) + 1] > 1$. Then there is a proper Levi subgroup $L$ of ${\bf G}_{\breve K}$ containing ${\bf T}_{\breve K}$ such that if $\mathbb{L}_h$ denotes the corresponding subgroup of $\bG_h$, then $\mathbb{K}_h \subseteq \mathbb{L}_h$.
\end{lm}

\begin{proof}
First we prove the following claim: there is an $s\in \bZ_{\geq 1}$ and a sequence $0 =: i_0 < 1 \leq i_1 < \dots < i_{s-1} < i_s := n$ of integers such that for each $1\leq j \leq s$, and for each $i_{j-1} + 1 \leq i\leq i_j$ (if $j>1$) resp.\ for each $1\leq i \leq i_1$ (if $j=1$), one has $w(i) = n - i_{j-1} - (i_j - i)$. Indeed, find the $1 \leq i_1 \leq n$ such that $w(i_1) = n$. It follows from the condition on $w$ that $w(i_1 - 1) = n-1$, $...$, $w(1) = n-(i_1-1)$. The maximal value which $w$ has on $\{i_1+ 1,\dots, n\}$ is $n - i_1$. Find the $i_1 + 1 \leq i_2 \leq n$ such that $w(i_2) = n-i_1$. It follows from the condition on $w$ that $w(i_2 - 1) = n-i_1 - 1$, $\dots$, $w(i_1+1) = n-i_2 + 1$. Then, proceed inductively until $i_s = n$ is reached. The claim is proven.

Note that $i_1 < n$, as $i_1 = n$ would imply $w = 1$, whereas $w \neq 1$ is assumed in the lemma. Let $L$ be the (proper) Levi subgroup of $\bfGL_{n,\breve K} = \bfG_{\breve K}$ containing ${\bf T}_{\breve K}$ of type $(i_1,i_2 - i_1, \dots, i_s - i_{s-1})$. From the claim it easily follows that $\mathbb{K}_h = \bU_h^- \cap w^{-1}\bU_h^-w \subseteq \mathbb{L}_h$.
% of block-diagonal matrices with blocks of size $i_1 \times i_1$, $(i_2-i_1) \times (i_2 - i_1)$, $\dots$, $(i_s - i_{s-1}) \times (i_s - i_{s-1})$. 
\end{proof}

For $i=1,2$ we have the composed maps 
\[ 
\pi_i \colon H_w \subseteq \bT_h \times \bT_h \rightarrow \bT_h \rightarrow \bT_1,
\] 
where the middle map is the projection to the $i$-th component, and the last map is the natural projection.
% $H_w^\circ$ is the connected component of $H_w$
For $1\leq i \neq j \leq n$, let $\alpha_{i,j}$ denote the root of $\GL_{n,\FF_q}$ corresponding to $(i,j)$th matrix entry. Recall from Section \ref{sec:Step1} that $\bT_1 \subseteq \bG_1 \subseteq \GL_{n,\FF_q}$ and that $\bT_1$ is the diagonal (and in fact elliptic with respect to the Frobenius $F_0$) torus of $\GL_{n,\FF_q}$. Let $\alpha_{i,j}$ be the roots of $\bT_1$ in $\GL_{n,\FF_q}$ corresponding to $(i,j)$th entry.

\begin{lm}\label{lm:regular_torus_in_general}
Let $\delta \colon \bZ/n\bZ \rightarrow \{0,1\}$ be a non-zero function, and let $\chi \colon \bG_m \rightarrow \bT_1$ be the cocharacter $X \mapsto \diag(X^{\delta(1)}, \dots, X^{\delta(n)})$. Then $S_{\chi} := \{t \in \bT_1 \colon t^{-1}F(t) \in {\rm im}(\chi) \}$ is a one-dimensional subgroup of $\bT_1$. Let $1 \leq j < i \leq n$. If $\delta$ does not factor as $\bZ/n\bZ \rightarrow \bZ/{\rm gcd}(n,i-j)\bZ \rightarrow \{0,1\}$, then the connected component $S_\chi^\circ$ of $S_\chi$ is not contained in the subtorus $\ker(\alpha_{i,j})$ of $\bT_1$.

In particular, if for any divisor $d>1$, $\delta$ does not factor as $\bZ/n\bZ \rightarrow \bZ/d\bZ \rightarrow \{0,1\}$, then $S_\chi^\circ$ is a not contained in any of the subtori $\ker(\alpha_{i,j})$ ($1\leq i\neq j\leq n$) of $\bT_1$.
\end{lm}

\begin{proof} Assume that $\delta$ does not factor through $\bZ/n\bZ \rightarrow \bZ/{\rm gcd}(n,i-j)\bZ$. As $\dim S_\chi = 1$, it suffices to show that $S_\chi \cap \ker(\alpha_{i,j})$ is finite. We write an element in $\bT_1$ as an $n$-tuple $(t_k)_{k=1}^n$ corresponding to the diagonal matrix with entries $t_1, \dots, t_n$. We have ${\rm im}(\chi) = \{ (a^{-\delta(k)})_{k=1}^n \in \bT_1 \colon a \in \bG_m\}$. Thus $(t_k)_{k=1}^n \in \bT_1$ lies in $S_\chi$ if and only if $t_1^{-1}t_n^q = a^{-\delta(1)}$, $t_2^{-1}t_1^q = a^{-\delta(2)}$, $\dots$, $t_n^{-1}t_{n-1}^q = a^{-\delta(n)}$. Thus $S_\chi$ is isomorphic to the one-dimensional subscheme of $\bG_m^2$,
\begin{equation}\label{eq:embedded_version_of_S}
\{t_1, a \in \bG_m^2 \colon t_1^{1 - q^n} = a^{\delta(1) + \sum_{k=2}^n q^{n-k+1} \delta(k)} \},
\end{equation}
which is embedded into $\bT_1$ by sending $(t_1, a)$ to the tuple $(t_k)_{k=1}^n$ with $t_k = t_1^{q^{k-1}}a^{\sum_{\lambda=2}^k q^{k-\lambda}\delta(\lambda)}$. Thus the intersection $S_\chi \cap \ker(\alpha_{i,j})$ is the closed subscheme of \eqref{eq:embedded_version_of_S} given by the equation $t_i = t_j$, i.e., 
\[
t_1^{q^{i-1} - q^{j-1}} = a^{\sum_{k=2}^j q^{j-k}\delta(k) - \sum_{k=2}^i q^{i-k}\delta(k)}
\]
Taking this to $(q^n - 1)$-th power, taking the equation in \eqref{eq:embedded_version_of_S} to the power $q^{i-1} - q^{j-1}$, and equalizing the left hand sides, we deduce that on $S_\chi \cap \ker(\alpha_{i,j})$ we must have
\[
a^{(q^{i-1} - q^{j-1})(\delta(1) + \sum_{k=2}^n q^{n-k+1} \delta(k))} = a^{(q^n - 1)(\sum_{k=2}^i q^{i-k}\delta(k) - \sum_{k=2}^j q^{j-k}\delta(k))}.
\]
Thus it suffices to show that 
\[
(q^{i-1} - q^{j-1})(\delta(1) + \sum_{k=2}^n q^{n-k+1} \delta(k)) \neq (q^n - 1)(\sum_{k=2}^i q^{i-k}\delta(k) - \sum_{k=2}^j q^{j-k}\delta(k)),
\]
or equivalently, that
\[
\sum_{k=i-1}^{n-1} q^k \delta(n-k+i) - \sum_{k=j-1}^{n-1}q^k \delta(n-k+j) \neq - \sum_{k=0}^{i-2} q^k \delta(i-k) + \sum_{k=0}^{j-2} q^k \delta(j-k),
\]
or that
\[
\sum_{k=0}^{n-1} q^k(\delta(i-k) - \delta(j-k)) \neq 0.
\]
Assume this is wrong, and this sum is $0$. All terms $\delta(i-k) - \delta(j-k)$ lie in the set $\{-1,0,1\}$ and hence $q^{n-1}$ is bigger than the sum of the absolute values of the remaining summands. It follows that we must have $\delta(i-n+1) - \delta(j-n+1) = 0$. Then we may continue in the same way with $q^{n-2}$ instead of $q^{n-1}$, etc. All in all we deduce that $\delta(i-k) = \delta(j-k)$ for all $k \in \bZ/n\bZ$. Or equivalently, that $\delta(k) = \delta(k + (i-j))$ for all $k \in \bZ/n\bZ$. But this is equivalent to saying that $\delta$ factors through $\bZ/n\bZ \rightarrow \bZ/{\rm gcd}(n,i-j)\bZ$, contradicting our assumption. 
\end{proof}

Now let $1 \neq w \in W_\cO$, such that there is no $2\leq i\leq n$ with $[w(i)] > [w(i-1) + 1] > 1$. Let $\mathbb{L}_h$ be as in Lemma \ref{lm:big_centralizer} and let $1 \leq i_1 < n$ be the size of its first block (cf. the proof of Lemma \ref{lm:big_centralizer}). Let $\delta \colon \bZ/n\bZ \rightarrow \{0,1\}$, $i \mapsto 1$ if $i\leq i_1$ and $i \mapsto 0$ otherwise. Let $\chi = (1_{i_1}, 0_{n-i_1})$ be the corresponding cocharacter. We have (again, cf. the proof of Lemma \ref{lm:big_centralizer}), $(w\delta)(i) = \delta(i + \lambda)$ for an appropriate $\lambda \in \bZ/n\bZ$. It follows from Lemma \ref{lm:big_centralizer} and the definition of $H_w$ that $\pi_1(H_w) \supseteq S_{w\chi}$ and $\pi_2(H_w) \supseteq S_\chi$. Hence also
\begin{equation}\label{eq:contains_regular}
\pi_1(H_w^\circ) \supseteq S_{w\chi}^\circ \quad \text{ and } \quad \pi_2(H_w^\circ) \supseteq S_\chi^\circ.
\end{equation}
From this together with Lemma \ref{lm:regular_torus_in_general} it follows that for $i=1,2$, $\pi_i(H_w^\circ)$ is not contained in any of the $\ker(\alpha_{i,j} \colon \bT_1 \rightarrow \bG_m)$ ($1\leq i\neq j \leq n$). Hence it also holds for $\pi_i(H_{w,{\rm red}}^\circ)$, where $H_{w,{\rm red}}^\circ$ is the reductive part of $H_w^\circ$ (it is a torus). By Proposition \ref{prop:torus_invariants} we now have $H_c^\ast(\widehat\Sigma_w)_{\theta^{-1},\theta'} = H_c^\ast(\widehat\Sigma_w^{H_{w,{\rm red}}^\circ})_{\theta^{-1},\theta'}$. Because $\pi_i(H_{w,{\rm red}}^\circ)$ is not contained in any of the $\ker(\alpha_{i,j})$, we have 
\[
\widehat\Sigma_w^{H_{w,{\rm red}}^\circ} \subseteq \{(1,1,\tau,1,1) \colon \tau \in \bT_h, F(\tau \dot w) = \tau \dot w \},
\]
and \eqref{eq:coh_of_Sigma_w} for $\widehat\Sigma_w$ easily follows (as in \cite[1.9, proof of claim (e)]{Lusztig_04}).

\subsection{Another extension of action}\label{sec:Step4} It remains to deal with the case $w = 1$. 
We first prove a more general result, again generalizing Lusztig's method. The proof does not depend on special properties of $\GL_n$ and can be carried out for any group, so we put ourselves -- until the end of Section \ref{sec:Step4} only -- in the general setup of \cite{CI_MPDL}. Let ${\bf G}$ be a reductive group over $K$, which is split over $\breve K$, and let ${\bf T}$, ${\bf T}'$ be two maximal $K$-rational, $\breve K$-split tori in ${\bf G}$. There is a natural inclusion of the reduced Bruhat--Tits building $\mathscr{B}_K$ of ${\bf G}$ over $K$ into the reduced Bruhat--Tits building $\mathscr{B}_{\breve K}$ of ${\bf G}$ over $\breve K$. Assume there is a point ${\bf y}$ in the intersection of $\mathscr{B}_K$ and the apartments of ${\bf T}$ and ${\bf T}'$ inside $\mathscr{B}_{\breve K}$. We have then the parahoric $\cO_K$-model $P_{\bf y}$ of ${\bf G}$ attached to ${\bf y}$. Its $\cO_{\breve K}$-points $P_{\bf y}(\cO_{\breve K})$ form the parahoric subgroup of ${\bf G}(\breve K)$ attached to ${\bf y}$, which is the stabilizer of ${\bf y}$. On $P_{\bf y}(\cO_{\breve K})$ we have the descending Moy--Prasad filtration given by certain subgroups $P_{\bf y}(\cO_{\breve K})^h$ ($h\geq 0$). Using the truncated loop group construction \cite[2.6]{CI_MPDL}, for any $h \geq 1$ one can defined an affine perfectly finitely presented perfect $\FF_q$-group $\bG_h$ satisfying
\[
\bG_h(\overline{\FF}_q) = P_{\bf y}(\cO_{\breve K})/ P_{\bf y}(\cO_{\breve K})^{(h-1)+}
\]
We denote by $F$ the (geometric) Frobenius on $\bG_{h,\overline{\FF}_q}$ and its closed subgroups. To a closed subgroup ${\bf H} \subseteq {\bf G}_{\breve K}$ one can naturally attach a closed subgroup $\mathbb{H}_h\subseteq \bG_h$, by first taking the schematic closure of ${\bf H}$ in $P_{\bf y}$ and then applying $L_h^+$. We write $\mathbb{H}_h^r := \ker(\mathbb{H}_h \rightarrow \mathbb{H}_r)$ for the kernel of the natural projection. We also write $G_h := \bG_h(\FF_q)$ and $H_h := \mathbb{H}_h(\FF_q)$ (the latter only if $\mathbb{H}_h$ is defined over $\FF_q$). For more details we refer to \cite[2.6]{CI_MPDL}.

Let ${\bf U},{\bf U}^-$ resp.\ ${\bf U}',{\bf U}^{\prime -}$ be the unipotent radicals of a pair of opposite Borel subgroups containing ${\bf T}$ resp.\ ${\bf T}'$ and let $\mathbb{U}_h,\mathbb{U}_h^-$ resp.\ $\mathbb{U}_h',\mathbb{U}_h^{\prime -}$ be the corresponding subgroups of $\bG_h$. We have the closed perfect subscheme of $\bG_h$, 
\[ 
S_{T,U,h} = \{g \in \bG_h \colon g^{-1}F(g) \in F\bU_h \}
\] 
with a $G_h \times T_h$-action by $(\gamma,t) \colon g \mapsto \gamma gt$. Similarly we have the perfect subscheme $S_{T',U',h} \subseteq \bG_h$. As already above, Lusztig's scheme $\Sigma = \{(x,x',y) \in F\bU_h \times F\bU_h' \times \bG_h \colon   xF(y) = yx'\}$ is very useful to compute the inner product between the virtual $G_h$-representations obtained from $S_{T,U,h}$ and $S_{T',U',h}$. More precisely, for $\overline{\bQ}_\ell^\times$-valued characters $\theta$ resp.\ $\theta'$ of $T_h$ resp.\ $T_h'$ we have
\[\left\langle H^\ast_c(S_{T,U,h})_\theta, H^\ast_c(S_{T',U',h})_{\theta'} \right\rangle_{G_h} = \dim_{\overline{\bQ}_\ell} H_c^\ast(\Sigma)_{\theta^{-1},\theta'} \]

To study $H_c^\ast(\Sigma)$ Lusztig in \cite{Lusztig_04} (and many authors in follow-up articles) used a locally closed decomposition $\Sigma = \coprod_{w \in W_{\bf y}(T',T)} \Sigma_w$, where $W_{\bf y}(T',T) = \{ \bT_1 v\colon v^{-1}\bT_1 v = \bT_1'\}$ is the transporter from $\bT_1'$ to $\bT_1$ in $\bG_1$ (= reductive quotient of the special fiber of $P_{\bf y}$) conjugating $\bT_1'$ to $\bT_1$. Now, we generalize this construction in a substantial way.

Let ${\bf V}$ resp.\ ${\bf V}'$ be the unipotent radical of a second Borel subgroup containing ${\bf T}$ resp.\ ${\bf T}'$. We have the corresponding subgroups $\mathbb{V}_h$, $\mathbb{V}_h'$ of $\bG_h$. For $v \in W_{\bf y}(T',T)$ we have the corresponding preimage $\mathbb{V}_h \bT_h v \mathbb{K}_{V,V',h}^1 \mathbb{V}'_h$ (with $\mathbb{K}_{V,V',r} := \mathbb{V}^{\prime -}_h \cap v^{-1}\mathbb{V}_h^- v$) of the Schubert cell in $\bG_1$ attached to $v$. We consider the following generalizations of $\widehat\Sigma_w$, $\Sigma_w$ from \cite{Lusztig_04}
\begin{align*}
\Sigma_{V,V',v} &:= \{ (x,x',y) \in F\bU_h \times F\bU_h' \times \mathbb{V}_h \bT_h \dot v \mathbb{K}_{V,V',h}^1 \mathbb{V}'_h \colon xF(y) = yx' \},
\\
\widehat\Sigma_{V,V',v} &:= \{ (x,x',y',\tau,z,y'') \in F\bU_h \times F\bU_h' \times \mathbb{V}_h \times \bT_h \times \mathbb{K}_{V,V',h}^1 \times \mathbb{V}'_h \colon \\
&\qquad\qquad\qquad\qquad\qquad\qquad\qquad\qquad\qquad xF(y' \tau \dot v z y'') = y' \tau \dot v z y''x' \},
\end{align*}
which have the same alternating sum of cohomology. The action of $T_h \times T_h'$ on $\widehat\Sigma_{V,V',v}$ is
\begin{equation}\label{eq:ThTh_action_widehatSigma_general}
(x,x',y',\tau,z,y'') \overset{(t,t')}{\mapsto} (txt^{-1},t'x't^{\prime -1},ty't^{-1},t \tau \dot vt^{\prime -1} \dot v^{-1},t'zt^{\prime -1},t'y''t^{\prime -1}).
\end{equation}
and by a similar formula for $\Sigma_{V,V',v}$. There is an element $v_0 = v_0(V,V') \in W_{\bf y}(T',T)$, such that the (generalized) Bruhat cell $\mathbb{V}_1\bT_1 v_0\mathbb{V}'_1$ is generic in $\bG_1$, i.e., $v_0^{-1}\mathbb{V}_h v_0 = \mathbb{V}_h^{\prime -}$. For this $v_0$ we have $\mathbb{K}_{V,V',h} = 1$. We can write $y' \in \mathbb{V}_h$ and $y'' \in \mathbb{V}'_h$ as 
\begin{align*}
y' &= y_1' y_2' && \text{where $y_1'\in \bU_h \cap \mathbb{V}_h, \, y_2' \in \bU_h^- \cap \mathbb{V}_h,$} \\
y'' &= y_1'' y_2'' && \text{where $y_1'' \in \bU_h^{\prime -} \cap \mathbb{V}_h',\, y_2'' \in \bU_h' \cap \mathbb{V}_h'$}.
\end{align*}
where $(t,t') \in \bT_h \times \bT_h'$ acts on $y_1',y_2'$ resp. $y_1'', y_2''$ by conjugation with $t$ resp. with $t'$.
% \begin{equation*}
% (y_1',y_2',y_1'',y_2'') \overset{(t,t')}{\mapsto} (ty_1't^{-1},ty_2't^{-1},t'y_1''t^{\prime -1},t' y_2'' t^{\prime -1})
% \end{equation*}
Changing the variables $x F(y_1') \mapsto x$, $x'F(y_2'')^{-1} \mapsto x'$ we can rewrite $\widehat\Sigma_{V,V',v_0}$ as
\begin{align*}
\{(x,y'_1,y_2',\tau,y_1'',y_2'') \in F\bU_h \times (\bU_h \cap \mathbb{V}_h) \times (\bU_h^- \cap \mathbb{V}_h) \times \bT_h  {}&{} \times (\bU_h^{\prime -} \cap \mathbb{V}_h') \times (\bU_h' \cap \mathbb{V}_h') \colon \\ & xF(y'_2 \tau \dot v y''_1) \in y'_1y_2' \tau \dot v  y''_1y_2'' F\bU_h'\}.
%\left\{ \begin{array}{cl} (x,y'_1,y_2',\tau,y_1'',y_2'') \in F\bU_h \times (\bU_h \cap \mathbb{V}_h) \times (\bU_h^- \cap \mathbb{V}_h) \times \bT_h \times (\bU_h^{\prime -} \cap \mathbb{V}_h') \times (\bU_h' \cap \mathbb{V}_h') \colon \\ \qquad \qquad \qquad \qquad \qquad \qquad \qquad \qquad \qquad \qquad \qquad \qquad xF(y'_2 \tau \dot v y''_1) \in y'_1y_2' \tau \dot v  y''_1y_2'' F\bU_h' \end{array} \right\}.
\end{align*}

Let 
% $H'_{v_0}$ denote the group consisting of $(t,t') \in \bT_h \times \bT_h'$ such that $F(t)t^{-1} = \dot v_0 t'F(t')^{-1} \dot v_0^{-1}$ centralizes both $\bU_h^- \cap \mathbb{V}_h$ and $\dot v_0 (F\bU_h^{\prime -} \cap \mathbb{V}_h') \dot v_0^{-1}$.
\[
H'_{v_0} = \{(t,t') \in \bT_h \times \bT_h' \colon F(t)t^{-1} = \dot v_0 t'F(t')^{-1} \dot v_0^{-1} \text{ centralizes $\bU_h^- \cap \mathbb{V}_h$ and $\dot v_0 (F\bU_h^{\prime -} \cap \mathbb{V}_h') \dot v_0^{-1}$} \}.
\]
Define an action of $H'_{v_0}$ on $\widehat\Sigma_{V,V',v_0}$ by 
\[
(x,y'_1,y_2',\tau,y_1'',y_2'') \overset{(t,t')}{\mapsto} (F(t)xF(t)^{-1},F(t)y'_1F(t)^{-1},ty_2't^{-1},t\tau \dot v_0 t^{\prime -1 }\dot v_0^{-1},t'y_1''t^{\prime -1},t'y_2''t^{\prime -1}).
\]
It extends the action of $T_h \times T_h'$. We have to show that it is well-defined, i.e., that if $(x,y'_1,y_2',\tau,y_1'',y_2'') \in \widehat\Sigma_{V,V',v_0}$, then the same holds for $(t,t').(x,y'_1,y_2',\tau,y_1'',y_2'')$. This reduces to show that
\[
xF(y_2')F(\tau)F(\dot v_0)F(y_1'') \in y_1'F(t)^{-1}t y_2' \tau \dot{v}_0 y_1''y_2'' F\bU_h' t^{\prime -1}F(t')
\]
Writing $y'' = y_1''y_2'' \in \mathbb{V}_h'$ as $y'' =: y''_3 y''_4$ with $y''_3 \in \mathbb{V}' \cap F\bU_h^{\prime -}$ and $y''_4 \in \mathbb{V}' \cap F\bU_h'$, it suffices to check that $F(t)t^{-1}$ commutes with $y_2' \in \bU_h^- \cap \mathbb{V}_h$ and that $t^{\prime -1}F(t') = \dot v_0^{-1} F(t)t^{-1} \dot v_0$ commutes with $y_3'' \in \mathbb{V}' \cap F\bU_h^{\prime -}$. This holds by definition of $H_{v_0}'$. We thus have proven the following lemma.
\begin{lm}\label{lm:extension_action_generic_cells_way_1}
The action of $T_h \times T_h'$ on $\widehat\Sigma_{V,V',v_0}$ extends to an action of the algebraic group $H_{v_0}'$ given by the above formula.
\end{lm}

Now returning to the proof Theorem \ref{thm:general_irred}, we apply Lemma \ref{lm:extension_action_generic_cells_way_1} to our ${\bf G}$ (= inner form of $\bfGL_n$), the point ${\bf y} = {\bf x}_b$, the diagonal (elliptic unramified) torus ${\bf T} = {\bf T}'$ of ${\bf G}$, the subgroup ${\bf U} = {\bf U}'$ of unipotent upper triangular matrices and to ${\bf V} = {\bf U}$, ${\bf V}' = {\bf U}^-$, $v_0 = 1$, in which case $\bU_h^- \cap \mathbb{V}_h = 1$ and $\dot v_0 (F\bU_h^- \cap \mathbb{V}_h') \dot v_0^{-1}$ is contained in $\mathbb{L}_h$ for some proper Levi subgroup ${\bf L}$ of ${\bf G}_{\breve K}$, and hence the reductive part $H_{1,{\rm red}}^{\prime \circ}$ of the connected component of $H_1'$ is big enough in the sense of Lemma \ref{lm:regular_torus_in_general}. Note finally that $\Sigma_1 = \Sigma_{U,U,1}$ is a closed subscheme of $\Sigma_{U,U^-,1}$, (in fact, on $\Sigma_{U,U^-,1}$, $y$ varies in $\bT_h \bU_h\bU_h^-$ and $\Sigma_1$ is given by the closed condition $y \in \bT_h \bU_h\bU_h^{-,1}$). Let $\widetilde \Sigma_1$ denote the pullback of $\Sigma_1$ along $\widehat \Sigma_{U,U^-,1} \rightarrow \Sigma_{U,U^-,1}$. It has the same alternating sum of cohomology as $\Sigma_1$, and it is clearly stable under the action of $H_1'$. Thus, exploiting Proposition \ref{prop:torus_invariants} in the last equality, we deduce 
$H_c^\ast(\widehat\Sigma_1) = H_c^\ast(\Sigma_1) = H_c^\ast(\widetilde\Sigma_1) = H_c^\ast(\widetilde\Sigma_1^{H_{1,{\rm red}}^{\prime \circ}})$ (and the same for $\theta^{-1} \otimes \theta'$-isotypic parts), hence verifying \eqref{eq:coh_of_Sigma_w} in the only remaining case $w=1$. This completes the proof of Theorem \ref{thm:general_irred}.

\section{A variation of the Mackey formula}\label{sec:VariationMackey}

We work with exactly the same setup and notation as in Section \ref{sec:Mackey} (and in particular Section \ref{sec:Step1}). Recall the presentation \eqref{eq:Xh_into_GGh} of $X_h$. Then 
\[
X_{h,n'} = \{g \in \bG_h \colon g^{-1}F(g) \in \bU_h^{-,1} \cap F\bU_h^1 \},
\]
is a closed perfect subscheme of $X_h$, stable under the action of $G_h\times T_h$. 
In fact, $X_h$ has a stratification in locally-closed pieces \cite{CI_DrinfeldStrat} indexed by divisors $r$ of $n'$, and $X_{h,n'}$ is precisely the closed stratum.

\begin{thm}\label{thm:relation_Xh_Xhn}
Let $\theta \colon T_h \rightarrow \overline{\QQ}_\ell^\times$ be a character. Assume that $p > n$, and that $\theta|_{T_h^1}$ has trivial stabilizer in $W_{\cO}^F$. Then 
\begin{equation}\tag{a}\label{eq:relation_Xh_Xhn_a}
\left\langle R_{T_h}^{G_h}(\theta), H_c^\ast(X_{h,n'})_\theta \right\rangle_{G_h} = 1
\end{equation}
and
\begin{equation}\tag{b}\label{eq:relation_Xh_Xhn_b}
\Big\langle H_c^\ast(X_{h,n'})_\theta, H_c^\ast(X_{h,n'})_\theta \Big\rangle_{G_h} = 1.
\end{equation}
\end{thm}

We prove Theorem \ref{thm:relation_Xh_Xhn} in Sections \ref{sec:proof_of_thm_XhXhn_a_multiplicative}-\ref{sec:proof_of_thm_XhXhn_b}. From Theorems \ref{thm:general_irred} and \ref{thm:relation_Xh_Xhn} we deduce:

\begin{cor}
Under the assumptions of Theorem \ref{thm:relation_Xh_Xhn}, $H_c^\ast(X_{h,n'})_\theta$ is up to sign an irreducible representation of $G_h$, and $H_c^\ast(X_{h,n'})_\theta \cong R_{T_h}^{G_h}(\theta)$.
\end{cor}

In the proof of Theorem \ref{thm:relation_Xh_Xhn} we use the following well-known fact.

\begin{prop}\label{prop:additive_extension}
Let $H$ be a connected algebraic group acting on (the perfection of) a scheme $Y$, separated and of finite type of $\overline{\bF}_q$. Then each $h \in H(\overline{\bF}_q)$ acts trivially in $H^i_c(Y,\overline{\bQ}_\ell)$ for each $i \in \bZ$.
\end{prop}
\begin{proof} See \cite[6.5]{DeligneL_76}.
\end{proof}

\begin{rem}\label{rem:thm5.1}
The proof of Theorem \ref{thm:relation_Xh_Xhn} follows the same pattern as the proof of Theorem \ref{thm:general_irred}, cf. Remark \ref{rem:two_principles}. Instead of $\widehat\Sigma_w$ one has similar schemes $\widehat\Sigma_{(1,n),w}$. Cases analogous to (a) and (c) of step (4) in Remark \ref{rem:two_principles} are handled very similar to the proof of Theorem \ref{thm:general_irred}. However, those $w$ falling into case (b) require separate treatement. Here we only can extend the action of $T_h^1 \times T_h^1$ on $\widehat\Sigma_{(1,n),w}$ to a commutative \emph{unipotent} group $H_w^1$. Then $H_{w,{\rm red}}^\circ = 1$ and Proposition \ref{prop:torus_invariants} is not applicable. Instead we use Proposition \ref{prop:additive_extension}, which is more delicate.
\end{rem}

% \begin{rem}
% Alltogether, we use two results to control the action of finite groups in cohomology, namely Propositions \ref{prop:torus_invariants} and \ref{prop:additive_extension}. 
% For our purposes, Proposition \ref{prop:torus_invariants} is stronger than Proposition \ref{prop:additive_extension}, which does not allow quantitative results in Section \ref{sec:proof_of_thm_XhXhn_a_additive} below. This explains why Theorem \ref{thm:relation_Xh_Xhn} is less general than Theorem \ref{thm:general_irred}.
% \end{rem}

\subsection{Proof Theorem \ref{thm:relation_Xh_Xhn}\eqref{eq:relation_Xh_Xhn_a}: multiplicative extension}\label{sec:proof_of_thm_XhXhn_a_multiplicative}
Parts of the proof follows along the same lines as the proof of Theorem \ref{thm:general_irred}, thus we will be slightly sketchy below. Similar as in \cite[Lemma 7.12]{CI_ADLV} we have an isomorphism
\[
(\bU_h^1 \cap F\bU_h^1) \times (\bU_h^{-,1} \cap F\bU_h^1) \rightarrow F\bU_h^1, \quad (g,x) \mapsto g^{-1}xF(g).
\]
Thus we have $G_h \times T_h$-equivariantly $X_{h,n'} \cong S_{h,n'}/(\bU_h^1 \cap F\bU_h^1)$, where
\begin{align*}
% X_{h,n} &= \{ g \in \bG_h \colon g^{-1}F(g) \in \bU_h^{-,1} \cap F\bU_h^1 \} = S_{h,n}/(\bU_h^1 \cap F\bU_h^1), \text{ where} \\
S_{h,n'} = \{ g \in \bG_h \colon g^{-1}F(g) \in F\bU_h^1 \}
\end{align*}
and $G_h \times T_h$ acts on $S_{h,n'}$ by $g,t \colon x \mapsto gxt$, and $(\bU_h^1 \cap F\bU_h^1)$ by right multiplication. Hence $H_c^\ast(X_{h,n'})_\theta \cong H_c^\ast(S_{h,n'})_\theta$. Using Lang's theorem, we have a $T_h \times T_h$-equivariant isomorphism
\[
G_h \backslash (S_h \times S_{h,n'}) \stackrel{\sim}{\rightarrow} \Sigma_{(1,n)} := \{(x,x',y) \in F\bU_h \times F\bU_h^1 \times \bG_h \colon xF(y) = yx' \}
\]
where $T_h \times T_h$ acts on $\Sigma_{(1,n)}$ by $(t,t') \colon (x,x',y) \mapsto (txt^{-1},t'x't^{\prime -1},tyt^{\prime -1})$. For $w \in W_\cO$ let $\Sigma_{(1,n),w} = \{(x,x',y) \in \Sigma_{(1,n)} \colon y \in \bG_{h,w}\}$ (it is an $T_h \times T_h$-stable locally closed perfect subscheme) and putting $\bK_h = \bU_h \cap \dot w^{-1}\bU_h^- \dot w$, we let 
\[
\widehat\Sigma_{(1,n),w} = \{(x,y_1,\tau,z,y_2) \in F\bU_h \times \bU_h \times \bT_h \times \mathbb{K}_h^1 \times \bU_h \colon xF(y_1\tau \dot w z y_2) \in y_1\tau \dot w z y_2F\bU_h^1 \}
\]
be the Zariski-locally trivial covering of $\Sigma_{(1,n),w}$ with $T_h \times T_h$-action given by the same formula as in \eqref{eq:ThTh_action_widehatSigma}. As in Section \ref{sec:Step1}, we have $\langle R_{T_h}^{G_h}(\theta), H_c^\ast(X_{h,n'})_\theta \rangle_{G_h} = \sum_{w \in W_\cO} \dim H_c^\ast(\widehat\Sigma_{(1,n),w})_{\theta^{-1},\theta}$. We claim that
\begin{equation}\label{eq:dimension_Sigma0nw}
\dim H_c^\ast(\widehat\Sigma_{(1,n),w})_{\theta^{-1},\theta} = \begin{cases} 1 & \text{if $w = 1$}, \\ 0 & \text{otherwise,} \end{cases}
\end{equation}
which implies the first formula of Theorem \ref{thm:relation_Xh_Xhn}. First assume $w$ satisfies the condition in Lemma \ref{lm:empty_Sigma_w}. Then $\widehat\Sigma_{(1,n),w} \subseteq \widehat\Sigma_w = \varnothing$ and we are done. 
Now assume that $w=1$. Then $\bG_{h,1} = \bU_h \cdot \bT_h \cdot \bU_h^{-,1}$, so
\[
\widetilde\Sigma_{(1,n),1} = \{(x,x',y_1,\tau,z) \in F\bU_h \times F\bU_h^1 \times \bU_h \times \bT_h \times \bU_h^{-,1} \colon xF(y_1\tau z) \in y_1\tau z x'\}
\]
is another a Zariski-locally trivial covering of $\Sigma_{(1,n),1}$ (with obvious $T_h \times T_h$-action), so that $H_c^\ast(\widehat\Sigma_{(1,n),1})_{\theta^{-1},\theta}) = H_c^\ast(\widetilde\Sigma_{(1,n),1})_{\theta^{-1},\theta})$, and we can replace $\widehat\Sigma_{(1,n),1}$ by $\widetilde\Sigma_{(1,n),1}$. We can uniquely write $z = z_1z_2$ with $z_1 \in \bU_h^{-,1} \cap F\bU_h^{-,1}$ and $z_2 \in \bU_h^{-,1} \cap F\bU_h^1$ and make the change of variables $xF(y_1) \mapsto x$, $z_2x' \mapsto x'$ (note that the latter works because $z_2 \in F\bU_h^1$!), so that $\widetilde \Sigma_{(1,n),1}$ is isomorphic to
\[
%\widetilde\Sigma_{(1,n),1} \cong 
\{(x,y_1,\tau,z_1,z_2) \in F\bU_h \times \bU_h \times \bT_h \times (\bU_h^{-,1} \cap F\bU_h^{-,1}) \times (\bU_h^{-,1} \cap F\bU_h^1) \colon xF(\tau z_1 z_2) \in y_1\tau z_1 F\bU_h^1\}.
\]
The $T_h \times T_h$-action on $\widetilde\Sigma_{(1,n),1}$ is given by 
\[
(t,t') \colon (x,y_1,\tau,z_1,z_2) \mapsto (txt^{-1}, ty_1t^{-1}, t\tau t^{\prime -1}, t'z_1t^{\prime -1}, t'z_2t^{\prime -1}).
\]
Let 
\[
H_1 := \{ (t,t') \in \bT_h \times \bT_h \colon t^{-1}F(t) = t^{\prime -1}F(t')\, \text{ centralizes $\bU_h \cap F\bU_h^-$} \}.
\]
As in Sections \ref{sec:Step3} and \ref{sec:Step4}, one can check that $H_1$ acts on $\widetilde\Sigma_{(1,n),1}$ by 
\[(t,t') \colon (x,y_1,\tau,z_1,z_2) \mapsto (F(t)xF(t)^{-1}, F(t)y_1F(t)^{-1}, t\tau t^{\prime -1}, t'z_1t^{\prime -1}, t'z_2t^{\prime -1}) \]
(and this action extends the action of $T_h \times T_h$). Since $\bU_h \cap F\bU_h^-$ is contained in the subgroup of $\bG_h$ attached to a proper rational Levi subgroup $L \subseteq {\bf G}_{\breve K}$, it follows that the connected component $H_{1,{\rm red}}^\circ$ of the reductive part of $H$ is big enough (in the sense of Lemma \ref{lm:regular_torus_in_general}), so that we deduce $\dim H_c^\ast(\widetilde\Sigma_{(1,n),1})_{\theta^{-1},\theta} = 1$, and hence \eqref{eq:dimension_Sigma0nw} for $w = 1$ (this is the same argument as at the end of Section \ref{sec:Step4}).

\subsection{Proof Theorem \ref{thm:relation_Xh_Xhn}\eqref{eq:relation_Xh_Xhn_a}: additive extension}\label{sec:proof_of_thm_XhXhn_a_additive}
It remains to show \eqref{eq:dimension_Sigma0nw} for $1 \neq w \in W_\cO$ not satisfying the condition from Lemma \ref{lm:empty_Sigma_w}. Assume $w$ is such an element. Let 
\[
H_w^1 := \{(t,t') \in \bT_h^1 \times \bT_h^1 \colon \dot{w}^{-1}t^{-1}F(t)\dot{w} = t^{\prime -1}F(t')\, \text{centralizes $\mathbb{K}_h^1$} \}.
\]
In $\widehat\Sigma_{(1,n),w}$ make the change of variables $xF(y_1) \mapsto x$, so that 
\[
\widehat\Sigma_{(1,n),w} = \{(x,y_1,\tau,z,y_2) \in F\bU_h \times \bU_h \times \bT_h \times \mathbb{K}_h^1 \times \bU_h \colon xF(\tau \dot w z) \in y_1\tau \dot w z y_2F(\bU_h^1 y_2^{-1}) \}
\]
with $T_h \times T_h$-action given by the same formula as in \eqref{eq:ThTh_action_widehatSigma}. Now
\[
(t,t') \colon (x,y_1,\tau,z,y_2) \mapsto (F(t)xF(t)^{-1}, F(t)y_1F(t)^{-1}, t\tau \dot{w}t^{\prime -1} \dot{w}^{-1}, t' z t^{\prime -1}, F(t')y_2F(t')^{-1})
\]
defines an action of $H_w^1$ on $\widehat\Sigma_{(1,n),w}$. In order to check this we have to show that if $(t,t') \in H_w^1$ and $(x,y_1,\tau,z,y_2) \in \widehat\Sigma_{(1,n),w}$, then also $(t,t').(x,y_1,\tau,z,y_2) \in \widehat\Sigma_{(1,n),w}$. After elementary cancellations this reduces to show that
\[
xF(\tau \dot w zt^{\prime -1}) \in y_1F(t)^{-1}t \tau \dot w z t^{\prime -1}F(t') y_2 F(t')^{-1} F(\bU_h^1 F(t')y_2^{-1}F(t')^{-1})
\]
But as $t' \in \bT_h^1$, we have $\bU_h^1 F(t') y_2^{-1} F(t')^{-1} = \bU_h^1 y_2^{-1}$, so this reduces to show that
\[
xF(\tau \dot w z) \in y_1F(t)^{-1}t \tau \dot w z t^{\prime -1}F(t') y_2 F(t^{\prime -1}\bU_h^1 y_2^{-1} t').
\]
Again, using $t' \in \bT_h^1$, we deduce that $t^{\prime -1}\bU_h^1 y_2^{-1} t' = \bU_h^1 y_2^{-1}$, so $(t,t').(x,y_1,\tau,z,y_2) \in H_w^1$.

Via the isomorphism $T_h \stackrel{\sim}{\rightarrow} U_L/U_L^h$ mapping a diagonal matrix $t = (t_i)_{i=1}^n$ to its upper left entry $t_1$, we identify $T_h$ with $U_L/U_L^h$ and $T_h^1$ with $U_L^1/U_L^h$. By Lemma \ref{lm:factor_through_norm_equiv_cond} (and the discussion in Section \ref{sec:case_bequalw}), the condition that $\theta|_{T_h^1}$ has trivial stabilizer in $W_\cO^F = \langle w_1^{n_0} \rangle$ translates to the condition that the restriction of $\theta$ to $U_L^1/U_L^h$ does not factor through any of the norm maps $N_{n/n_0s} \colon U_L^1/U_L^h \twoheadrightarrow U_{K_{n_0s}}^1/U_{K_{n_0 s}}^h$, where $1\leq s < n'$ goes through all divisors of $n'$. Let $H_w^{1,\circ}$ be the connected component of $H_w^1$.
\begin{lm}\label{lm:Hw1_is_big_enough}
If $(t,t')$ varies through $(T_h^1 \times T_h^1) \cap H_w^{1,\circ}$, then $t_1^{-1}t_1'$ varies (at least) through all elements of ${\rm ker}({\rm N}_{n/n_0s})$ for some divisor $1\leq s < n'$ of $n'$ ($s$ depends on $w$).
\end{lm}

Before proving this lemma, we use it to finish the proof of Theorem \ref{thm:relation_Xh_Xhn}\eqref{eq:relation_Xh_Xhn_a}. Indeed, by assumption on $\theta$ for each divisor $s<n'$ of $n'$ there is an element $x = x_s \in \ker\N_{n/n_0s} \subseteq U_L^1/U_L^h$ such that $\theta(x_s) \neq 1$. By Lemma \ref{lm:Hw1_is_big_enough} we can find a divisor $s< n'$ of $n'$ and an element $(t,t') \in (T_h^1 \times T_h^1) \cap H_w^{1,\circ}$ such that $t_1^{-1}t_1' = x_s$, and hence $\theta(t_1^{-1}t_1') \neq 1$. Seeing $\theta$ as a character of $T_h^1$ again, this simply means that $\theta(t) \neq \theta(t')$, and it follows that the $T_h \times T_h$-character $\theta^{-1} \otimes \theta$ is non-trivial on $(T_h^1 \times T_h^1) \cap H_w^{1,\circ}$. 
% But the induced action of a connected algebraic group in the cohomology of a separated scheme of finite type over $\overline{\FF}_q$ is trivial \cite[Corollary 6.5]{DeligneL_76} and the same holds after perfection, hence 
By Proposition \ref{prop:additive_extension} we thus deduce $H_c^i(\widehat\Sigma_{(1,n),w})_{\theta^{-1},\theta} = 0$  for each $i\geq 0$, which shows claim \eqref{eq:dimension_Sigma0nw} for all remaining elements $w$, and hence also Theorem \ref{thm:relation_Xh_Xhn}\eqref{eq:relation_Xh_Xhn_a}.

\begin{rem}
The basic idea in the above arguments is the same as in \cite[Lemma 6.7]{DeligneL_76}. This gives hope to generalize them to a far more general setup (e.g.\ all unramified maximal tori in all reductive groups). 
\end{rem}

Towards the proof of Lemma \ref{lm:Hw1_is_big_enough}, for positive integers $s, r$ such that $s$ divides $r$, we define morphisms of perfect $\FF_q$-schemes
\[
\Nm_{r/s} \colon \bW_h^{\times,1} \rightarrow \bW_h^{\times,1} \quad x \mapsto \Nm_{r/s}(x) := \prod_{i=0}^{\frac{r}{s}-1}\sigma^s(x).
\]

\begin{proof}[Proof of Lemma \ref{lm:Hw1_is_big_enough}]
By assumption, $w$ does not satisfy the condition of Lemma \ref{lm:empty_Sigma_w}. Thus by Lemma \ref{lm:big_centralizer} there is a proper Levi subgroup $L \subseteq {\bf G}_{\breve K}$ containing ${\bf T}_{\breve K}$, such that if $\bL_h$ is the corresponding subgroup of $\bG_h$, we have $\bK_h \subseteq \bL_h$. We may assume $L$ is maximal, so that there is an $1\leq m \leq n-1$, such that $L = \GL_{m,\breve K} \times \GL_{n-m,\breve K}$ (upper left and lower right diagonal blocks). More precisely, we may (and do) choose that $m$ to be the $i_1$ from the proof of Lemma \ref{lm:big_centralizer}. In fact, by our explicit description of $W_\cO \cong \prod_{i=1}^{n_0} S_{n'}$ in Section \ref{sec:Step1}, we see that as $w \in W_\cO$, our choice $m = w^{-1}(n)$ must be an integer dividing $n_0$. Let $\chi = (1_m,0_{n-m})$ be a cocharacter of ${\bf T}_{\breve K}$. From the explicit form of $w$ determined in Lemma \ref{lm:big_centralizer}, we see that $w\chi = (0_{n-m}, 1_m)$. Let $\mathbb{Y}_{h,\chi} \subseteq \bT_h$ denote the subgroup of $\bT_h$ corresponding to the subgroup $\im(\chi)$ of ${\bf T}_{\breve K}$ (thus $\mathbb{Y}_{h,\chi} \cong \bW_h^\times$). As $\im(\chi)$ centralizes $L$, $\mathbb{Y}_{h,\chi}$ centralizes $\bL_h$ and hence also $\bK_h$. Thus
\[
H_w^1 \supseteq 
H_{w,\chi}^1 := \{(t,t') \in \bT_h^1 \times \bT_h^1 \colon \dot{w}^{-1}t^{-1}F(t)\dot{w} = t^{\prime -1}F(t') \in \mathbb{Y}_{h,\chi}^1 \},
\]
and the same inclusion holds if we take connected components on both sides. Thus we may replace $H_w^1$ by $H_{w,\chi}^1$. Let $(t,t') \in \bT_h^1 \times \bT_h^1$. Write $t = \diag(t_i)_{i=1}^n$ and $t' = \diag(t_i')_{i=1}^n$ with $t_i,t'_i \in \bW_h^{\times, 1}$. Let $x$ be a $\bW_h^{\times, 1}$-``coordinate'' on $\mathbb{Y}_{h,\chi}^1$ (it is an $(h-1)$-tuple of $\bA^1$-coordinates). We can eliminate all ``coordinates'' $t_i$ ($i\neq n$) and $t_i'$ ($i\neq m$) by expressing them through $x$ and $t_m$, $t_n'$. More precisely,
\[
H_{w,\chi}^1 \cong \{(x,t_m,t_n') \in \bW_h^{\times,1} \times \bW_h^{\times,1} \times \bW_h^{\times,1} \colon \sigma^n(t_n)t_n^{-1} = \Nm_{m/1}(x) = \sigma^n(t_m')t_m^{\prime -1} \}.
\]
We see that on $H_{w,\chi}^1$, the equation $\sigma(t_n^{-1}t_m') = t_n^{-1}t_m'$ holds, so that $t_n^{-1}t_m'$ can take only finitely many values. On $H_{w,\chi}^{1,\circ}$ we must in particular have $t_n = t_m'$, or equivalently (using the expression of $t_1$, $t_1'$ through $t_n$, $t_m'$) we have
\begin{equation}\label{eq:rel_t1t1prime_on_Hwchi1}
\sigma^{n-m}(t_1) = t_1'
\end{equation}
on $H_{w,\chi}^{1,\circ}$. Furthermore, $H_{w,\chi}^{1,\circ}$ is contained in the perfect scheme (isomorphic to)
\[
\{ (x,t_n) \in \bW_h^{\times,1} \times \bW_h^{\times,1} \colon \sigma^n(t_n)t_n^{-1} = \Nm_{m/1}(x) \}
\]
Now let $1\leq g = {\rm gcd}(m,n)<n$. As $\sigma^n(t_n)t_n^{-1} = \Nm_{n/1}(\sigma(t_n)t_n^{-1}) = \Nm_{g/1}(\Nm_{n/g}(\sigma(t_n)t_n^{-1}))$,
and $\Nm_{m/1}(x) = \Nm_{g/1}(\Nm_{m/g}(x))$, we have $\Nm_{g/1}(\Nm_{n/g}(\sigma(t_n)t_n^{-1})\Nm_{m/g}(x)^{-1}) = 1$ on this scheme, and hence $\Nm_{n/g}(\sigma(t_n)t_n^{-1})\Nm_{m/g}(x)^{-1}$ is discrete on it. Hence $H_{w,\chi}^{1,\circ}$ is contained in the perfect scheme (isomorphic to)
\[
\{ (x,t_n) \in \bW_h^{\times,1} \times \bW_h^{\times,1} \colon \Nm_{n/g}(\sigma(t_n)t_n^{-1}) = \Nm_{m/g}(x) \}.
\]
After replacing $\sigma$ by $\sigma^g$, Lemma \ref{lm:connectedness_Euclidean_alg} shows that this last perfect $\overline{\FF}_q$-scheme is connected, so that it is equal to $H_{w,\chi}^{1,\circ}$. On $H_{w,\chi}^1$, $t_1 = \sigma(t_n)$, so that (after replacing $\sigma(x)$ by $x$ which is harmless here), we have
\[
H_{w,\chi}^{1,\circ} \cong \{ (x,t_1) \in \bW_h^{\times,1} \times \bW_h^{\times,1} \colon \Nm_{n/g}(\sigma(t_1)t_1^{-1}) = \Nm_{m/g}(x) \}
\]
Now $H_{w,\chi}^{1,\circ} \cap (T_h^1 \times T_h^1)$ is the locus in $H_{w,\chi}^{1,\circ}$ defined by $x = 1$. Thus we deduce 
\[
H_{w,\chi}^{1,\circ} \cap (T_h^1 \times T_h^1) = \{(t,t') \in T_h^1 \times T_h^1 \colon t_1' = \sigma^{n-m}(t_1) \text{ and } \Nm_{n/g}(\sigma(t_1)t_1^{-1}) = 1\}
\]
(recall that in $T_h^1$, $t$ is determined by its first entry $t_1$). Note that $\Nm_{n/g}(\sigma(t_1)t_1^{-1}) = 1$ simply means that $\Nm_{n/g}(t_1)$ is $\sigma$-stable. As $m$ is divisible by $n_0$, $T_h^1 = \bW_h^{\times,1}(\FF_{q^n}) = U_L^1/U_L^h$ and the restriction of $\Nm_{n/g}$ to $T_h^1 \cong U_L^1/U_L^h$ is $\N_{n/g}$, the lemma now follows from Lemma \ref{lm:image_of_a_norm_map}. 
\end{proof}
 
\begin{lm}\label{lm:image_of_a_norm_map}
Suppose $(n,p) = 1$. Let $1\leq m\leq n-1$ and put $g = {\rm gcd}(n,m)$. Let
\[
\alpha \colon \{ y \in U_L^1/U_L^h \colon \N_{n/g}(y) \in U_K^1/U_K^h \} \rightarrow U_L^1/U_L^h, \quad y \mapsto \sigma^{n-m}(y)y^{-1}.
\]
Then $\im(\alpha) = \ker(\N_{n/g} \colon U_L^1/U_L^h \rightarrow U_{K_g}^1/U_{K_g}^h)$.
\end{lm}
\begin{proof}
For arbitrary $a \in \bZ$ we have 
\[
\N_{n/g}(y) \in U_K^1/U_K^h \Rightarrow \N_{n/g}(\sigma^a(y)y^{-1}) = \sigma^a(\N_{n/g}(y))\N_{n/g}(y)^{-1} = 1 \Rightarrow \sigma^a(y)y^{-1} \in \ker(\N_{n/g}).
\]
Hence $\im(\alpha) \subseteq \ker(\N_{n/g})$. Let $y \in \ker(\alpha)$. Then $\N_{n/g}(y)$ is rational and $\sigma^{n-m}(y) = y$ and $\sigma^n(y) = y$. The last two equalities together are equivalent to $\sigma^g(y) = y$. Hence $\N_{n/g}(y) = \frac{n}{g}y$, and hence $y$ is rational (as $\N_{n/g}(y)$ is, and $(n,p) = 1$). Conversely, if $y$ is rational, then surely $y\in \ker(\alpha)$. Thus $\ker(\alpha) = U_K^1/U_K^h$. Now the source of $\alpha$ is the preimage under the (surjective) map $\N_{n/g} \colon U_L^1/U_L^h \rightarrow U_{K_g}^1/U_{K_g}^h$ of $U_K^1/U_K^h$, hence the size of the source of $\alpha$ is $\#\ker(\N_{n/g}) \cdot \#(U_K^1/U_K^h)$. Thus $\#\im(\alpha) = \frac{\#(\text{source of $\alpha$})}{\#\ker(\alpha)} = \#\ker(\N_{n/g})$. As we already know that $\im(\alpha) \subseteq \ker(\N_{n/g})$ and both sets are finite, we are done. 
\end{proof}

For positive integer $s$ define the $\FF_q$-morphism 
\[
\tr_{s/1} \colon \bG_a\rightarrow \bG_a, \quad x \mapsto \tr_{s/1}(x) := \sum_{i=0}^{s-1}x^{q^i}.
\]

\begin{lm}\label{lm:connectedness_Euclidean_alg}
Let $r>s\geq 1$ be coprime integers. Suppose $p>s$. The closed perfect subscheme
\[
R_h = \{(y,x)\in \bW_h^{\times,1} \times \bW_h^{\times,1} \colon \Nm_{r/1}(\sigma(y)y^{-1}) = \Nm_{s/1}(x) \}
\] 
of $\bW_h^{\times,1} \times \bW_h^{\times,1}$ is connected. More precisely, for $h \geq 2$ the fibers of $R_h \rightarrow R_{h-1}$ are isomorphic to $\bA^1$ (note that $R_1$ is a point).
\end{lm}
\begin{proof} 
It suffices to prove that the fibers of $R_h \rightarrow R_{h-1}$ are isomorphic to $\bA^1$. The fibers of $R_h \rightarrow R_{h-1}$ are isomorphic to closed sub-(perfect schemes) of $\bG_a^2$ (with coordinates $X,Y$) given by the equation
\[
C \colon \tr_{r/1}(Y^q - Y) = \tr_{s/1}(X) + const.
\]
where $const$ is a constant term depending on the point in $R_{h-1}$. As $\tr_{r/1}(Y^q - Y) = Y^{q^r} - Y$, one can eliminate this constant term by changing the variable $Y + c \mapsto Y$ (for an appropriate $c \in \overline{\FF}_q$). So we assume $const = 0$. We may assume $s>1$, as otherwise we obviously have $C \cong \bA^1$. Put $r_0 := r$, $r_1 := s$ and define $r_i \in \bZ_{\geq 0}$ ($i\geq 2$), $\gamma_i \in \bZ_{> 0}$ ($i\geq 1$) by $r_i = \gamma_{i+1}r_{i+1} + r_{i+2}$ and $r_{i+2} < r_{i+1}$ for $i \geq 0$. Say this stops at $i=\alpha$, that is $r_{\alpha+1} = {\rm gcd}(r,s) = 1$, $r_{\alpha+2} = 0$.

Via the change of variables $X + Y^{q^{r-s+1}} - Y \mapsto X$, $C$ is isomorphic to the curve 
\[
C_1 \colon \tr_{r_1/1}(X) = \tr_{r_2/1}(Y^q - Y).
\]
Now $\tr_{r_2/1}(Y^q - Y) = Y^{q^{r_2}} - Y$, so that we can successively make a series of changes of variables of the form $Y + X^{q^\beta} \mapsto Y$ for appropriate $\beta \in \bZ_{\geq 0}$, to eliminate all powers of $X$ with exponent greater than $q^{r_2}$. This shows that $C_1$ is isomorphic to the curve
\[
C_2 \colon \tr_{r_3/1}(X) + \gamma_2\tr_{r_2/1}(X) = \tr_{r_2/1}(Y^q - Y).
\]
Now we successively apply the perfection of Lemma \ref{lm:Euclidean_algorithm_performed} to $C_2$ and the initial tuple of integers $(a_1,b_1,c_1,d_1) = (1,\gamma_2, r_3, r_2)$. Consider the operation $(a,b,c,d) \mapsto (b,a+b\gamma,r,c)$ on quadruples of integers (satisfying $0<c<d$) where $0\leq r < d$ and $\gamma>0$ are defined by $d = \gamma c + r$. First of all, if $a,b > 0$, then also $b,a+b\gamma > 0$. Moreover, the operation leaves invariant the sum of products of 1st and 3rd and of 2nd and 4th entries: $ac+bd = br + (a+b\gamma)c$. Thus if $(a_i,b_i,c_i,d_i)$ is the tuple after $(i-1)$th iteration step, we have $a_ic_i + b_id_i = r_3 + \gamma_2r_2 = r_1 = s<p$. Also we have $c_i = r_{i+2}$, $ d_i = r_{i+1}$, and hence $0< c_i < d_i$ as long as $i\leq \alpha-1$. All this implies that $0< a_i,b_i,c_i,d_i<p$ and $0<c_i < d_i$ for each $i=1,2,\dots,\alpha-1$, so that Lemma \ref{lm:Euclidean_algorithm_performed} indeed applies in each step, as long as $i<\alpha$. The last application (for $i=\alpha-1$) produces a quadruple $(a_\alpha,b_\alpha, c_\alpha, d_\alpha) = (b_{\alpha-1},a_{\alpha-1}+b_{\alpha-1}d_{\alpha-1},0,1)$ and $C_2$ is thus isomorphic to the curve
\[
b_\alpha X = Y^q - Y,
\]
and by the same preservation property of the sum $ac+bd$ we have that still $0<b_\alpha<p$ holds. Thus this curve is isomorphic to $\bA^1_{\overline{\FF}_q}$, and we are done.
\end{proof}

The following lemma works for schemes of finite type over $\FF_p$, so we denote (in this lemma only) by $\bA_{\FF_p}$ the usual affine space over $\FF_p$.

\begin{lm}\label{lm:Euclidean_algorithm_performed}
Let $a,b,c,d$ be positive integers with $a,b<p$ and $c<d$. Write $d = \gamma c + r$ with $0 \leq r < c$. Then the curve in $\bA_{\FF_p}^2$ given by the equation 
\[
C_1 \,\colon\, a \tr_{c/1}(x) + b \tr_{d/1}(x) = \tr_{d/1}(y^q - y)
\]
is $\FF_p$-isomorphic to the curve in $\bA_{\FF_p}^2$ given by the equation
\[
C_2 \,\colon\, b \tr_{r/1}(x) + (a+b\gamma) \tr_{c/1}(x) = \tr_{c/1}(y^q - y).
\]
\end{lm}
\begin{proof}
Make the change of variables $x + b^{-1}(y^q - y) \mapsto x$ (by assumption $b<p$, as $0 < c< d$). Thus $C_1$ is isomorphic to the curve
\[
C_1' \,\colon\, a \tr_{c/1}(x) + ab^{-1} \tr_{c/1}(y^q - y) + b \tr_{d/1}(x) = 0.
\]
Via the change of variables $-a^{-1}by \mapsto y$, $C_1'$ gets isomorphic to
\[
C_1'' \,\colon\, a \tr_{c/1}(x) + b \tr_{d/1}(x) = \tr_{c/1}(y^q - y).
\]
We have $\tr_{d/1}(y^q - y) = y^{q^d} - y$. Thus we may successively make the changes of variables of the form $y + x^{q^\alpha}$ (for appropriate $\alpha \in \bZ_{\geq 0}$), to eliminate all powers of $x$ with exponent greater than $q^c$. This does not affect the first summand $a\tr_{c/1}(x)$ and after all these changes $C_1''$ gets isomorphic to the curve
\[
C_1''' \,\colon\, a \tr_{c/1}(x) + b (\gamma \tr_{c/1}(x) + \tr_{r/1}(x))  = \tr_{c/1}(y^q - y),
\]
which is the same as $C_2$.
\end{proof}

\subsection{Proof Theorem \ref{thm:relation_Xh_Xhn}\eqref{eq:relation_Xh_Xhn_b}}\label{sec:proof_of_thm_XhXhn_b}
Again, we work in the setup of Section \ref{sec:Step1}. For $w \in W_\cO$ put
\[
\widehat\Sigma_{(n,n),w} := \{(x,y_1,\tau,z,y_2) \in F\bU_h^1 \times F\bU_h^1 \times \bU_h \times \bT_h \times \mathbb{K}_h^1 \times \bU_h \colon xF(y_1\tau \dot w z) \in y_1\tau \dot w z y_2F(\bU_h^1 y_2^{-1})   \},
\]
and
\[
\widetilde\Sigma_{(n,n),1} := \{(x,x',y_1,\tau,z) \in F\bU_h^1 \times F\bU_h^1 \times \bU_h \times \bT_h \times \bU_h^{-,1} \colon xF(y_1\tau z) = y_1\tau z x'\}
\]
with natural $T_h \times T_h$-actions (like in Section \ref{sec:proof_of_thm_XhXhn_a_multiplicative}). Similar as in the beginning of Section \ref{sec:proof_of_thm_XhXhn_a_multiplicative} it suffices to check that
\begin{align*}
H_c^\ast(\widehat\Sigma_{(n,n),w})_{\theta^{-1},\theta} &= 0 \quad \text{for $1\neq w \in W_\cO$, and} \\ 
\dim H_c^\ast(\widetilde\Sigma_{(n,n),1})_{\theta^{-1},\theta} &= 1.
\end{align*}
First consider the case $w \neq 1$. As $x \in F\bU_h^1$ and $y_1$ varies in $\bU_h$, we can not make the change of variables $xF(y_1) \mapsto x$ as in the proof of Theorem \ref{thm:relation_Xh_Xhn}\eqref{eq:relation_Xh_Xhn_a}. However we can define an action of $H_w^1$ on $\widehat\Sigma_{(n,n),w}$ by
\begin{align*}
(t,t') \colon &(x,y_1,\tau,z,y_2) \mapsto \\ 
&(F(t)xF(y_1)F(t)^{-1}F(F(t)y_1^{-1}F(t)^{-1}), F(t)y_1F(t)^{-1}, t\tau \dot{w}t^{\prime -1} \dot{w}^{-1}, t' z t^{\prime -1}, F(t')y_2F(t')^{-1})
\end{align*}
Note that $F(t)xF(y_1)F(t)^{-1}F(F(t)y_1^{-1}F(t)^{-1}) \in F\bU_h^1$ (on the one side it is contained in $F\bU_h$ as $x,F(y_1) \in F\bU_h$; on the other side it must lie in $\bG_h^1$ as $t,x \in \bG_h^1$). The proof that this indeed is an action goes exactly the same way as in Section \ref{sec:proof_of_thm_XhXhn_a_additive}. The rest of the argument for $\widehat\Sigma_{(1,n),w}$ goes then through exactly as for $\widehat\Sigma_{(1,n),w}$ in Section \ref{sec:proof_of_thm_XhXhn_a_additive}.

Now let $w = 1$. As $x,x',z\in \bG_h^1$, the equation defining $\widetilde\Sigma_{(n,n),1}$ modulo $\bG_h^1$ reduces to $F(y_1\tau) = y_1\tau$. From this it easily follows that $y_1 \in \bG_h^1$. Hence $y_1 \in \bU_h^1$. Hence the change of variables $xF(y_1) \mapsto x$ makes sense (such that the new variable $x$ again lives in $F\bU_h^1$), and the rest of the argument for $\widetilde\Sigma_{(n,n),1}$ goes exactly the same way as for $\widetilde\Sigma_{(1,n),1}$ in Section \ref{sec:proof_of_thm_XhXhn_a_multiplicative}.

\section{Cuspidality}\label{sec:Cuspidality}

We go back to the setup of Section \ref{sec:representations_prelim}. Let $\theta$ be a smooth character of $T = L^\times$ of level $h\geq 1$ in general position. Recall that the induced character of $T_h$ is again denoted by $\theta$, and that it is also in general position.
%(case $h=1$ is already handled in \cite[Section 12.1]{CI_ADLV}). 
% We denote the character of $T_h$ induced by $\theta$ again by $\theta$. 
By Corollary \ref{cor:irreducibility}, $R_{T_h}^{G_h}(\theta)$ is up to sign an irreducible $G_\cO$-representation, hence in particular $R_T^G(\theta)$ is up to sign a genuine representation. We write $|R_{T_h}^{G_h}(\theta)|$ resp.\ $|R_T^G(\theta)|$ for the genuine representation among $\pm R_{T_h}^{G_h}(\theta)$ resp.\ $\pm R_T^G(\theta)$.

\begin{thm}\label{thm:cuspidality_general}
Let $\theta$ be a smooth character of $T = L^\times$ in general position. Then $|R_T^G(\theta)|$ is a finite direct sum of irreducible supercuspidal representations of $G$.
\end{thm}

\begin{proof}
There are many (essentially equivalent) ways to deduce this theorem from Proposition \ref{prop:no_trivial_char}. By \cite[Theorem 1]{Bushnell_90} it suffices to prove that $\Xi_\theta := \Ind_{ZG_{\cO}}^G |R_{T_h}^{G_h}(\theta)|$ is admissible. Let $K \subseteq G$ be a compact open subgroup. We have to show that $(\Xi_\theta)^K$ is finite-dimensional. Conjugating $K$ into $G_\cO$ and making it smaller if necessary, we may assume that $K = \ker(G_\cO \rightarrow G_r)$ for some $r > 0$. Frobenius reciprocity gives
\[
(\Xi_\theta)^K = \bigoplus_{g \in G_\cO Z\backslash G/K} |R_{T_h}^{G_h}(\theta)|^{ZG_\cO \cap gKg^{-1}}.
\]
Thus we have to show that there are only finitely many non-vanishing summands on the right. If ${\bf S}$ denotes a maximal split torus of $\mathbf{G}$ whose apartment in $\mathscr{B}_K(\bfG) = \mathscr{B}_{\breve K}^{F_b}$ contains the vertex stabilized by $G_{\cO}$, then by the rational Iwahori-Bruhat decomposition, $ZG_{\cO}\backslash G /G_\cO \cong X_\ast({\bf S}/\mathbf{Z})_{\rm dom}$. Hence any element of $ZG_\cO\backslash G/K$ has a representative of the form $g = \varpi^{\mu} x$ with $x \in G_\cO$, $\mu \in X_\ast({\bf S})_{\rm dom}$. Now $K$ is normal in $G_\cO$, so $g K g^{-1} = \varpi^\mu K \varpi^{-\mu}$ only depends on $\mu$. Moreover, any coset $ZG_\cO\varpi^\mu G_\cO$ contains only finitely many cosets from $ZG_\cO\backslash G/K$. Thus it suffices to show that for all but finitely many $\mu \in X_\ast(T_0/\mathbf{Z})_{\rm dom}$, $|R_{T_h}^{G_h}(\theta)|^{ZG_\cO \cap \varpi^\mu K \varpi^{-\mu}} = 0$. It is easy to see that for all but finitely many such $\mu$, there is a proper $K$-rational parabolic subgroup $\mathbf{G}$ with unipotent radical ${\bf N}$, such that if $N = {\bf N}(K)$, then $N \cap G_\cO \subseteq \varpi^\mu K \varpi^{-\mu}$. Thus it is enough to show that for each such $N$ we have $|R_{T_h}^{G_h}(\theta)|^{N \cap G_\cO} = 0$. As by Corollary \ref{cor:irreducibility}, $|R_{T_h}^{G_h}(\theta)| = \pm R_{T_h}^{G_h}(\theta)$ is a genuine representation, it suffices to show that $R_{T_h}^{G_h}(\theta)^{N \cap G_\cO} = 0$ (we have the natural map of Grothendieck groups of smooth representations with $\overline{\bQ}_\ell$-coefficients $r \colon K_0(G_\cO) \rightarrow K_0(N \cap G_\cO)$ induced by restriction, and $R_{T_h}^{G_h}(\theta)^{N \cap G_\cO} = 0$ means $\langle {\bf 1}, r(R_{T_h}^{G_h}(\theta)) \rangle = 0$, where ${\bf 1}$ is the trivial representation). This follows from Proposition \ref{prop:no_trivial_char}.
\end{proof}

\begin{prop}\label{prop:no_trivial_char}
Let $N$ be the unipotent radical of a proper $K$-rational parabolic subgroup of $\mathbf{G}$. Then 
\[
R_{T_h}^{G_h}(\theta)^{N \cap G_\cO} = 0.
\] 
\end{prop}
We prove Proposition \ref{prop:no_trivial_char} in Section \ref{sec:proof_no_triv_char_kappa_0} in the case $\kappa = 0$, and in Section \ref{sec:proof_no_triv_char_kappa_arbitrary} in general. The proof in the general case is more technical, but follows exactly the same idea as in the special case $\kappa = 0$. For reasons of clarity we explain the special case first. 

The explicit description in Lemma \ref{lm:description_of_quotient} used in the proof of Proposition \ref{prop:no_trivial_char} is -- to the author's knowledge -- already new for classical Deligne--Lusztig varieties, i.e., when $h=1$ (and $\kappa = 0$). In particular, for the Coxeter-type variety for $\GL_{n,\FF_q}$ it gives an alternative and much more direct proof of the cuspidality result for Coxeter-type varieties \cite[Theorem 8.3]{DeligneL_76}, which is the last statement of the following corollary to Proposition \ref{prop:no_trivial_char}.
\begin{Cor}
Let $n \geq 1$, and let $X$ be a Deligne--Lusztig variety of Coxeter type attached to $\GL_{n,\FF_q}$. Let $\theta$ be an \emph{arbitrary} character of $T_1 \cong \FF_{q^n}^\times$, the corresponding $\GL_n(\FF_q)$-representation $R(\theta)$ realized in the cohomology of $X$, satisfies $R(\theta)^{N(\FF_q)} = 0$, for any unipotent radical $N$ of a proper rational parabolic subgroup of $\GL_n$. In particular, if $\theta$ is in general position, the genuine $\GL_n(\FF_q)$-representation $|R(\theta)|$ is irreducible cuspidal.
\end{Cor}

\begin{rem} The proof of Proposition \ref{prop:no_trivial_char} is based on the key lemmas \ref{lm:description_of_quotient}, \ref{lm:description_of_quotient_kappa_arbitrary}, where the quotient $N_h \backslash X_h$ is determined. If $\overline{X}_h$ denotes the quotient of $X_h$ by the $T_h$-action, then (the cohomology of) $N_h \backslash \overline{X}_h$ can probably be computed in big generality by same methods as in \cite[(2.10)]{Lusztig_76_Inv} (where Coxeter-type Deligne--Lusztig varieties in the flag manifold for a reductive group $\bG$ over $\FF_q$ are studied, in particular $h=1$). Proofs of Lemmas \ref{lm:description_of_quotient}, \ref{lm:description_of_quotient_kappa_arbitrary} suggest that the quotients $N_h \backslash X_h$ are harder to understand than $N_h \backslash \overline{X}_h$. 

For $h=1$ and $\bG$ arbitrary reductive group over $\FF_q$, a quotient similar to $N_h \backslash X_h$ appears in \cite[Section 3.2]{BonnafeR_06}, \cite{Dudas_13} and a couple of related articles. The methods used in \cite{BonnafeR_06} are indirect in the sense that the structure of the \emph{tame} fundamental group of the multiplicative group $\bG_{m, \overline{\FF}_q}$ is used. In our situation these methods would only apply in the case $h=1$, because for $h>1$ the natural covering $X_h \rightarrow \overline{X}_h$ is wildly ramified. 
\end{rem}

\subsection{Proof of Proposition \ref{prop:no_trivial_char} for $\kappa = 0$}\label{sec:proof_no_triv_char_kappa_0}

For $N$ to have a convenient form, we take $b = 1$. We also take $\dot w_1$ to be the element $b_0$ as in \eqref{eq:new_Cox_rep}. Then literally $G = \bfGL_n(K)$, $G_\cO = \bfGL_n(\cO_K)$ and $G_h = \GL_n(\cO_K/(\varpi^h))$. Let $N_h$ denote the image of $N \cap G_\cO$ in $G_h$.  We can assume that $N$ is the unipotent radical of a \emph{maximal} proper parabolic subgroup. Moreover, conjugating $N$ if necessary, we may assume that there is an $1\leq i_0\leq n-1$, such that $N$ consists of matrices $u = (u_{ij})_{1\leq i,j \leq n}$ with $u_{ii} = 1 \forall 1\leq i\leq n$, and $u_{ij} = 0$ unless $i = j$ or ($1\leq i \leq i_0$ and $n-i_0 < j \leq n$). As the actions of $G_h$ and $T_h$ on $X_h$ commute, we have
\[
R_{T_h}^{G_h}(\theta)^{N_h} = H_c^\ast(X_h)_\theta^{N_h} = H_c^\ast(N_h\backslash X_h)_\theta.
\]

We introduce some convenient notation. For $r \geq 1$, and an $r \times r$-matrix $g$, let $|g| := \det g$. For $x = (x_i)_{i=1}^r \in \bW_h(R)^r$, write $g_r(x)$ for the $r \times r$-matrix whose $i$th column is $\sigma^{i-1}(x)$. Also we put 
\[
Y_{r,h} := \{x \in \bW_h^r \colon |g_r(x)| \in \bW_h^\times \}.
\]
This is a functor on ${\rm Perf}_{\FF_q}$, which is represented by an affine perfectly finitely presented perfect $\FF_q$-scheme. The description of $X_h$ in \cite[7.2]{CI_ADLV} says precisely that $X_h \subseteq Y_{n,h}$ is a closed subset defined by the condition $\sigma(|g_n(x)|) = (-1)^{n-1}|g_n(x)|$. 

\begin{lm}\label{lm:quotient_exists_and_smooth}
The quotient $N_h\backslash X_h$ exists as a perfect scheme, and $X_h \rightarrow N_h \backslash X_h$ is finite \'etale. 
% and $N_h\backslash X_h$ is smooth over $\overline{\FF}_q$.
\end{lm}
\begin{proof} $X_h$ is affine and $N_h$ finite, so the quotient exists. As the action has no fixed points the last claim also follows.
\end{proof}

\begin{lm}\label{lm:description_of_quotient}
There is an isomorphism of perfect schemes
\[ 
\alpha \colon N_h\backslash X_h \rightarrow \left\{ (m,x') \in Y_{i_0,h} \times Y_{n-i_0,h} \colon \frac{|g_{i_0}(m)|}{|g_{n - i_0}(x')|^{\sum_{j=1}^{i_0-1} \sigma^j}} \in \bW_h^\times(\FF_q) \right\},
\]
induced by $x = (x_i)_{i=1}^n \mapsto ((m_i(x))_{i=1}^{i_0},(x_i)_{i=i_0+1}^n)$, where $m_i(x)$ is the $(n-i_0 +1) \times (n-i_0 + 1)$-minor of $g_n(x)$ given by
\[
m_i(x) := \begin{vmatrix} x_i & \sigma(x_i) & \dots & \sigma^{n - i_0}(x_i) \\
x_{i_0+1} & \sigma(x_{i_0+1}) & \dots & \sigma^{n - i_0}(x_{i_0+1}) \\
x_{i_0+2} & \sigma(x_{i_0+2}) & \dots & \sigma^{n - i_0}(x_{i_0+2}) \\
\dots &\dots &\dots &\dots \\
x_n & \sigma(x_n) & \dots & \sigma^{n - i_0}(x_n) \\
\end{vmatrix}
\]
\end{lm}
\begin{proof} It is clear that the assignment in the lemma defines an $N_h$-equivariant morphism $X_h \rightarrow (\bW_h)^{i_0} \times (\bW_h)^{n-i_0}$ (with trivial $N_h$-action on the right). Thus it induces a map $N_h\backslash X_h \rightarrow (\bW_h)^{i_0} \times (\bW_h)^{n-i_0}$.
 
A standard argument shows that for $x = (x_i)_{i=1}^n \in X_h(R)$ with corresponding $x' = (x_i)_{i=i_0+1}^n$ and $m = (m_i(x))_{i=1}^{i_0}$, one has that $g_{n-i_0}(x') \in \bW_h^\times(R)$ (see e.g.\ \cite[Lemma 6.13]{CI_ADLV}). This combined with Lemma \ref{lm:equation_determinants} below, shows that we also have $|g_{i_0}(m)| \in \bW_h^\times(R)$. Thus (using Lemma \ref{lm:equation_determinants} again), we see that $\alpha$ is well-defined. 

To prove the lemma, it now suffices to check that $\alpha$ is an isomorphism of \'etale sheaves on $\Perf_{\overline{\FF}_q}$. First we check that as a map of \'etale sheaves, $\alpha$ is surjective. Let $R \in \Perf_{\overline{\FF}_q}$. Let $Z$ denote the target of $\alpha$, and let $m = (m_i)_{i = 1}^{i_0}, x'= (x_i')_{i=i_0+1}^n$ be a element of $Z(R)$. We construct a preimage $x = (x_i)_{i=1}^n \in X_h(R')$ for some \'etale $R$-algebra $R'$. Take $x_i = x_i'$ for $i_0+1\leq i\leq n$. Now, we can find an (finite) \'etale $R$-algebra $R'$, and for each $1 \leq i \leq i_0$, an $x_i = \sum_{j=0}^{h-1} [x_{i,j}]\varpi^j \in \bW_h(R')$ such that 
\begin{equation}\label{eq:expansion_of_mix_first_row}
m_i = m_i(x) = \sum_{k=0}^{n-i_0} (-1)^{k+1} \sigma^k(x_i) \cdot |g_{n-i_0,k}(x')|,
\end{equation}
holds in $\bW_h(R')$, where $g_{n-i_0,k}(x')$ denotes the $(n-i_0) \times (n-i_0)$-matrix whose columns are $x', \sigma(x'), \dots, \widehat{\sigma^k(x')}$, $\dots, \sigma^{n-i_0}(x')$ (here $\widehat{\cdot}$ means that the vector $\cdot$ is omitted). Indeed, note that for $k=n-i_0$ and for $k = 0$, we have 
\begin{equation}\label{eq:some_dets_are_invertible}
|g_{n-i_0,n-i_0}(x')| = |g_{n-i_0,0}(x')| = |g_{n-i_0}(x')| \in \bW_h^\times(\FF_q) 
\end{equation}
Thus, fixing an $i$, and proceeding successively for $j=0,1, \dots, h-1$, we can take \eqref{eq:expansion_of_mix_first_row} modulo $\varpi^{j+1}$ and resolve it for $x_{i,j}$, noting that each time to find a solution we need a (finite) \'etale extension of $R$. Thus $\alpha$ is an epimorphism of \'etale sheaves.

By Lemma \ref{lm:mono_perf_schemes} it remains to show that $\alpha(R) \colon (N_h\backslash X_h)(R) \rightarrow Z(R)$ is injective whenever $R$ is an algebraically closed field. With notation as above, for a fixed $1\leq i\leq i_0$ and $x_{i,0}, x_{i,1}, \dots, x_{i,j-1}$, Equation \eqref{eq:expansion_of_mix_first_row} gives an equation for $x_{i,j}$ of degree precisely $q^{n-i_0}$ (by \eqref{eq:some_dets_are_invertible}), which is separable (by \eqref{eq:some_dets_are_invertible} again). Doing this for each $1\leq i\leq i_0$ and $0 \leq j<h$, we obtain precisely $q^{i_0(n-i_0)h}$ possible values for $x = (x_i)_{i=1}^n \in \bW_h(R)^n$ which map to the given point $(m,x') \in Z(R)$. By Lemma \ref{lm:equation_determinants} all those $x$ automatically lie in $X_h(R)$. This shows that each fiber of the composition of $X_h(R) \rightarrow (N_h \backslash X_h)(R)$ with $\alpha(R)$ has precisely $q^{i_0(n-i_0)h} = \#N_h$ points, i.e., that $\alpha(R)$ is injective. The lemma is proven.
\end{proof}

\begin{lm}\label{lm:equation_determinants}
Let $n \geq 2$, $1\leq i_0 \leq n-1$. For an $\FF_q$-algebra $R$ and $x = (x_i)_{i=1}^n \in Y_{n,h}(R)$, let $m = (m_i(x))_{i=1}^{i_0} \in Y_{i_0,h}(R)$, $x'= (x_i)_{i=i_0+1}^n \in Y_{n-i_0,h}(R)$. Then 
\begin{equation}\label{eq:det_equation}
|g_{i_0}(m)| = |g_n(x)| \cdot |g_{n - i_0}(x')|^{\sum_{j=1}^{i_0-1} \sigma^j}
\end{equation}
\end{lm}
\begin{proof}
For $v = (v_j)_{j = 1}^r \in Y_{r,h}(R)$, and $1\leq i\leq r$, let $v^{(i)} = (v_j)_{j = 1; j \neq i}^n \in Y_{r-1,h}(R)$ denote the vector $v$ with $i$-th coordinate omitted. The claim is tautological for $i_0 = 1$ (in particular, we may assume $n > 2$). We use induction on $i_0$. Expanding along the first column and using the induction hypothesis (for $n-1,i_0-1$), we get
\begin{align*}
|g_{i_0}(m)| = \sum_{i=1}^{i_0} (-1)^{i+1} m_i \sigma\left(|g_{i_0 - 1}(m^{(i)})|\right) 
= \sum_{i=1}^{i_0} (-1)^{i+1} m_i \sigma\left(|g_{n-1}(x^{(i)})| \cdot \prod_{j=1}^{i_0 - 2}\sigma^j \left(|g_{n-i_0}(x')|\right)\right)
\end{align*}
To show that this equals the right hand side of \eqref{eq:det_equation} it suffices to show that
\begin{equation}\label{eq:intermediate_determinantal_equation}
\sum_{i=1}^{i_0}(-1)^{i+1} m_i \sigma\left(|g_{n-1}(x^{(i)})|\right) = |g_n(x)|\cdot \sigma\left(|g_{n-i_0}(x')|\right).
\end{equation}
This follows from a classical minor identity of Turnbull \cite{Turnbull_1909}. We use the more modern source \cite{Leclerc_93}. Let us first recall some notation from \cite{Leclerc_93}. Let $S$ be a ring (commutative, with $1$). For $1\leq i\leq n$, let $a_i, b_i \in S^n$. Then the $2\times n$-tableau  
\[ T = \tikz[baseline=(M.west)]{%
    \node[draw, inner sep=2pt, matrix of math nodes] (M) {%
      a_1 & a_2 & \dots & a_n \\
      b_1 & b_2 & \dots & b_n \\
      };
      \node[fit=(M-2-1) (M-2-2) (M-2-3) (M-2-4),inner sep=-1pt] (R2) {};
\draw(R2.north -| M.west) -- (R2.north -| M.east);
  } \in S
\]
is the product of the determinants of the two $n\times n$-matrices $A$ and $B$, where the $i$-th column of $A$ resp.\ $B$ is $a_i$ resp.\ $b_i$. Similarly one defines an $s\times n$-tableau for each positive integer $s$. The entries of the tableau are the elements $a_i$, $b_i$. More generally we need tableaux with boxes containing some of the entries. Let $T$ be a $s\times n$-tableau, let $A$ be a subset of elements of $T$. For a permutation $\sigma$ of elements of $A$, let $\sigma(T)$ denote the tableau obtained from $T$, where the elements of $A$ were permuted by $\sigma$. Then the tableau $\tau$ = ($T$ with boxes around entries in $A$) is defined as the alternating sum $\sum_\sigma {\rm sgn}(\sigma) \sigma(T)$, where the sum is taken over the cosets of the symmetric group on $A$, modulo the subgroup, which leaves unchanged the rows of $T$. We give an example for $n = 4$, $s=2$:
\[ \tikz[baseline=(M.west)]{%
    \node[draw, inner sep=2.5pt, matrix of math nodes] (M) {%
      a_1 & a_2 & a_3 & a_4 \\
      b_1 & b_2 & b_3 & b_4 \\
      };
      \node[fit=(M-2-1) (M-2-2) (M-2-3) (M-2-4),inner sep=-1pt] (R2) {};
\draw(R2.north -| M.west) -- (R2.north -| M.east);
\node[draw,fit=(M-1-1),inner sep=-1pt] {};
\node[draw,fit=(M-2-3),inner sep=-2pt] {};
\node[draw,fit=(M-2-4),inner sep=-2pt] {};
} \,=\, 
\tikz[baseline=(M.west)]{%
    \node[draw, inner sep=2.5pt, matrix of math nodes] (M) {%
      a_1 & a_2 & a_3 & a_4 \\
      b_1 & b_2 & b_3 & b_4 \\
      };
      \node[fit=(M-2-1) (M-2-2) (M-2-3) (M-2-4),inner sep=-1pt] (R2) {};
\draw(R2.north -| M.west) -- (R2.north -| M.east);
} \,-\, 
\tikz[baseline=(M.west)]{%
    \node[draw, inner sep=2.5pt, matrix of math nodes] (M) {%
      b_3 & a_2 & a_3 & a_4 \\
      b_1 & b_2 & a_1 & b_4 \\
      };
      \node[fit=(M-2-1) (M-2-2) (M-2-3) (M-2-4),inner sep=-1pt] (R2) {};
\draw(R2.north -| M.west) -- (R2.north -| M.east);
} \,-\, 
\tikz[baseline=(M.west)]{%
    \node[draw, inner sep=2.5pt, matrix of math nodes] (M) {%
      b_4 & a_2 & a_3 & a_4 \\
      b_1 & b_2 & b_3 & a_1 \\
      };
      \node[fit=(M-2-1) (M-2-2) (M-2-3) (M-2-4),inner sep=-1pt] (R2) {};
\draw(R2.north -| M.west) -- (R2.north -| M.east);
}.
\]

To continue with our proof, we take $S = \bW_h(R)$. For $1\leq i \leq n$, let ${\bf i} = (0_{i-1}, 1, 0_{n-i}) \in \bW_h(R)^n$ denote the $i$-th coordinate vector. An easy computation shows that
\begin{align*}
|g_n(x)|\cdot \sigma\left(|g_{n-i_0}(x')|\right) - &\sum_{i=1}^{i_0}(-1)^{i+1} m_i \sigma\left(|g_{n-1}(x^{(i)})|\right) 
= \pm \,
\tikz[baseline=(M.west)]{%
    \node[draw, inner sep=2.5pt, matrix of math nodes, ampersand replacement = \&] (M) {%
 {\bf 1} \& {\bf 2} \& \dots \& {\bf i_0} \& \sigma(x) \&  \dots \& \sigma^{n-i_0}(x) \\
      x \& \sigma(x) \&  \& \dots \& \&  \& \sigma^{n-1}(x) \\
 };
      \node[fit=(M-2-1) (M-2-2) (M-2-3) (M-2-4),inner sep=-1pt] (R2) {};
\draw(R2.north -| M.west) -- (R2.north -| M.east);
\node[draw,fit=(M-1-1),inner sep=-1pt] {};
\node[draw,fit=(M-1-2),inner sep=-2pt] {};
\node[draw,fit=(M-1-4),inner sep=-2pt] {};
\node[draw,fit=(M-2-1),inner sep=-1pt] {};
}
\end{align*}
With other words, to show \eqref{eq:intermediate_determinantal_equation} it suffices to show that the tableau on the right side vanishes. Towards this we have 
\[
\tikz[baseline=(M.west)]{%
    \node[draw, inner sep=2.5pt, matrix of math nodes, ampersand replacement = \&] (M) {%
 {\bf 1} \& {\bf 2} \& \dots \& {\bf i_0} \& \sigma(x) \&  \dots \& \sigma^{n-i_0}(x) \\
      x \& \sigma(x) \&  \& \dots \& \&  \& \sigma^{n-1}(x) \\
 };
      \node[fit=(M-2-1) (M-2-2) (M-2-3) (M-2-4),inner sep=-1pt] (R2) {};
\draw(R2.north -| M.west) -- (R2.north -| M.east);
\node[draw,fit=(M-1-1),inner sep=-1pt] {};
\node[draw,fit=(M-1-2),inner sep=-2pt] {};
\node[draw,fit=(M-1-4),inner sep=-2pt] {};
\node[draw,fit=(M-2-1),inner sep=-1pt] {};
} = 
\tikz[baseline=(M.west)]{%
    \node[draw, inner sep=2.5pt, matrix of math nodes, ampersand replacement = \&] (M) {%
 {\bf 1} \& {\bf 2} \& \dots \& {\bf i_0} \& \sigma(x) \&  \dots \& \sigma^{n-i_0}(x) \\
      x \& \sigma(x) \&  \& \dots \& \&  \& \sigma^{n-1}(x) \\
 };
      \node[fit=(M-2-1) (M-2-2) (M-2-3) (M-2-4),inner sep=-1pt] (R2) {};
\draw(R2.north -| M.west) -- (R2.north -| M.east);
\node[draw,fit=(M-1-1),inner sep=-1pt] {};
\node[draw,fit=(M-1-2),inner sep=-2pt] {};
\node[draw,fit=(M-1-4),inner sep=-2pt] {};
\node[draw,fit=(M-1-5),inner sep=-2pt] {};
\node[draw,fit=(M-1-7),inner sep=-2pt] {};
\node[draw,fit=(M-2-1),inner sep=-1pt] {};
} = 0
\]
Here the first equality is immediate from the definition of a tableau with boxes and the fact that the entries $\sigma(x), \dots, \sigma^{n-i_0}(x)$ appear in the second row, and the second equality is an application of Turnbull's identity \cite{Turnbull_1909} (see \cite[Proposition 1.2.2(i)]{Leclerc_93}), which claims that if the number $k$ of boxed entries satisfies $k > n$, then the tableau vanishes. Indeed, viewed as a function on the boxed entries the tableau is a linear \emph{alternating} (not only skew-symmetric as stated in  the proof of \cite[Proposition 1.2.2(i)]{Leclerc_93}) form on $S^n$ in $k$ variables, which must therefore vanish, as $\Lambda^k_S M = 0$ for any finitely generated $S$-module $M$ which can be generated by $n$ elements (in \emph{loc.~cit.} the proof is only formulated when $S$ is a field, but it generalizes to all rings). 
\end{proof}

We continue with the proof of Proposition \ref{prop:no_trivial_char} for $\kappa = 0$. The group $\bG_{m,\overline{\FF}_q}^2$ acts on $Y_{i_0,h} \times Y_{n-i_0,h}$ by 
\begin{equation}\label{eq:action_of_Gm2_on_YtimesY} 
(\tau_1,\tau_2) \colon (y,z) \mapsto (\tau_1y,\tau_2z). 
\end{equation}
(here $\tau_1 y := (\tau_1y_i)_{i=1}^{i_0}$ means entry-wise multiplication, and similarly for $z$). This action restricts to an action of the closed subgroup 
\[
H_0 := \left\{(\tau_1,\tau_2) \in \bG_m^2 \colon \tau_1^{\sum_{j=0}^{i_0-1} \sigma^j} \left(\prod_{i=0}^{n-i_0-1} \sigma^i(\tau_2) \right)^{-\sum_{j=1}^{i_0-1} \sigma^j} = 1 \right\} 
\]
on $\alpha_0(N_h\backslash X_h)$, where $\alpha_0$ is as in Lemma \ref{lm:description_of_quotient}. By Lemma \ref{lm:description_of_quotient} $\alpha_0$ induces an isomorphism on \'etale cohomology. Now $H$ is $1$-dimensional, hence its connected component $H^\circ$ is a $1$-dimensional torus. Therefore the projection of $H^\circ$ to at least one of the $\bG_m$-factors of the ambient group $\bG_m^2$ is non-constant, hence surjective. Hence $\alpha_0(N_h \backslash X_h)^{H^\circ} = \varnothing$. 

The action of $T_h \cong \bW_h^\times(\FF_{q^n})$ on $X_h$ induces an action on $N_h\backslash X_h$, which under $\alpha_0$  is compatible with the $T_h$-action on $\alpha_0(N_h\backslash X_h)$ given by $t \colon (m,x') \mapsto (m \cdot \prod_{j=0}^{i_0-1}\sigma^j(t), x' \cdot t)$ (both products mean scalar multiplication). This action of $T_h$ commutes with the above action of $H_0$ on $\alpha_0(N_h\backslash X_h)$. The explicit description in Lemma \ref{lm:description_of_quotient} also shows that $\alpha_0(N_h\backslash X_h)$ is affine. Thus the $T_h$-equivariant version of the well-known result \cite[10.15 Proposition]{DigneM_91} gives
\[
\dim_{\overline{\bQ}_\ell} H_c^\ast(N_h\backslash X_h)_\theta = \dim_{\overline{\bQ}_\ell} H_c^\ast(\alpha_0(N_h\backslash X_h)^{H^\circ})_\theta = 0.
\]
This finishes the proof of Proposition \ref{prop:no_trivial_char} in the case $\kappa = 0$.

\subsection{Proof of Proposition \ref{prop:no_trivial_char} for arbitrary $\kappa$}\label{sec:proof_no_triv_char_kappa_arbitrary}

Let $\kappa$ be arbitrary. Let $c := \begin{psmallmatrix} 0 & \varpi \\ 1_{n_0 - 1} & 0 \end{psmallmatrix}^{k_0}$, and for $r\geq 1$ let $b_r := \bigoplus_r c$ be the block-diagonal $n_0r\times n_0r$-matrix with blocks equal to $c$. Let $b = b_{n'}$ (it is the special representative corresponding to $\kappa,n$ as in \cite[\S5.2.2]{CI_ADLV}). Let $\dot w = b_0t_{\kappa,n}$ be as in \eqref{eq:new_Cox_rep}. We have then the corresponding groups $\bfG, \bfG_\cO, \bfT, G_h, \dots$ as in Section \ref{sec:basic_notation}. A maximal rational parabolic subgroup of ${\bf G}$    is determined by an integer $1 \leq i_0 \leq n'-1$. Its unipotent radical ${\bf N}$ consists of matrices $(A_{ij})_{1\leq i,j \leq n'}$ where each $A_{ij}$ is a $n_0\times n_0$-matrix, and $A_{ii} = 1_{n_0}$, $A_{ij} = 0$, unless $i=j$ or ($1\leq i \leq i_0$ and $n'-i_0 + 1\leq j \leq n$). Let $l$ denote an integer which modulo $n_0$ is the multiplicative inverse of $k_0$. Moreover, for $a \in \bZ$ define $[a]_{n_0} \in \bZ$ by the requirement that $1 \leq [a]_{n_0} \leq n_0$ and $[a]_{n_0} \equiv a \,\mod n_0$. The subgroup $N_h$ of $G_h$ corresponding to ${\bf N}$ (see Section \ref{sec:Step1}) consists of $n\times n$-matrices of the same shape, where now each of the $n_0 \times n_0$-blocks $A_{ij}$ with $1\leq i \leq i_0$ and $n'-i_0 + 1\leq j \leq n$ is of the form $\sum_{\lambda = 0}^{n_0-1} \varpi^{-\lfloor \frac{\lambda k_0}{n_0} \rfloor} c^\lambda \diag(a_{\lambda}, \sigma^{[l]_{n_0}}(a_{\lambda}), \sigma^{[2l]_{n_0}}(a_{\lambda}), \dots, \sigma^{[(n_0-1)l]_{n_0}}(a_{\lambda}))$ with $a_0 \in \bW_h(\FF_{q^{n_0}})$ and $a_\lambda \in \bW_{h-1}(\FF_{q^{n_0}})$ for $\lambda > 0$. In particular, $\# N_h = q^{n_0 (h + (n_0 - 1)(h-1)) i_0(n'-i_0)}$.

Let $r \geq 1$ and let $Z_{n_0,r,h} = \{(x_i)_{i=1}^{n_0r} \colon \text{ $x_i \in \bW_h$ if $i \equiv 1 \,\mod n_0$ and $x_i \in \bW_{h-1}$ otherwise} \}$. This is a affine, perfectly finitely presented perfect $\FF_q$-scheme. For a perfect $\FF_q$-algebra $R$ and $x \in Z_{n_0,r,h}(R)$ let $g_{n_0,r}(x)$ denote the $n_0r \times n_0r$-matrix whose $i$-th column is $\varpi^{-\lfloor \frac{(i-1)k_0}{n_0}\rfloor }(b_r\sigma)^{i-1}(x)$ (the entries of $g_{n_0,r}(x)$ are either in $\bW_h(R)$ or in $\bW_{h-1}(R)$ or in $\varpi\bW_{h-1}(R) \subseteq \bW_h(R)$). The determinant $|g_{n_0,r}(x)|$ of $g_{n_0,r}(x)$ is a well-defined element of $\bW_h(R)$. Let 
\[
Y_{n_0,r,h} = \{x \in Z_{n_0,r,h} \colon |g_{n_0,r}(x)| \in \bW_h^\times \}
\]
The description of $X_h$ in \cite[7.2]{CI_ADLV} says precisely that $X_h \subseteq Y_{n_0,n',h}$ is the subset defined by the closed condition that $\sigma(|g_{n_0,n'}(x)|) = (-1)^{n'-1} |g_{n_0,n'}(x)|$. 

To simplify notation we write $s:= n_0i_0$ from now on. For $x \in X_h$ and $1\leq i \leq s$, let $m_i(x)$ denote the $(n - s + 1) \times (n-s+ 1)$-minor obtained from $g_{n_0,n'}(x)$ by removing all rows except for the $i$-th and $s + 1,s +2, \dots, n$-th and all but the first $n-s + 1$ columns. Then $m_i(x)$ makes sense as an element of $\bW_h$ resp.\ of $\bW_{h-1}$ if $i \equiv 1 \,\mod n_0$ resp.\ if $i \not\equiv 1 \,\mod n_0$. Thus $(m_i(x))_{i=1}^s \in Z_{n_0,i_0,h}$. The analogue of Lemma \ref{lm:quotient_exists_and_smooth} for $N_h \backslash X_h$ holds with the same proof. We have the following generalization of Lemma \ref{lm:description_of_quotient}.

\begin{lm}\label{lm:description_of_quotient_kappa_arbitrary} The assignment $x = (x_i)_{i = 1}^n \in X_h \mapsto m = (m_i(x))_{i=1}^s, x' = (x_i)_{i=s + 1}^n$ induces an isomorphism of perfect schemes,
\[ 
\alpha_{\kappa} \colon N_h\backslash X_h \rightarrow \left\{ (m,x') \in Y_{n_0,i_0,h} \times Y_{n_0,n'-i_0,h} \colon \frac{|g_{n_0,i_0}(m)|}{|g_{n_0, n' - i_0}(x')|^{\sum_{j=1}^{s-1} \sigma^j}} \in \bW_h^\times(\FF_q) \right\}.
\]
\end{lm}

\begin{proof}
Using the description of $N_h$ given above, one checks that $m_i(x)$ is stable under the $N_h$-action on $X_h$. Now the proof proceeds in a completely analogous fashion to the proof of Lemma \ref{lm:description_of_quotient} (with Lemma \ref{lm:equation_determinants} replaced by its generalization Lemma  \ref{lm:equation_determinants_kappa_arbitrary}).
\end{proof}

\begin{lm}\label{lm:equation_determinants_kappa_arbitrary}
Let $n \geq 2$, $1\leq i_0 \leq n-1$. For a perfect $\FF_q$-algebra $R$ and $x = (x_i)_{i=1}^n \in Y_{n_0,n',h}(R)$, we have $m = (m_i(x))_{i=1}^s \in Y_{n_0,i_0,h}(R)$, $x'= (x_i)_{i=i_0+1}^n \in Y_{n_0,n'-i_0,h}(R)$ and 
\begin{equation}\label{eq:det_equation_kappa_arbitrary}
|g_{n_0,i_0}(m)| = |g_{n_0,n'}(x)| \cdot |g_{n_0,n' - i_0}(x')|^{\sum_{j=1}^{i_0-1} \sigma^j}
\end{equation}
\end{lm}
\begin{proof} It is known that for $x \in Y_{n_0,n',h}(R)$, we have $x' \in Y_{n_0,n'-i_0,h}(R)$ (see \cite[Lemma 6.13]{CI_ADLV}). Thus the similar claim for $m$ follows, once \eqref{eq:det_equation_kappa_arbitrary} is shown. To show \eqref{eq:det_equation_kappa_arbitrary} we first notice that all entries of $g_{n_0,i_0}(m)$ (and not only those in the first column) are in fact $(n-s + 1) \times (n-s + 1)$-minors of $g_{n_0,n'}(x)$. More precisely, for $1\leq i,j\leq s$ the $(i,j)$-th entry of $g_{n_0,i_0}(m)$ is the minor of $g_{n_0,n'}(x)$ obtained by removing all columns except those with numbers $j, j+1, \dots, j+n-s$, and all rows except those with numbers $i, s, s+1, \dots, n$. Let $X_i$ denote the $i$-th row of $g_{n_0,n'}(x)$. Let also ${\bf a}$ denote the $a$-th standard basis vector of a free rank $n$ module (over an arbitrary ring). Using the formalism of tableaux with boxes (as in the proof of Lemma \ref{lm:equation_determinants}), but now for the \emph{rows} of $g_{n_0,n'}(x)$, we can express $|g_{n_0,i_0}(m)|$ as the $s \times n$-tableau with boxes:
\[
\tikz[baseline=(M.west)]{%
    \node[draw, inner sep=2.5pt, matrix of math nodes, ampersand replacement = \&] (M) {%
 X_1 \& X_{s + 1} \& X_{s + 2} \& \dots \&  \& X_n \& {\bf n-s+2} \& {\bf n-s+3} \& \dots \& \& \& {\bf n} \\
{\bf 1} \& X_2 \& X_{s + 1} \& X_{s + 2} \& \dots \& X_{n-1} \& X_n \& {\bf n-s+2} \& {\bf n-s+3} \& \dots \& \& {\bf n} \\
{\bf 1} \& {\bf 2} \& X_3 \& X_{s+1} \& \& \dots \& X_{n-1} \& X_n \& {\bf n-s+2} \& {\bf n-s+3} \& \dots \& {\bf n} \\
\mbox{} \&\&\&\& \&\& \&\& \&\&\&  \\
\dots \&\&\&\& \&\& \dots \&\& \&\&\& \dots \\
\mbox{} \&\&\&\& \&\& \&\& \&\&\& \\
{\bf 1} \& {\bf 2}\& \dots \& \& {\bf s-2} \& X_{s-1} \&  X_{s+1} \& X_{s+2} \& \dots \& X_{n-1} \& X_n \& {\bf n} \\
{\bf 1} \& {\bf 2}\& \& \dots \& \& {\bf s-1} \& X_s \&  X_{s+1} \& \dots \&  \& X_{n-1} \& X_n \\
};
     \node[fit=(M-2-1) (M-2-2) (M-2-3) (M-2-4),inner sep=-1pt] (R2) {};
\draw(R2.north -| M.west) -- (R2.north -| M.east);
     \node[fit=(M-3-1) (M-3-2) (M-3-3) (M-3-4),inner sep=-1pt] (R3) {};
\draw(R3.north -| M.west) -- (R3.north -| M.east);

\draw(R3.south -| M.west) -- (R3.south -| M.east);
    
\node[fit=(M-7-1) (M-7-2),inner sep=-1pt] (R7) {};
\draw(R7.north -| M.west) -- (R7.north -| M.east);

\node[fit=(M-8-1) (M-8-2),inner sep=-1pt] (R8) {};
\draw(R8.north -| M.west) -- (R8.north -| M.east);

\node[draw,fit=(M-1-1),inner sep=-1pt] {};
\node[draw,fit=(M-2-2),inner sep=-2pt] {};
\node[draw,fit=(M-3-3),inner sep=-2pt] {};
\node[draw,fit=(M-7-6),inner sep=-2pt] {};
\node[draw,fit=(M-8-7),inner sep=-2pt] {};
} 
\]
As each of the entries $X_{s+1}, X_{s+2}, \dots, X_n$ appears in each row of this tableau, it is equal to
\[
\tikz[baseline=(M.west)]{%
    \node[draw, inner sep=2.5pt, matrix of math nodes, ampersand replacement = \&] (M) {%
 X_1 \& X_{s + 1} \& X_{s + 2} \& \dots \&  \& X_n \& {\bf n-s+2} \& {\bf n-s+3} \& \dots \& \& \& {\bf n} \\
{\bf 1} \& X_2 \& X_{s + 1} \& X_{s + 2} \& \dots \& X_{n-1} \& X_n \& {\bf n-s+2} \& {\bf n-s+3} \& \dots \& \& {\bf n} \\
{\bf 1} \& {\bf 2} \& X_3 \& X_{s+1} \& \& \dots \& X_{n-1} \& X_n \& {\bf n-s+2} \& {\bf n-s+3} \& \dots \& {\bf n} \\
\mbox{} \&\&\&\& \&\& \&\& \&\&\&  \\
\dots \&\&\&\& \&\& \dots \&\& \&\&\& \dots \\
\mbox{} \&\&\&\& \&\& \&\& \&\&\& \\
{\bf 1} \& {\bf 2}\& \dots \& \& {\bf s-2} \& X_{s-1} \&  X_{s+1} \& X_{s+2} \& \dots \& X_{n-1} \& X_n \& {\bf n} \\
{\bf 1} \& {\bf 2}\& \& \dots \& \& {\bf s-1} \& X_s \&  X_{s+1} \& \dots \&  \& X_{n-1} \& X_n \\
};
     \node[fit=(M-2-1) (M-2-2) (M-2-3) (M-2-4),inner sep=-1pt] (R2) {};
\draw(R2.north -| M.west) -- (R2.north -| M.east);
     \node[fit=(M-3-1) (M-3-2) (M-3-3) (M-3-4),inner sep=-1pt] (R3) {};
\draw(R3.north -| M.west) -- (R3.north -| M.east);

\draw(R3.south -| M.west) -- (R3.south -| M.east);
    
\node[fit=(M-7-1) (M-7-2),inner sep=-1pt] (R7) {};
\draw(R7.north -| M.west) -- (R7.north -| M.east);

\node[fit=(M-8-1) (M-8-2),inner sep=-1pt] (R8) {};
\draw(R8.north -| M.west) -- (R8.north -| M.east);

\node[draw,fit=(M-1-1),inner sep=-1pt] {};
\node[draw,fit=(M-2-2),inner sep=-2pt] {};
\node[draw,fit=(M-3-3),inner sep=-2pt] {};
\node[draw,fit=(M-7-6),inner sep=-2pt] {};
\node[draw,fit=(M-8-7),inner sep=-2pt] {};
\node[draw,fit=(M-8-8),inner sep=-2pt] {};
\node[draw,fit=(M-8-11),inner sep=-2pt] {};
\node[draw,fit=(M-8-12),inner sep=-2pt] {};
} 
\]
Apply (second) Turnbull's identity \cite[Proposition 1.2.2(ii)]{Leclerc_93} to the last row of this tableau, deducing that it is equal to
\[
\tikz[baseline=(M.west)]{%
    \node[draw, inner sep=2.5pt, matrix of math nodes, ampersand replacement = \&] (M) {%
 {\bf 1} \& X_{s + 1} \& X_{s + 2} \& \dots \&  \& X_n \& {\bf n-s+2} \& {\bf n-s+3} \& \dots \& \& \& {\bf n} \\
{\bf 1} \& {\bf 2} \& X_{s + 1} \& X_{s + 2} \& \dots \& X_{n-1} \& X_n \& {\bf n-s+2} \& {\bf n-s+3} \& \dots \& \& {\bf n} \\
{\bf 1} \& {\bf 2} \& {\bf 3} \& X_{s+1} \& \& \dots \& X_{n-1} \& X_n \& {\bf n-s+2} \& {\bf n-s+3} \& \dots \& {\bf n} \\
\mbox{} \&\&\&\& \&\& \&\& \&\&\&  \\
\dots \&\&\&\& \&\& \dots \&\& \&\&\& \dots \\
\mbox{} \&\&\&\& \&\& \&\& \&\&\& \\
{\bf 1} \& {\bf 2}\& \dots \& \& {\bf s-2} \& {\bf s-1} \&  X_{s+1} \& X_{s+2} \& \dots \& X_{n-1} \& X_n \& {\bf n} \\
X_1 \& X_2 \& \& \dots \& \& X_{s-1} \& X_s \&  X_{s+1} \& \dots \&  \& X_{n-1} \& X_n \\
};
     \node[fit=(M-2-1) (M-2-2) (M-2-3) (M-2-4),inner sep=-1pt] (R2) {};
\draw(R2.north -| M.west) -- (R2.north -| M.east);
     \node[fit=(M-3-1) (M-3-2) (M-3-3) (M-3-4),inner sep=-1pt] (R3) {};
\draw(R3.north -| M.west) -- (R3.north -| M.east);

\draw(R3.south -| M.west) -- (R3.south -| M.east);
    
\node[fit=(M-7-1) (M-7-2),inner sep=-1pt] (R7) {};
\draw(R7.north -| M.west) -- (R7.north -| M.east);

\node[fit=(M-8-1) (M-8-2),inner sep=-1pt] (R8) {};
\draw(R8.north -| M.west) -- (R8.north -| M.east);

\node[draw,fit=(M-1-1),inner sep=-1pt] {};
\node[draw,fit=(M-2-2),inner sep=-2pt] {};
\node[draw,fit=(M-3-3),inner sep=-2pt] {};
\node[draw,fit=(M-7-6),inner sep=-2pt] {};
} 
\]
Here all boxes can be removed without changing the value of the tableau, as any non-trivial permutation produces a zero $s \times n$-tableau (as at least one row will contain two equal entries and hence be equal to $0$). The resulting tableau (without boxes) is precisely the right hand side of \eqref{eq:det_equation_kappa_arbitrary}.
\end{proof}

\begin{rem}
In the proof of Lemma \ref{lm:description_of_quotient_kappa_arbitrary}, the fact that the entries of $g_{n_0,i_0}(m)$ are certain minors of $g_{n_0,n'}(x)$ can be shown by a somewhat tedious but straightforward calculation, which we omit here. To illustrate the principle, we give an example. Let $n = 9$, $\kappa = 6$, so that $n' = 3$, $n_0 = 3$, $k_0 = 2$. Let $i_0 = 2$. We have the two minors of $g_{n_0,n'}(x)$,

\[m_2 = \left|\begin{smallmatrix} x_2 & \varpi \sigma(x_3) & \sigma^2(x_1) & \sigma^3(x_2) \\ x_7 & \varpi \sigma(x_8) & \varpi \sigma^2(x_9) & \sigma^3(x_7) \\ x_8 & \varpi \sigma(x_9) & \sigma^2(x_7) & \sigma^3(x_8)\\ x_9 & \sigma(x_7) & \sigma^2(x_8) & \sigma^3(x_9) \end{smallmatrix}\right| \qquad \text{ and } \qquad 
M:= \left|\begin{smallmatrix} \varpi \sigma(x_2) & \varpi \sigma^2(x_3) & \sigma^3(x_1) & \varpi\sigma^4(x_2) \\ \varpi\sigma(x_8) & \varpi \sigma^2(x_9) & \sigma^3(x_7) & \varpi\sigma^4(x_8) \\ \varpi\sigma(x_9) & \sigma^2(x_7) & \sigma^3(x_8) & \varpi\sigma^4(x_9)\\ \sigma(x_7) & \sigma^2(x_8) & \sigma^3(x_9) & \sigma^4(x_7) \end{smallmatrix}\right|
\] 
the first corresponding to rows $2,7,8,9$ and columns $1,2,3,4$, and the second corresponding to rows $1,7,8,9$ and $2,3,4,5$.  The first of these minors is by definition the $(2,1)$-entry of $g_{n_0,i_0}(m)$, and the fact mentioned above claims that the second minor is equal to the $(1,2)$-entry of $g_{n_0,i_0}(m)$, that is, to $\varpi \sigma(m_2) \in \varpi \bW_{h-1} \subseteq \bW_h$. First, $M$ makes sense as an element of $\varpi \bW_{h-1}$. To compute it, we may lift its entries to elements in $\bW$, where we can multiply rows and columns by powers of $\varpi$, to see that
\[
M = \varpi^{-2} \left|\begin{smallmatrix} \varpi \sigma(x_2) & \varpi^2 \sigma^2(x_3) & \varpi\sigma^3(x_1) & \varpi\sigma^4(x_2) \\ \varpi\sigma(x_8) & \varpi^2 \sigma^2(x_9) & \varpi\sigma^3(x_7) & \varpi\sigma^4(x_8) \\ \varpi\sigma(x_9) & \varpi\sigma^2(x_7) & \varpi\sigma^3(x_8) & \varpi\sigma^4(x_9)\\ \sigma(x_7) & \varpi\sigma^2(x_8) & \varpi\sigma^3(x_9) & \sigma^4(x_7) \end{smallmatrix}\right| = \varpi
\left|\begin{smallmatrix} \sigma(x_2) & \varpi \sigma^2(x_3) & \sigma^3(x_1) & \sigma^4(x_2) \\ \sigma(x_8) & \varpi \sigma^2(x_9) & \sigma^3(x_7) & \sigma^4(x_8) \\ \sigma(x_9) & \sigma^2(x_7) & \sigma^3(x_8) & \sigma^4(x_9)\\ \sigma(x_7) & \varpi\sigma^2(x_8) & \varpi\sigma^3(x_9) & \sigma^4(x_7) \end{smallmatrix}\right| = \varpi \sigma(m_2)
\]
(after reducing modulo $\varpi^h\bW$), as claimed. 
\end{rem}

We continue with the proof of Proposition \ref{prop:no_trivial_char}. The group $\bG_{m,\overline{\FF}_q}^2$ acts on $Y_{n_0,i_0,h} \times Y_{n_0,n-i_0,h}$ by the same formula as in \eqref{eq:action_of_Gm2_on_YtimesY}. This action restricts to an action of the closed subgroup 
\[
H_\kappa := \left\{(\tau_1,\tau_2) \in \bG_m^2 \colon \tau_1^{\sum_{j=0}^{s-1} \sigma^j} \left(\prod_{i=0}^{n-s-1} \sigma^i(\tau_2) \right)^{-\sum_{j=1}^{s-1} \sigma^j} = 1 \right\} 
\]
on $N_h\backslash X_h \cong \alpha_\kappa(N_h\backslash X_h)$, where $\alpha_\kappa$ is as Lemma \ref{lm:description_of_quotient_kappa_arbitrary}. Now $H$ is $1$-dimensional, hence its connected component $H^\circ$ is a $1$-dimensional torus. The rest of the argument is exactly as at the end of Section \ref{sec:proof_no_triv_char_kappa_0}. Proposition \ref{prop:no_trivial_char} is now proven.

\section{Review of some representation theory}

We fix an isomorphism $\overline{\QQ}_\ell \cong \mathbb{C}$ and use it to identify the isomorphism classes of smooth complex with smooth $\overline{\bQ}_\ell$-representations of all involved groups. For a finite dimensional (complex or $\overline{\QQ}_\ell$-) representation $\rho$ of a group, we denote by ${\rm deg}(\rho)$ the degree of $\rho$.

\subsection{Square-integrable representations}\label{sec:square_integrables} We recall some well-known results about square-integrable  representations of $p$-adic reductive groups due to Harish-Chandra. For a detailed treatment we refer to \cite{HarishChandra_70} (see also \cite{Cartier_Corvalis_79}). 

In this section let ${\bf G}$ be an arbitrary reductive group over $K$ and $G = {\bf G}(K)$. Let $Z$ be the (K-valued points of) the maximal split torus contained in the center of ${\bf G}$. Let $\psi \colon Z \rightarrow \overline{\QQ}_\ell^\times$ be a unitary character of $Z$. We fix now an invariant Haar measure on $G/Z$ (recall that $G$ is unimodular). We work with complex-valued representations of $G$. Let $\cE_2(G,\psi)$ denote the set of equivalence classes of irreducible unitary representations $(\pi,V)$ of $G$, which have central character $\psi$ and satisfy
\begin{equation}\label{eq:squareIntegrability}
\int_{G/Z} |(u,\pi(g)v)|^2 d\bar{g}< + \infty
\end{equation}
where $(\cdot,\cdot)$ denotes the scalar product in the Hilbert space $V$ (the integral makes sense as $\psi$ is unitary). These are the \emph{square-integrable} representations with central character $\chi$. All irreducible supercuspidal representations with unitary central character are square-integrable \cite[\S3]{HarishChandra_70}. 

For a given $\pi \in \cE_2(G,\psi)$, the integral \eqref{eq:squareIntegrability} is equal to $d(\pi,d\bar{g})|u|^2|v|^2$, where the constant $d(\pi,d\bar{g})>0$ is independent of $u,v$ (and thus only depends on $\pi$ and the chosen measure $d\bar{g}$). The constant $d(\pi,d\bar{g})$ is called the \emph{formal degree} of $\pi$ (with respect to $d\bar{g}$). Let $H$ be a compact open subgroup of $G$. If $d\bar{g},d\bar{g}'$ are two invariant Haar measures on $G$, then $d(\pi,d\bar{g}){\rm vol}(HZ/Z,d\bar{g}) = d(\pi,d\bar{g}'){\rm vol}(HZ/Z,d\bar{g}')$. Moreover, if $\pi \in \cE_2(G,\psi)$ is of the form $\pi = \cind_{ZH}^G \tau$ for an (automatically finite-dimensional) representation $\tau$ on which $Z$ acts by the character $\psi$, then $d(\pi,d\bar{g}){\rm vol}(HZ/Z, d\bar{g}) = \deg \tau$ (cf. \cite[1.6]{Cartier_Corvalis_79}).

For any $\pi \in \cE_2(G,\psi)$ and a smooth irreducible representation $\rho$ of $H$, let $(\pi:\rho)$ denote the multiplicity of $\rho$ in the restriction of $\pi$ to $H$. We need the following estimate due to Harish-Chandra.
\begin{thm}[see {\cite[p.6]{HarishChandra_70}}]\label{thm:key_estimation_general} Given $H$, $\rho$ as above, let $\pi \in \cE_2(G,\psi)$. Then
\begin{equation}\label{eq:HC_ineq}
\sum_{\pi \in \cE_2(G,\psi)} d(\pi,d\bar{g}){\rm vol}(HZ/Z,d\bar{g})(\pi:\rho) \leq \deg\rho.
\end{equation}
\end{thm}

\subsection{Traces on elliptic elements}\label{sec:traces_ell_elements}

For the moment keep the assumptions of Section \ref{sec:square_integrables} (in particular, $\bfG$ is arbitrary reductive). Let $\mathcal{H}(G)$ denote the convolution algebra of locally constant compactly supported functions on $G$. Fix a Haar measure $dg$ on $G$. For any smooth $G$-representation $(\pi,V)$, $\mathcal{H}(G)$ acts in $V$ by $\pi(f)v = \int_G f(g)\pi(g)vdg$ for all $v \in V$, $f \in \mathcal{H}(G)$. If $\pi$ is admissible, then $\pi(f)$ has finite dimensional range, and hence a trace. Let $G^{\rm reg, ss}$ denote the set of regular semi-simple elements of $G$. It is open dense in $G$. The following result due to  Harish-Chandra and Lemaire ensures the existence of a trace of a finite length $G$-representation on regular semisimple elements of $G$. 

\begin{thm}[{see \cite[Theorem 1]{Henniart_ICM_06}}]\label{thm:traces_exist}
Let $\pi$ be a finite length (hence admissible) smooth representation of $G$. Then there is a unique (hence invariant under conjugation) locally constant function ${\rm tr}(\pi, \cdot)$ on $G^{\rm reg,ss}$ of $G$, locally integrable on $G$, such that for all $f \in \mathcal{H}(G)$, one has ${\rm tr}\,\pi(f) = \int_G {\rm tr}(\pi,g) f(g) dg$.
\end{thm}

Now assume again, that $G = \bfG(K)$ for an inner form $\bfG$ of $\bfGL_n$. For $g \in G$, let $P(g)$ denote the reduced characteristic polynomial of $g$. Two elements of $g_1,g_2 \in G^{\rm reg,ss}$ are conjugate in $G$ if and only if $P(g_1) = P(g_2)$. All said above applies to $\bfGL_n(K)$ as a special case. Moreover, for an elements $g\in G^{\rm reg, ss}$ there is a unique up to conjugation element $g' \in \bfGL_n(K)^{\rm reg, ss}$ such that $P(g_1) = P(g_2)$. This has a partial converse. Let $G^{\rm ell} \subseteq G^{\rm reg, ss}$ denote the (open) subset of elliptic elements. For any $g' \in \bfGL_n(K)^{\rm ell}$ there is a unique up to conjugation $g \in G^{\rm ell}$ with the same (reduced) characteristic polynomial. The local Jacquet--Langlands correspondence is then the following result, which in its most general form is due to Deligne--Kazhdan--Vigneras \cite{DeligneKV_84} and Badulescu \cite{Badulescu_02}.

\begin{thm}[{see \cite[Theorem 2]{Henniart_ICM_06}}]
There is a unique bijection $\pi' \leftrightarrow \pi = {\rm JL}(\pi')$ between the sets of $\mathscr{A}^2(G)$ and $\mathscr{A}^2(\bfGL_n(K))$ of smooth irreducible square-integrable representations of $\bfGL_n(K)$ and $G$, such that $\tr(\pi,g) = (-1)^{n-n'}\tr(\pi',g')$ whenever $g \in G^{\rm ell}$, $g' \in \bfGL_n(K)^{\rm ell}$ with $P(g) = P(g')$.
\end{thm}

Now we recall a result from \cite{CI_ADLV}. An (elliptic) element $x \in T \cong L^\times$ is called \emph{very regular}, if $x \in \cO_L^\times$ and the image of $x$ in the residue field $\cO_L/\fp_L \cong \FF_{q^n}$ has trivial stabilizer in $\Gal(L/K)$. This definition does not depend on the choice of the isomorphism $T \cong L^\times$ as in Section \ref{sec:form_of_torus}. Write $\theta^\gamma := \theta \circ \gamma$ for $\gamma \in \Gal(L/K)$, $\theta \colon L^\times \rightarrow \overline{\bQ}_\ell^\times$.

\begin{prop}[{Theorem 11.2 of \cite{CI_ADLV}}]\label{prop:traces_of_very_regulars}
Let $\theta \colon T \rightarrow \overline{\bQ}_\ell^\times$ be smooth and $x \in T$ very regular. Then $\tr(R_T^G(\theta), x) = \pm \sum_{\gamma \in \Gal(L/K)}\theta^\gamma(x)$.
\end{prop}

\subsection{Special cases of local Langlands and Jacquet--Langlands correspondences}\label{sec:LLC_JLC_generalities}
As in the introduction, to a character $\theta \colon L^\times \rightarrow \overline{\bQ}_\ell^\times$ one can attach the $n$-dimensional representation $\sigma_\theta = \Ind_{\mathcal{W}_L}^{\mathcal{W}_K} (\theta\cdot \mu)$  of the Weil group of $K$, where we recall that $\mu$ is the rectifying character of $L^\times$, given by $\mu|_{U_L} = 1$ and $\mu(\varpi) = (-1)^{n-1}$. The representation $\sigma_\theta$ is irreducible if and only if $\theta$ is in general position. In this case, the local Langlands correspondence attaches to $\sigma_\theta$ the irreducible supercuspidal $\GL_n(K)$-representation $\pi_\theta^{\GL_n} := {\rm LL}(\sigma_\theta)$. Moreover, the local Jacquet--Langlands correspondence attaches to $\pi_\theta^{\GL_n}$ the irreducible supercuspidal $G$-representation $\pi_\theta := {\rm JL}(\pi_\theta^{\GL_n})$.

Moreover, $\theta$ is in general position if and only if it is admissible in the sense of \cite{Howe_77}, and the construction of Howe \cite{Howe_77} attaches to it an irreducible supercuspidal $\GL_n(K)$-representation, which is (equivalent to) $\pi_\theta^{\GL_n}$. With other words, with notation as in the introduction, the diagram
\[
\begin{tikzcd}
\mathscr{X}/\Gal_{L/K} \arrow[d, "{\theta \mapsto \sigma_\theta}", labels=left]  \arrow[rd, "{\rm Howe}"] & & \\
\mathscr{G}_K^\varepsilon(n) \arrow[r, "{\rm LL}"] & \mathscr{A}_K^\varepsilon(n,0) \arrow[r, "{\rm JL}"] & \mathscr{A}_K^\varepsilon(n,\kappa) 
\end{tikzcd}
\]
commutes. 

\section{Realization of ${\rm LL}$ and ${\rm JL}$ in the cohomology of \texorpdfstring{$X_{\dot w}^{DL}(b)$}{$X$} in some cases}\label{sec:LLC_JLC}

We now will prove Theorem \hyperlink{thm:a}{A} from the introduction. Let $\theta\colon T \cong L^\times \rightarrow \overline{\QQ}_\ell^\times$ be a smooth character in general position. Let $\pi_\theta = {\rm JL}({\rm LL}(\sigma_\theta)) \in \cA_K(n,\kappa)$ be as in Section \ref{sec:LLC_JLC_generalities}.

\subsection{Degree of $R_T^G(\theta)$ and formal degree of $\pi_\theta$}\label{sec:degree_geometric} First we check that the degree of $R_{T_h}^{G_h}(\theta)$ matches with the formal degree of $\pi_\theta$ (see Section \ref{sec:square_integrables}). Here we use results from \cite{CI_DrinfeldStrat}. Fix a Howe decomposition for $\theta$: there is a unique tower of fields $L = L_t \supsetneq L_{t-1} \supsetneq \dots \supsetneq L_1 \supsetneq L_0 = K$ such that
\begin{equation*}
     \theta = (\chi \circ \N_{L/K})(\phi_1 \circ \N_{L/L_1}) \cdots (\phi_{t-1} \circ \N_{L/L_{t-1}}) (\phi_t)
\end{equation*}
for some \textit{primitive} characters $\chi, \phi_1, \ldots, \phi_t$ of $K^\times, L_1^\times, \ldots, L_t^\times$ respectively. Denote by $h_1, \dots, h_t$ the levels of $\phi_1, \dots, \phi_t$ respectively and put $d_k = [L:L_k]$, in particular, $d_0 = n$, $d_t = 1$.
Also, $\theta|_{U_L^1}$ is in general position if and only if $h_t > 1$.

\begin{lm}\label{lm:geometric_dimension}
Assume $p>n$. Assume $\theta|_{U_L^1}$ is in general position. Then
\begin{equation}\label{eq:dim_formula_geometric_side}
\deg |R_{T_h}^{G_h}(\theta)| = q^{\frac{1}{2}n\left[n(h_1-1) - (h_t - 1) - \sum_{k=1}^{t-1}d_k(h_k - h_{k+1})\right]} \prod_{i=1}^{n'-1} (q^{n_0(n'-i)} - 1).
\end{equation}
\end{lm}

\begin{proof} As $R_T^G(\theta \cdot (\psi \circ \N_{L/K})) \cong R_T^G(\theta) \otimes (\psi\circ\det)$ \cite[Lemma 8.4]{CI_ADLV}, we may assume $\chi = 1$, i.e., $h=h_1$. The assumptions along with Theorem \ref{thm:relation_Xh_Xhn} imply that $R_{T_h}^{G_h}(\theta) \cong H_c^\ast(X_{h,n'})_\theta$. We may assume that $b$ is a Coxeter-type representative (as in \cite[5.2.1]{CI_ADLV}). For $t \in T_h$ put $S_{1,t} = \{x \in X_{h,n'} \colon F^n(x) = xt \}$. As in \cite[Lemma 9.3]{CI_ADLV} we see that $S_{1,t} = \varnothing$, unless $t=1$ (in \emph{loc.~cit.} we worked with the special representative for $b$ and this explains the sign $(-1)^{n'-1}$ appearing there). Further, one has $S_{1,1} = G_h$ \cite{CI_DrinfeldStrat}, and so 
\[ 
\#S_{1,1} = \left(\prod_{i=1}^{h-1}\#G_{i+1}^i\right)\cdot \#G_1  = q^{n^2(h-1)} \cdot \prod_{i=0}^{n'-1} (q^{n_0n'} - q^{n_0 i}) = q^{n^2(h-1) + \frac{1}{2}n(n'-1)} \cdot \prod_{i=0}^{n'-1} (q^{n_0(n'-i)} - 1),
\]
as $G_1 \cong ({\rm Res}_{\FF_{q^{n_0}}/\FF_q} \GL_{n',\FF_{q^{n_0}}})(\FF_q)$ and as $\#G_{i+1}^i = q^{n^2}$ for each $i\geq 1$. Boyarchenko's trace formula \cite[Lemma 2.12]{Boyarchenko_12} and the determination \cite[Theorem 6.1.1]{CI_DrinfeldStrat} of the scalar by which $F^n$ acts in the non-vanishing cohomology group $H_c^{r_\theta}(X_h)_\theta$ gives
\[
\dim |R_{T_h}^{G_h}(\theta)| = \dim |H_c^\ast(X_{h,n'})_\theta| = \frac{(-1)^{r_\theta}}{(-1)^{r_\theta}q^{\frac{nr_\theta}{2}}\#T_h} \cdot \#S_{1,1},
\]
The lemma now follows by an easy calculation, as $\#T_h = (q^n - 1)q^{n(h-1)}$, and as $r_\theta = (n'-n) + h_t + (n-2)h + \sum_{k=1}^{t-1}d_k(t_k - t_{k+1})$ by \cite[Corollary 6.1.2]{CI_DrinfeldStrat}. (Technically speaking, one has to check that the choices (of $U,b,w$) made here and in \cite[\S4]{CI_DrinfeldStrat} are coherent and give rise to isomorphic $X_h$'s. This follows from a calculation with matrices.) \qedhere
\end{proof}

On the other side we use the computation of the formal degree of $\pi_\theta^{\GL_n}$ from \cite{CorwinMS_90}.

\begin{lm}\label{lm:degree_on_algebraic_side}
Assume that $\theta|_{U_L^1}$ is in general position. For any left invariant Haar measure $d\bar{g}$ on $G/Z$, $d(\pi_\theta,d\bar{g}){\rm vol}(ZH/Z,d\bar{g})$ is equal to the right hand side of \eqref{eq:dim_formula_geometric_side}. In particular, we have
\[
d(\pi_\theta,d\bar{g}){\rm vol}(G_\cO Z/Z,d\bar{g}) = \deg |R_{T_h}^{G_h}(\theta)|.
\]
\end{lm}
\begin{proof}
The product on the left hand side in the lemma is independent of $d\bar{g}$, so it is enough to show the lemma for a fixed (left invariant) Haar measure. Let $d\bar{x}$ be the Haar measure on $G/Z$, normalised such that the Steinberg representation ${\rm St}_G$ of $G$ satisfies $d({\rm St}_G,d\bar{x}) = 1$. 
Then by Macdonald's formula \cite[\S3.7]{SilbergerZ_96} (see also \cite[Proposition 5.4]{Kariyama_13}), we have
\begin{equation}\label{eq:Macdonalds_formula}
{\rm vol}(G_\cO Z/Z,d\bar{x}) = \frac{1}{n}\prod_{i=1}^{n'-1}(q^{n_0i} - 1).
\end{equation}
% By \eqref{eq:pi_and_rho_theta} and Lemma \ref{lm:formaldegree_and_dimension}, we have $\deg \rho_\theta = d(\pi_\theta, d\bar{x}){\rm vol}(ZG_\cO,d\bar{x})$. 
The normalized formal degree $d(\pi,d\bar{g})$ is stable under the Jacquet--Langlands correspondence \cite{DeligneKV_84,BushnellHL_10}, so we deduce by using \eqref{eq:Macdonalds_formula},
\[
d(\pi,d\bar{x}){\rm vol}(ZH/Z,d\bar{x}) = d(\pi_\theta^{\GL_n}, d\bar{x}^{\GL_n})\cdot \frac{1}{n}\prod_{i=1}^{n'-1}(q^{n_0i} - 1),
\]
where $d\bar{x}^{\GL_n}$ is the measure $d\bar{x}$ in the special case $n'=n$.
Now the normalized formal degree of $\pi_\theta^{\GL_n}$ is determined in \cite[Theorem 2.2.8]{CorwinMS_90} and coincides with the right hand side of \eqref{eq:dim_formula_geometric_side}.
\end{proof}

\subsection{Comparison}\label{sec:comparison}

We now prove Theorem \hyperlink{thm:A}{A}. Assume $p>n$ and assume that $\theta|_{U_L^1}$ is in general position. Let $Z = K^\times$ be the center of $G$. For a smooth character $\phi$ of $K^\times$ we have $R_T^G(\theta \cdot (\phi \circ \N_{L/K})) \cong R_T^G(\theta) \otimes (\phi \circ \det)$ \cite[Lemma 8.4]{CI_ADLV}. An analogous formula holds for $\pi_\theta$. Hence we may twist both sides of the equality claimed in the theorem by a smooth character $\phi$ of $K^\times$. Thus we are reduced to the case that $\theta|_Z$ is unitary. Fix an invariant Haar measure $d\bar{g}$ on $G/Z$. 

By Theorem \ref{thm:cuspidality_general}, there exists a finite set $I$ and an irreducible supercuspidal representation $\pi_i$ of $G$ for each $i \in I$ such that $|R_T^G(\theta)| \cong \bigoplus_{i=1}^s \pi_i$. It is easy to see (e.g.\ using \cite[Lemma 2.12]{Boyarchenko_12}) that the central character of $R_T^G(\theta)$ is $\theta|_Z$. From this and the fact that all supercuspidal representations are square-integrable it follows that $\pi_i \in \cE_2(G, \theta|_Z)$ for all $i$. As by assumption $(p,n) = 1$, each $\pi_i$ is attached to a pair $(E_i/K,\chi_i)$ with $E_i/K$ is a separable degree $n$ extension and $\chi_i$ is an admissible character of $E_i^\times$ in the sense of \cite{Howe_77} (indeed, Howe's construction also works for inner forms of $\bfGL_n$, so that there is no need to pass to the more general constructions of Yu \cite{Yu_01} and Kaletha \cite{Kaletha_19}). Let $I_{nr} \subseteq I$ denote the subset of those $i\in I$, for which $E_i/K$ is unramified, i.e., $E_i \cong L$. For each $i \in I$, $\pi_i$ has a well-defined trace on regular elliptic elements of $G$, and in particular on the very regular elements of $T \cong L^\times$. If $i \in I \sm I_{nr}$, then $\pi_i \cong \cind_{HE_i}^G\tau_i$, where $H \subseteq G_\cO$ is certain (explicitly determined) compact open subgroup, which is not maximal compact, and $E_i^\times$ is appropriately embedded as a subgroup of $G(K)$ normalizing $H$. In particular, for $i \in I \sm I_{nr}$, no conjugate of a very regular element $x \in T$ lies in $HE_i^\times$ (in fact, $x$ has precisely one fixed point on $\mathscr{B}_K$, which has to be a vertex, so it is contained in no stabilizer of a facet of $\mathscr{B}_K$ of dimension $\geq 1$). By \cite[(A.14) Theorem]{BushnellH_96}, $\tr(\pi_i, x) = 0$ for $i \not\in I_{nr}$, and hence for any very regular element $x \in T \cong L^\times$, we have 
\[
\pm \sum_{\gamma \in \Gal(L/K)} \theta^\gamma(x) = \tr(|R_T^G(\theta)|, x) = \sum_{i \in I_{nr}} \tr(\pi_i, x) = \sum_{i \in I_{nr}} c_i \sum_{\gamma \in \Gal(L/K)} \chi_i^\gamma(x),
\]
where $c_i \in \{\pm 1\}$, the first equality is Proposition \ref{prop:traces_of_very_regulars} and the last follows from \cite[3.1 Th\'eor\`eme]{Henniart_92} (in fact, it shows the claim only for $\bfGL_n$, but this along with trace relations defining the Jacquet--Langlands correspondence give also the other cases). We now use the argument from \cite[2.8]{Henniart_92}: if $x \in U_L$ is very regular and $y \in U_L^1$, then $xy \in U_L$ is again very regular. Thus letting $x$ be a fixed very regular element of $U_L$ and varying $y \in U_L^1$ we obtain an equality of finite linear combinations of smooth characters of the group $U_L^1$. We may find an integer $h'$ such that $\theta$ and all $\chi_i$'s are trivial on $U_L^{h'}$, and replace $U_L^1$ by its finite quotient $U_L^1/U_L^{h'}$. As $\theta|_{U_L^1}$ is in general position, the coefficient of $\theta|_{U_L^1}$ on the left hand side is $\theta(x) \neq 0$. By linear independence of characters of a finite group there is at least one $i_0 \in I_{nr}$ with $\chi_{i_0}|_{U_L^1} = \theta|_{U_L^1}$. 

Frobenius reciprocity for the compact induction shows that 
\begin{equation}\label{eq:multiplicity_positive}
(\pi_i:|R_{T_h}^{G_h}(\theta)|) \geq 1 \quad \text{ for $i \in I$}. 
\end{equation}
with notation as in Section \ref{sec:square_integrables}. Fix a Haar measure $d\bar{g}$ on $G/Z$. By Lemma \ref{lm:degree_on_algebraic_side} we have $d(\pi_{i_0},d\bar{g}){\rm vol}(G_\cO Z/Z,d\bar{g}) = \deg |R_{T_h}^{G_h}(\theta)|$, so that Theorem \ref{thm:key_estimation_general} implies $(\pi_{i_0} \colon |R_{T_h}^{G_h}(\theta)|) = 1$ and $(\pi \colon |R_{T_h}^{G_h}(\theta)|) = 0$ for all $\pi \in \cE_2(G,\theta|_{K^\times})$, $\pi \neq \pi_{i_0}$. Combining this with \eqref{eq:multiplicity_positive} we see that $I = \{i_0\}$.
It remains to show that $\chi_{i_0} = \theta$. Either one can apply \cite[Theorem 11.3]{CI_ADLV} (as we now know that $R_T^G(\theta) \cong \pi_{i_0}$ is irreducible), or alternatively use that we already know $\chi_{i_0} = \theta$ on $K^\times U_L^1$, and then apply the same argument as in \cite[5.3]{Henniart_93}. Theorem \hyperlink{thm:A}{A} is proven.

\section{Deligne--Lusztig sheaves on ${\rm Isoc}_G$ and ${\rm Bun}_G$}\label{sec:DL_sheaves_on_Isoc_and_Bun}

In this last (sketchy) section we let the element $b$ vary. The resulting family of Deligne--Lusztig spaces gives rise to a certain $p$-adic Deligne--Lusztig stack, whose construction we briefly outline here. This allows us put our results in the context of the work of Zhu and Xiao--Zhu \cite{Zhu_20, XiaoZ_17} and the seminal work of Fargues--Scholze \cite{FarguesScholze} on the geometrization of the local Langlands correspondence. After a brief investitation of the $p$-adic Deligne--Lusztig stack in \S\ref{sec:stack_of_isoc}-\ref{sec:special_case_stack}, we use it in \S\ref{sec:DL_sheaf_on_isoc}-\ref{sec:diamond_attached_to_isoc}, to construct a \emph{$p$-adic Deligne--Lusztig sheaf} ${\rm DL}_{w,\theta}$ on the $v$-stack of $\bfG$-isocrystals ${\rm Isoc}_{\bfG}$, as well as the corresponding sheaf ${\rm DL}'_{w,\theta}$ on the $v$-stack  ${\rm Bun}_{\bfG}$ of $\bfG$-bundles on the relative Fargues--Fontaine curve. Finally, we restate our main result, Theorem \hyperlink{thm:A}{A} in terms of ${\rm DL}_{w,\theta}$ and state a conjecture relating ${\rm DL}'_{w,\theta}$ to Fargues' Hecke eigensheaf.

% These latter constructions, carried out in \S\ref{sec:DL_sheaf_on_isoc} and \S\ref{sec:diamond_attached_to_isoc}, are conditional on the  existence of a six functor formalism for solid pro-\'etale sheaves on $v$-stacks on ${\rm Perf}_{\overline{\bF}_q}$, which is not developped yet (this is, however, the aim of the ongoing work of L. Mann and the second author).

\smallskip

We denote by $\bfG$ any unramified reductive group over $K$.

\subsection{Stack of isocrystals}\label{sec:stack_of_isoc}
In \cite{Zhu_20, XiaoZ_17} Zhu and Xiao--Zhu consider the stack (for the fpqc-topology) 
\[{\rm Isoc}_{\bfG} \colon \Perf_{\overline\FF_q}  \ni R \mapsto \left(\text{groupoid of $\bfG$-torsors over $\Spec \left(\bW(R)[\varpi^{-1}]\right) / \varphi^\bZ$}\right).\] 
For example, ${\rm Isoc}_{\bfGL_n}(R)$ is the groupoid of locally free $\bW(R)[1/\varpi]$-modules $\cE$ of rank $n$ equipped with a $\varphi$-equivariant isomorphism $\cE \stackrel{\sim}{\rightarrow} \varphi^\ast\cE$, where $\varphi$ is the automorphism of $\bW(R)[1/\varpi]$ lifting the $q$-power Frobenius of $R$. 

Trivializing the torsor, one obtains a presentation as a quotient stack, ${\rm Isoc}_{\bfG} \cong L\bfG/{\rm Ad}_\sigma L\bfG$, where ${\rm Ad}_\sigma$ denotes Frobenius-twisted conjugation. By \cite[Lm.~5.9 and Thm.~5.1]{Ivanov_DL_indrep}, ${\rm Isoc}_{\bfG}$ is even a stack for the arc-topology from \cite{BhattM_18}, and hence also for the $v$-topology from \cite{BhattS_17}. 
The geometric points of ${\rm Isoc}_{\bfG}$ are given by Kottwitz' set $B(\bfG)$ of $\sigma$-conjugacy classes in $\bfG(\breve K)$. 

\subsection{$p$-adic Deligne--Lusztig stack}

There is a stack over ${\rm Isoc}_{\bfG}$, whose fibers are $p$-adic Deligne--Lusztig spaces, very closely related to the spaces $X_{\dot w_1}^{DL}(b)$ introduced in \S\ref{eq:perfect_DL_space_w_b}. More precisely, we have the ``Borel-level'' $p$-adic Deligne--Lusztig spaces $X_w(b)$, introduced in \cite[Def. 8.3]{Ivanov_DL_indrep}, which are the quotients of $\dot X_{\dot w_1}^{DL}(b)$ by the torus action. By varying the parameter $b \in L\bfG$, we obtain a stack over $L\bfG/{\rm Ad}_\sigma L\bfG \cong {\rm Isoc}_{\bfG}$, whose fibers are precisely the spaces $X_w(b)$, in the follwing way. 

Let $\bfT \subseteq \mathbf{B} \subseteq \bfG$ be a quasi-split torus with Weyl group $W$, contained in a $K$-rational Borel subgroup. We have the Bruhat-decomposition $(\bfG/\mathbf{B})^2 = \coprod_{w \in W} \mathcal{O}(w)$. Attached to $w \in W$ we may consider the $v$-sheaf $\mathfrak{X}_w$ defined by the Cartesian diagram
\[
\xymatrix@R+0.2cm@C+0.2cm{
\mathfrak{X}_w \ar[r] \ar[d] & L\mathcal{O}(w) \times L\bfG \ar@^{(->}[d] \\
L(\bfG/\mathbf{B}) \times L\bfG \ar[r] & L(\bfG/\mathbf{B})^2 \times L\bfG
}
\]
where the lower map is $(g,b) \mapsto (g,b\sigma(g),b)$. One checks that $L\bfG$ acts on $\mathfrak{X}_w$ by $h \colon (g,b) \mapsto (hg, hb\sigma(h)^{-1})$, and that the map $\mathfrak{X}_w \rightarrow L\bfG$, $(g,b) \mapsto b$ is $L\bfG$-equivariant with respect to this action on $\mathfrak{X}_w$ and the $\sigma$-twisted conjugation on $L\bfG$. 
% Now, merely by definition of a quotient stack, 
This means that $\mathfrak{X}_w \rightarrow L\bfG$ descends to a map of $v$-stacks,
\begin{equation}\label{eq:DL_stack_general}
\alpha_w \colon X_w \rightarrow {\rm Isoc}_{\bfG}, 
\end{equation}
where $X_w := [\mathfrak{X}_w/L\bfG]$ is the quotient stack. Moreover, the fiber of \eqref{eq:DL_stack_general} over a geometric point $\Spec \overline{\bF}_q \rightarrow [\Spec \overline{\bF}_q/\bfG_b(K)] \subseteq {\rm Isoc}_{\bfG}$ corresponding to $b \in \bfG(\breve K)$, is $X_w(b)$. 

\subsection{Map to ${\rm Isoc}_{\bfT}$}
Choose a lift $\dot w \in \bfG(\breve K)$ of $w$, contained in a hyperspecial subgroup of $\bfG(\breve K)$. We claim that there is a map
\begin{equation}
\gamma_{\dot w} \colon X_w \rightarrow [L\bfT/{\rm Ad}_{\sigma_w} L\bfT] \,\left(\cong {\rm Isoc}_{\bfT_w}\right) \cong \coprod_{\tau \in X_\ast(\bfT)_{\langle \sigma_w\rangle}} [\Spec \overline\bF_q/\bfT_w(K)],
\end{equation}
where $\sigma_w = {\rm Ad}(w) \circ \sigma$ is the $w$-twisted Frobenius on $\bfT$, and $\bfT_w$ is the corresponding $K$-form of $\bfT$. To define it, it suffices to define an $L\bfG$-equivariant map $\tilde \gamma_{\dot w} \colon \mathfrak{X}_w \rightarrow [L\bfT/{\rm Ad}_{\sigma_w} L\bfT]$, where $L\bfG$ acts trivially on $[L\bfT/{\rm Ad}_{\sigma_w} L\bfT]$. Note that we have a well-defined map $\beta_{\dot w} \colon L(\mathbf{B}w\mathbf{B}) \rightarrow L\bfT$, $u_1 t \dot w u_2 \mapsto t$ ($u_1,u_2$ in the unipotent radical of $\mathbf{B}$). Now, let $\tilde \gamma_{\dot w}(g,b)$ be the image of $\beta_{\dot w}(g^{-1}b\sigma(g))$ under $L\bfT \rightarrow [L\bfT/{\rm Ad}_{\sigma_w} L\bfT]$. One checks that this is well-defined and $L\bfG$-equivariant, i.e., we have defined $\gamma_{\dot w}$. This map can indeed depend on the choice of the lift $\dot w$.\footnote{It appears more natural to change the target of the map $\gamma_{\dot w}$. This would give a stacky version (with $b$ varying) of the map $\alpha_{w,b}$ from \cite[Prop.~4.2]{Ivanov_pDL_Cox}} However, the following holds.

\begin{rem} If $\dot w'$ is another lift of $w$ contained in a hyperspecial subgroup, then it follows from \cite[Lm.~3.8]{Ivanov_pDL_Cox}, that the image of $\dot w^{-1}\dot w'$ in $X_\ast(\bfT)_{\langle \sigma_w\rangle}$ in fact lies in $\ker(H^1(k,\bfT_w) \rightarrow H^1(k,\bfG)) = \ker(X_\ast(\bfT)_{\langle \sigma_w\rangle, {\rm tors}} \rightarrow \pi_1(\bfG)_{\langle \sigma \rangle})$. In particular, if $\bfG$ and $w$ are such that this kernel is zero, the map $\gamma_{\dot w}$ only depends on $w$, not on the lift $\dot w$. Note that this is the case when $\bfG = \bfGL_n$ and $w$ is a Coxeter element. We simply write $\gamma_w$ in this case. 
\end{rem}

Pulling $\gamma_{\dot w} \colon X_w \rightarrow [L\bfT/{\rm Ad_\sigma} L\bfT]$ back along $\Spec\overline\bF_q \rightarrow [\Spec\overline\bF_q /\bfT_w(K)] \subseteq [L\bfT/{\rm Ad}_\sigma L\bfT]$, we obtain the natural $\bfT_w(K)$-torsor on $X_w$.

\subsection{A special case}\label{sec:special_case_stack}
% In the case of the present article, $X_w$ admits a more explicit descrption. 
For $\bfG = \bfGL_n$ and arbitrary $w \in W$, $X_w \rightarrow {\rm Isoc}_{\bfG}$ admits the following more explicit description. Let $\bfT$ be the diagonal torus, $\mathbf{B}$ the upper triangular Borel.
% , and $w = \left(\begin{smallmatrix}0&1\\1_{n-1}&0\end{smallmatrix}\right)$ Coxeter. We describe the fibers of the map $X_w \rightarrow {\rm Isoc}_{\bfG}$ from \eqref{eq:DL_stack_general}. 
For $R \in {\rm Perf}_{\overline{\bF}_q}$ an object of $X_w(R)$ over $(\cE,\sigma_{\cE}) \in {\rm Isoc}_{\bfG}(R)$ is given by a complete flag $\cE^\bullet =(\cE^0 \supset \cE^1 \supset \dots \supset \cE^{n-1})$ in $\cE$ (i.e., a $\mathbf{B}$-torsor on $\Spec\bW(R)[1/\varpi]$), subject to the condition that the relative position of $\cE^\bullet$ and $\sigma_{\cE}(\cE^\bullet)$ is $w$. Here, if ${\rm Bun}_{\mathbf{B}}(R)$ denote the $\mathbf{B}$-torsors on $\Spec \bW(R)[1/\varpi]$, the relative position is a locally constant map ${\rm Bun}_{\mathbf{B}} \times {\rm Bun}_{\mathbf{B}} \rightarrow W$, which is defined by sending $\cE_1^\bullet,\cE_2^\bullet$ to the unique element $w \in W \cong S_n$, such that ${\rm rk}(\cE^i \cap \cE^j) = \#\{1 \leq \ell \leq j \colon w(\ell) \leq i\}$.

% together with a trivialization of the graded object ${\rm gr}(\cE^\bullet) = \bigoplus_i \cE^i/\cE^{i+1}$, subject to the condition that .

\subsection{Deligne--Lusztig sheaf on ${\rm Isoc}_{\bfG}$}\label{sec:DL_sheaf_on_isoc}

Fix a smooth character $\theta \colon \bfT_w(K) \rightarrow \overline\bQ_\ell^\times$. Each connected component of ${\rm Isoc}_{\bfT_w} = [L\bfT/{\rm Ad}_{\sigma_w} L\bfT]$ is of the form $[\Spec \overline\bF_q/\bfT_w(K)]$ and hence admits a pro-\'etale local system attached to $\theta$. Let $\mathcal{L}_\theta$ be the pro-\'etale local system on ${\rm Isoc}_{\bfT_w}$, whose restriction to each connected component is this local system. 
\emph{The rest of \S\ref{sec:DL_sheaf_on_isoc} and \S\ref{sec:diamond_attached_to_isoc} is conditional on an adequate six functor formalism of solid pro-\'etale sheaves on small $v$-stacks over perfect schemes}. Note that similar formalism exists for diamonds, cf. \cite[Chap. VII]{FarguesScholze}. We thus assume the existence of the categories $D_{\rm lis}({\rm Isoc}_{\bfG},\overline\bQ_{\ell}) \subseteq D_{\blacksquare}({\rm Isoc}_{\bfG},\bZ_\ell)$, and consider
% Supposing the adequate categories of solid pro-\'etale sheaves and maps between them are constructed, 
% Then obtain by pullback and pushforward a (solid pro-\'etale) complex
\[ 
{\rm DL}_{w,\theta} := \alpha_{w,\natural}\gamma_w^\ast(\mathcal{L}_{\theta}) \in D_{\rm lis}({\rm Isoc}_{\bfG}, \overline \bQ_\ell).
\] 
where $\alpha_{w,\natural}$ is the \emph{left} adjoint to $\alpha_w^\ast$.
The fiber $x_b^\ast{\rm DL}_{w,\theta}$ of ${\rm DL}_{w,\theta}$ at a geometric point $x_b \colon \Spec \overline\bF_q \rightarrow {\rm Isoc}_{\bfG}$ corresponding to $b \in \bfG(\breve K)$ is the $\theta$-part of the cohomology of $\dot X_{\dot w}(b)$, i.e., a ``$p$-adic Deligne--Lusztig complex'' $R_w(\theta)$ in the derived category of smooth $\bfG_b(K)$-representations. Specializing to the case of the present article, we can reformulate our main result as follows.

\begin{cor}\label{cor:main_result_in_terms_of_DLsheaf} If $\bfG = \bfGL_n$, $w$ Coxeter, $b$ basic, $p>n$ and $\theta$ as in Theorem \hyperlink{thm:A}{A}, $x_b^\ast{\rm DL}_{w,\theta}$ has non-vanishing cohomology in exactly one degree, $r_\theta$ (cf. \cite[Thm.6.1.1]{CI_DrinfeldStrat}). This cohomology equals $(-1)^{r_\theta} {\rm JL}({\rm LL}(\sigma_\theta))$, with notations as in the introduction.
\end{cor}

Concerning the values at other points, we conjecture the following.

% We conjecture that the stalks of ${\rm DL}_{w,\theta}$ at non-basic points of ${\rm Isoc}_{\bfG}$ vanish:

\begin{conj}\label{conj:nonbasic_fibers_of_DL_vanish}
In the situation of Corollary \ref{cor:main_result_in_terms_of_DLsheaf}, if $b$ is non-basic, then $x_b^\ast{\rm DL}_{w,\theta} = 0$.
\end{conj}

\begin{rem}
Zhu and Xiao--Zhu constructed in \cite{Zhu_20, XiaoZ_17} a certain category, which should be thought of as the derived category of $\overline\bQ_\ell$-sheaves on ${\rm Isoc}_{\bfG}$, cf. \cite[4.5]{Zhu_20}. However, their definition is quite technical and it seems that a solid pro-\'etale category $D_{\rm lis}({\rm Isoc}_{\bfG},\overline\bQ_{\ell})$ should be the more natural object. Moreover, it seems to be a common belief that Xiao and Zhu's category coincides with Fargues--Scholze's category $D_{\rm lis}({\rm Bun}_{\bfG},\overline\bQ_\ell)$. It seems reasonable that a possible strategy to prove this fact would be to establish equivalence of both categories with $D_{\rm lis}({\rm Isoc_{\bfG}}, \overline{\bQ}_\ell)$.
\end{rem}

\subsection{Diamond associated with ${\rm Isoc}_{\bfG}$ and relation to ${\rm Bun}_{\bfG}$} \label{sec:diamond_attached_to_isoc}

Let ${\rm Perfd}$ denote the category of affinoid perfectoid spaces over $\overline\bF_q$. Attached to any $v$-stack $X$ on $\Perf_{\overline\bF_q}$, we have a $v$-stack $X^\diamond$ on ${\rm Perfd}$, defined as the stackification of the category fibered in groupoids, sending ${\rm Spa}(R,R^+)$ to $X(R^+)$. On the level of solid pro-\'etale derived categories, we have a pullback functor $c_X \colon D_{\rm lis}(X,\overline\bQ_\ell) \rightarrow D_{\rm lis}(X^\diamond,\overline\bQ_\ell)$

On the other side, we have the small $v$-stack ${\rm Bun}_{\bfG}$ sending $S \in {\rm Perfd}$ to the groupoid of $\bfG$-bundles on the relative Fargues--Fontaine curve $X_S^{\rm FF}$, cf. \cite[Chap. III]{FarguesScholze}. There is a natural morphism of $v$-stacks
\[
f \colon {\rm Isoc}_{\bfG}^{\diamond} \rightarrow {\rm Bun}_{\bfG}. 
\]
Indeed, using the Tannakian formalism one is reduced to construct a map in the case $\bfG = \bfGL_n$. Fix $S = {\rm Spa}(R,R^+) \in {\rm Perfd}$. An object of ${\rm Isoc}_{\bfG}^{\diamond} (S)$ is given by a locally free $\bW(R^+)[1/\varpi]$-module $\mathscr{E}$ of rank $n$ plus a $\varphi$-linear isomorphism $\sigma_\cE \colon \mathscr{E} \stackrel{\sim}{\rightarrow} \mathscr{E}$. Write $X_S^{\rm FF} = Y_S/\varphi^\bZ$, with $Y_S = {\rm Spa}\bW(R^+) \sm V(\varpi[\varpi_R])$, where $\varpi_R$ is any pseudo-uniformizer of $R^+$, cf. \cite[II.1.2]{FarguesScholze}. Then $\mathscr{E} \otimes_{\bW(R^+)[1/\varpi]} \mathcal{O}_{Y_S}$ is a rank $n$ vector bundle on $Y_S$ and the $\varphi$-linear automorphism $\sigma_\cE \otimes \varphi$ descends it to a vector bundle on $X_S$, i.e., to an object of ${\rm Bun}_{\bfG}(S)$. 

The following remark was explained to the second author by I. Gleason.

\begin{rem}
Although both topological spaces, $|{\rm Isoc}_{\bfG}|$ and $|{\rm Bun}_{\bfG}|$, have the same underlying set $B(\bfG)$, they are unequal in general. Indeed, the basic locus is open in ${\rm Bun}_{\bfG}$ and closed in ${\rm Isoc}_{\bfG}$. However, they are related via the surjective maps $|{\rm Isoc}_{\bfG}| \leftarrow |{\rm Isoc}_{\bfG}^{\diamond}| \rightarrow |{\rm Bun}_{\bfG}|$, where $|{\rm Isoc}_{\bfG}^{\diamond}|$ has more points than both other spaces, but they get identified under the two maps in different ways, cf. \cite{Gleason_21}.
\end{rem}

Applying the functors 
\[
D_{\rm lis}({\rm Isoc}_{\bfG}, \overline\bQ_\ell) \stackrel{c^\ast}{\rightarrow} D_{\rm lis}({\rm Isoc}_{\bfG}^{\diamond}, \overline\bQ_\ell) \stackrel{f_\ast}{\rightarrow} D_{\rm lis}({\rm Bun}_{\bfG}, \overline\bQ_\ell).
\]
to ${\rm DL}_{w, \theta}$, we obtain ${\rm DL}'_{w,\theta} := f_\ast c^\ast({\rm DL}_{w,\theta}) \in D_{\rm lis}({\rm Bun}_{\bfG},\overline\bQ_\ell)$.
% We try to establish a (conjectural) relation between our complex ${\rm DL}_{w, \theta}$ and certain sheaves on ${\rm Bun}_{\bfG}$. On the level of derived categories there should be maps 
% \[
% D_{\rm proet}({\rm Isoc}_{\bfG}, \overline\bQ_\ell) \stackrel{a}{\leftarrow} D_{\rm proet}({\rm Isoc}_{\bfG}^{\diamond}, \overline\bQ_\ell) \stackrel{b}{\rightarrow} D_{\rm proet}({\rm Bun}_{\bfG}, \overline\bQ_\ell).
% \]
% Now ${\rm DL}_{w, \theta}$ gives rise to a complex ${\rm DL}'_{w,\theta} := b_\ast a^\ast({\rm DL}_{w,\theta}) \in D_{\rm proet}({\rm Bun}_{\bfG},\overline\bQ_\ell)$. 
Expecting a good behavior of $c^\ast$ and $f_\ast$, it seems very reasonable to extend Corollary \ref{cor:main_result_in_terms_of_DLsheaf} and Conjecture \ref{conj:nonbasic_fibers_of_DL_vanish} to the following conjecture.

\begin{conj}\label{conj:DL_sheaf_on_Bun_G}
For $\bfG = \bfGL_n$, $w$ Coxeter and $\theta$ as in Theorem \hyperlink{thm:A}{A}, the stalks of ${\rm DL}'_{w,\theta}$ are
\[
i_b^\ast {\rm DL}'_{w,\theta} = \begin{cases} {\rm LL}({\rm JL}(\sigma_\theta))& \text{if $b$ basic} \\ 0 & \text{otherwise,} \end{cases}
\]
where $i^b \colon {\rm Bun}_{\bfG}^b \hookrightarrow {\rm Bun}_{\bfG}$ is the locally closed substack corresponding to $b$ (cf. \cite[Thm.~I.2.7]{FarguesScholze}).
\end{conj}

This conjecture uniquely determines ${\rm DL}_{w,\theta}'$. This allows us to (conjecturally) relate it to Fargues' Hecke eigensheaf ${\rm Aut}_{\sigma_{\theta}} \in D_{\rm lis}({\rm Bun}_{\bfG},\overline\bQ_{\ell})$, which in the case of $\bfGL_n$ was constructed by Ansch\"utz--LeBras \cite[Thm. 1.2]{AnschuetzLB_21}. In our setup we do not have Hecke operators, so ${\rm DL}_{w,\theta}'$ a priori does not come equipped with the natural transformations $\eta_{(V_i)_{i \in I}}$, which define the Hecke eigensheaf property (cf. \cite[3.3]{AnschuetzLB_21}). However, just by comparing the stalks we get:

\begin{cor} Suppose Conjecture \ref{conj:DL_sheaf_on_Bun_G} holds true. Then ${\rm DL}_{w,\theta}'$ is the object of $D_{\rm lis}({\rm Bun}_{\bfG},\overline\bQ_{\ell})$ underlying the Hecke eigensheaf ${\rm Aut}_{\sigma_{\theta}}$, constructed in \cite[Thm. 1.2]{AnschuetzLB_21}.
\end{cor}

% The idea is that --by varying the parameter $b$ in (a close variant of the) definition \eqref{eq:perfect_DL_space_w_b}-- obtains an arc-stack, lying over ${\rm Isoc}_G$, whose fibers at the geometric points are quotients of the spaces $X_{\dot w}(b)$ by the torus action. To stay in line with the rest of this article, we only consider the case $G = \GL_n$ here. Then we also have a concrete description of the stack in terms of isocrystals.

% For $b \in G(\breve k)$. Consider now the sheaf $X_w(b)$, which  vaiant of \eqref{eq:perfect_DL_space_w_b}, 
% \[
% R \mapsto \{g \in L\bfGL_n(R)/L\bfU_0(R) \colon g^{-1}b\sigma(g) \in L\bfU_0(R)\dot w_1 L\bfU_0(R) \}.
% \]

% In the spirit of \cite{Ivanov_DL_indrep}, where a attached to an element $w \in W$ we may define a \emph{$p$-adic Deligne--Lusztig stack} $X_w^{\rm DL}$ as follows. Let $\mathfrak{X}_w$ denote the fibre product in the diagram 
% \[
% \xymatrix{
% \mathfrak{X}_w \ar[r] & 
% }
% \]

\bibliography{bib_ADLV_CC}{}
\bibliographystyle{alpha}

\end{document}